\documentclass[11pt,letterpaper]{amsart}
\usepackage{comment,amsmath,amssymb,amsthm,changepage,pifont,
            graphicx,tikz,soul,xcolor,longtable,mathtools,ifthen}
\usepackage[new]{old-arrows}
\usepackage[english]{babel}
\usepackage[foot]{amsaddr}
\usetikzlibrary{patterns}
\usetikzlibrary{decorations.pathreplacing}
\allowdisplaybreaks

\newtheorem{theorem}{Theorem}

\newtheorem{lemma}[theorem]{Lemma}
\newtheorem{corollary}[theorem]{Corollary}
\newtheorem{proposition}[theorem]{Proposition}
\newtheorem{conjecture}[theorem]{Conjecture}
\newtheorem{problem}[theorem]{Problem}
\theoremstyle{remark}
\newtheorem{remark}[theorem]{Remark}
\newtheorem{definition}[theorem]{Definition}

\newcommand{\FF}{\mathbf{F}}
\renewcommand{\AA}{\mathbf{A}}
\newcommand{\ZZ}{\mathbf{Z}}

\newcommand{\QQ}{\mathbf{Q}}
\newcommand{\PP}{\mathbf{P}}

\newcommand{\OO}{\mathcal{O}}
\newcommand{\DD}{\mathcal{D}}
\newcommand{\EE}{\mathcal{E}}
\newcommand{\HH}{\mathcal{H}}

\newcommand{\optionalvalue}{0}
\newcommand{\optional}[2][]{\ifthenelse{\optionalvalue=1}{#2}{#1}}
\DeclareMathOperator{\diag}{diag}
\DeclareMathOperator{\res}{R}
\DeclareMathOperator{\charac}{char}

\DeclareMathOperator{\Pic}{Pic}

\DeclareMathOperator{\id}{id}

\DeclareMathOperator{\Norm}{Norm}
\DeclareMathOperator{\AGL}{AGL}
\DeclareMathOperator{\Tr}{Tr}

\DeclareMathOperator{\Gal}{Gal}
\DeclareMathOperator{\mult}{mult}
\DeclareMathOperator{\Ind}{Ind}
\DeclareMathOperator{\Res}{Res}
\DeclareMathOperator{\Hom}{Hom}
\DeclareMathOperator{\Sym}{Sym}

\DeclareMathOperator{\Span}{span}
\DeclareMathOperator{\vol}{vol}

\usepackage[colorinlistoftodos]{todonotes}
\usepackage[colorlinks=true, allcolors=blue]{hyperref}

\title[Scrollar invariants, syzygies and representations of $S_{\lowercase{d}}$]{Scrollar invariants, syzygies and representations of the 
symmetric group}
\author{Wouter Castryck}
\author{Floris Vermeulen}
\author{Yongqiang Zhao}
\date{}

\begin{document}


\begin{abstract}
We give an explicit minimal graded free resolution, in terms of representations of the symmetric group $S_d$, of a Galois-theoretic configuration of $d$ points in $\PP^{d-2}$ that was studied by Bhargava in the context of ring parametrizations. 
When applied to the geometric generic fiber of a simply branched degree $d$ cover of $\PP^1$ by a relatively canonically embedded curve $C$, our construction gives a new interpretation for the splitting types 
of the syzygy bundles appearing in its relative minimal resolution.
Concretely, our work implies that all these splitting types consist of scrollar invariants of
resolvent covers. This vastly generalizes a prior observation due to Casnati, namely that the first syzygy bundle of a degree $4$ cover splits according to the scrollar invariants of its cubic resolvent. 
Our work also shows that the splitting types of the syzygy bundles, together with the multi-set of scrollar invariants, belong to a much larger class of multi-sets of invariants that can be attached to $C \to \PP^1$: one for each irreducible representation of $S_d$, i.e., one for each partition of $d$. 
\end{abstract}

\maketitle

%% \tableofcontents %% Just for papers exceeding 50 pages.

\section{Introduction} \label{sec:intro}

\subsection{\nopunct} \label{ssec:introdefscrollar} This article, which is an extended version of~\cite{self}, is concerned with the ``scrollar invariants" of a curve $C$ with respect to a separable morphism $\varphi : C \to \PP^1$. Throughout, all curves are assumed to be smooth, projective and geometrically integral, unless otherwise stated.
We recall, e.g.\ from~\cite[\S1.2]{coppenskeemmartens}, that the scrollar invariants of $C$ with respect to $\varphi$ are the positive integers $e_1 \leq e_2 \leq \ldots \leq e_{d-1}$ for which
\begin{equation} \label{pushfwd}
 \varphi_\ast \mathcal{O}_C \cong \mathcal{O}_{\PP^1} \oplus \mathcal{O}_{\PP^1}(-e_1) \oplus  
 \ldots \oplus \mathcal{O}_{\PP^1}(-e_{d-1}),
\end{equation}
where $d$ denotes the degree of $\varphi$.
%which
%implies that
%\begin{equation*}
% \varphi_\ast \mathcal{O}_C(nD) \cong \mathcal{O}_{\PP^1}(n) \oplus \mathcal{O}_{\PP^1}(n-e_1 - 2) \oplus  
% \ldots \oplus \mathcal{O}_{\PP^1}(n-e_{d-1} - 2) 
%\end{equation*}
%for all $n \in \ZZ$, yielding the 
Some  prefer the equivalent characterization of $e_i$ as the minimal $n$ for which
$h^0(C,nD) - h^0(C,(n-1)D) > i$, 
with $D$ any geometric fiber of $\varphi$.
The scrollar invariants sum up to $g + d - 1$, with $g$  the genus of $C$, and they satisfy $e_{d-1} \leq (2g+2d-2)/d$; this upper bound will be referred to as the ``Maroni bound".  
As a side remark, let us point out that the scrollar invariants $e_i$ are the function-field analogues of $\log \lambda_i$, with $\lambda_1 \leq \lambda_2 \leq \ldots \leq \lambda_{d-1}$ the non-trivial successive minima of the Minkowski lattice attached to a degree $d$ number field~\cite[\S7]{hessRR}; thus, studying the scrollar invariants of $C$ with respect to $\varphi$ is closely related to studying the ``geometry" of the corresponding function field extension, in Minkowski's sense.

\subsection{\nopunct} \label{ssec:differentscrollar} We caution the reader for an ambiguity in the existing literature: several references, in fact including~\cite[\S1.2]{coppenskeemmartens}, define the scrollar invariants of $C$ with respect to $\varphi$ as the integers $e_1' \leq e_2' \leq \ldots \leq e_{d-1}'$ for which $\varphi_\ast\OO_C(K_C) \cong \mathcal{O}_{\PP^1}(-2) \oplus \mathcal{O}_{\PP^1}(e_1') \oplus \ldots \oplus \mathcal{O}_{\PP^1}(e_{d-1}')$, with $K_C$ some canonical divisor on $C$. The Riemann--Roch theorem implies that $e_i' = e_i - 2$ for all $i$.\footnote{See also Footnote~\ref{footnote:dt}.} Consequently, when interpreting our results for this alternative definition, the shift by $-2$ must be taken into account. 
%We have opted for the definition from~\ref{ssec:introdefscrollar} because it makes for cleaner statements.%; see also~\ref{ssec:differentschreyer}.
%; see e.g.~\cite[p.\,242]{coppenskeemmartens}.

\subsection{Contributions.} \label{ssec:informalmainresult} 
Consider a degree $d$ cover $\varphi : C \to \PP^1$ over a field $k$ with $\charac k = 0$ or $\charac k > d$, and assume for technical convenience
that $\varphi$ is simply branched, i.e., geometrically, all non-trivial ramification is of type $(2, 1^{d-2})$. This ensures that the Galois closure
\[ \overline{\varphi} : \overline{C} \to C \stackrel{\varphi}{\to} \PP^1  \]
has the full symmetric group $S_d$ as its Galois group over $\PP^1$~\cite[Lem.\,6.10]{fulton}.
By the normal basis theorem we can view $L = k(\overline{C})$ as the regular representation of $S_d$ over $k(t) = k(\PP^1)$, and its decomposition $L = \oplus_{\lambda \vdash d} W_\lambda$  into isotypic components induces a decomposition
\begin{equation} \label{eq:regdecompVB} 
 \overline{\varphi}_\ast \mathcal{O}_{\overline{C}} \cong \bigoplus_{\lambda \vdash d} \mathcal{W}_\lambda
\end{equation}
into vector bundles $\mathcal{W}_\lambda$ of rank $(\dim V_\lambda)^2$, where $V_\lambda$ denotes the irreducible representation (i.e., the Specht module) corresponding to the partition $\lambda$.
%; the notation $R_{ \{\id \}}$ is motivated in~\ref{ssec:introresolvent} below.
%As a first contribution, we show that the multi-set $\{ e_1, e_2, \ldots, e_{d-1} \}$ of scrollar invariants of $C$ with respect to $\varphi$ naturally belongs to a large class of multi-sets of numerical invariants 
%that one can attach to $\varphi$, namely one multi-set per partition $\lambda \vdash d$.
%In more detail: 
% the simple branching assumption implies that the Galois closure $L$ of the 
%degree $d$ field extension
%\begin{equation} \label{eq:FFextension} 
%$k(t) =  k(\PP^1) \subseteq k(C)$ 
%\end{equation}
%corresponding to $\varphi$
%has the full symmetric group $S_d$ as its Galois group over $k(t)$, see~\cite[Lem.\,6.10]{fulton},
%so by the normal basis theorem we can view $L$ as the regular representation of $S_d \cong \Gal(L/k(t))$. Thus it decomposes as
%\begin{equation} \label{eq:regdecomp} 
%  L = \bigoplus_\lambda W_\lambda, \qquad 
% W_\lambda \cong V_\lambda^{\dim V_\lambda}
%\end{equation}
% into irreducible representations $V_\lambda$ (e.g., Specht modules), with $\lambda$ running over all partitions of $d$. 
As will be explained in Section~\ref{sec:scrollar.invs},
we can further decompose $\mathcal{W}_\lambda$ as
 \begin{equation} \label{eq:multiscrollarVB} 
 \begin{array}{ccccccc} \mathcal{O}_{\PP^1}(-e_{\lambda,1}) & \oplus & \mathcal{O}_{\PP^1}(-e_{\lambda,2}) & \oplus & \ldots & \oplus & \mathcal{O}_{\PP^1}(-e_{\lambda,\dim V_\lambda}) \\
 \mathcal{O}_{\PP^1}(-e_{\lambda,1}) & \oplus & \mathcal{O}_{\PP^1}(-e_{\lambda,2}) & \oplus & \ldots & \oplus & \mathcal{O}_{\PP^1}(-e_{\lambda,\dim V_\lambda}) \\ 
                     \vdots & & \vdots & & \ddots & & \vdots \\ 
\mathcal{O}_{\PP^1}(-e_{\lambda,1}) & \oplus & \mathcal{O}_{\PP^1}(-e_{\lambda,2}) & \oplus & \ldots & \oplus & \mathcal{O}_{\PP^1}(-e_{\lambda,\dim V_\lambda}), \end{array}
 \end{equation}
where every column (i.e., every ``vertical slice") contains $\dim V_\lambda$ copies of the same entry.
Our main objects of study are the integers obtained by selecting a ``horizontal slice":

\begin{definition} \label{def_scrollar_lambda}
Under the above notation and assumptions, we call $\{ e_{\lambda,1}, \ldots,$ $e_{\lambda, \dim V_{\lambda}} \}$ the multi-set of ``scrollar invariants of $\lambda$ with respect to $\varphi$". 
\end{definition}

\noindent One sees that the $d! - 1$ scrollar invariants of $\overline{C}$ with respect to $\overline{\varphi}$ are obtained by taking the union, over all non-trivial partitions $\lambda \vdash d$,
of the multi-sets of scrollar invariants of $\lambda$ with respect to $\varphi$, where each multi-set is to be considered with multiplicity $\dim V_\lambda$.

% each of the isotypic components $W_\lambda$  comes naturally equipped with a multi-set of $\dim W_\lambda = (\dim V_\lambda)^2$ ``scrollar invariants"
% \begin{equation} \label{eq:multiscrollar} 
% \begin{array}{cccc} e_{\lambda,1} & e_{\lambda,2} & \ldots & e_{\lambda,\dim V_\lambda} \\
%                     e_{\lambda,1} & e_{\lambda,2} & \ldots & e_{\lambda,\dim V_\lambda} \\ 
%                     \vdots & \vdots & \ddots & \vdots \\ 
%                     e_{\lambda,1} & e_{\lambda, 2}&  \ldots & e_{\lambda, \dim V_\lambda}, \end{array}
% \end{equation}
% where every column (i.e., every ``vertical slice") contains $\dim V_{\lambda}$ copies of the same entry. 
% When taking the union of the blocks~\eqref{eq:multiscrollar} over all non-trivial partitions $\lambda \vdash d$, we find a multi-set of size $d! - 1$ which turns out to consist of the scrollar invariants of the degree $d!$ covering 
% corresponding to our Galois closure $k(t) \subseteq L$.

% We refer to the multi-set obtained by selecting a ``horizontal
%slice" of~\eqref{eq:multiscrollar}, i.e., a single row, as the multi-set of ``scrollar invariants of $\lambda$ with respect to $\varphi$". 

Definition~\ref{def_scrollar_lambda} generalizes the notion of scrollar invariants of $C$ with respect to $\varphi$. Indeed, as a consequence to Proposition~\ref{prop:scrollar.invs.of.hooks} below, we recover $\{ e_1, e_2, \ldots, e_{d-1} \}$ as the multi-set of scrollar invariants of the  partition $(d-1, 1)$ with respect to $\varphi$.
Some basic properties generalize as well: e.g., in~\ref{ssec:volanddual} we will prove a ``volume formula" for the sum of the scrollar invariants with respect to any partition $\lambda \vdash d$, thereby generalizing the identity $e_1 + e_2 + \ldots + e_{d-1} = g + d - 1$. We will also prove a duality statement relating the scrollar invariants with respect to $\lambda$ to those with respect to the dual partition $\lambda^*$ (i.e., the partition obtained by transposing its Young diagram).

\subsection{} \label{ssec:MS} We remark that these generalized scrollar invariants have appeared before, at least implicitly.
Indeed, they describe the splitting types of the underlying vector bundle $\mathcal{E}_\lambda$ of the parabolic bundle attached to $V_\lambda$ under the Mehta--Seshadri correspondence~\cite{MS}, where $V_\lambda$ is viewed as a representation of $\pi_1^{\text{geom}}(\PP^1 \setminus \text{branch locus of $\varphi$})$ via its natural map to $\Gal(C / \PP^1) \cong S_d$. Equivalently, one finds $\mathcal{E}_\lambda$ as the Deligne canonical extension to $\PP^1$ of the local system attached to this representation. We have $\mathcal{W}_\lambda \cong \mathcal{E}_\lambda^{\dim V_{\lambda}}$.
The reader is forwarded to the recent works by Landesman--Litt~\cite[\S2]{landesmanlitt2},~\cite[\S3]{LandesmanLitt} and the references therein for further details.

% USEFUL REFERENCE: https://mathoverflow.net/questions/17786/why-are-local-systems-and-representations-of-the-fundamental-group-equivalent

% ALSO USEFUL: section 1.3.4 of https://math.berkeley.edu/~dcorwin/files/etale.pdf and p137 of Tamas Szamuely "Galois groups and fundamental groups"

\subsection{} 
For certain partitions $\lambda \vdash d$, we managed to relate the corresponding multi-sets of scrollar invariants to known data. 
The easiest cases are the hooks, with Young diagrams
  
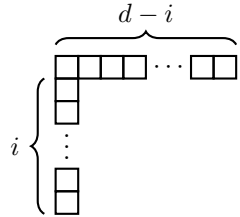
\begin{equation*} %\label{eq:youngdiagramforsum_e}
    \begin{tikzpicture}[scale=0.3,baseline=(current  bounding  box.center)]
      \draw[thick] (0,0) rectangle (1,1);
      \draw[thick] (0,1) rectangle (1,2);
      \node at (0.5,3.33) {\small $\vdots$};
      \draw[thick] (0,4) rectangle (1,5);

      \draw[thick] (0,5) rectangle (1,6);
 %     \draw[thick] (1,5) rectangle (2,6);
      
      \draw[thick] (0,6) rectangle (1,7);
      \draw[thick] (1,6) rectangle (2,7);
      \draw[thick] (2,6) rectangle (3,7);
      \draw[thick] (3,6) rectangle (4,7);      
      \draw[thick] (6,6) rectangle (7,7);
      \draw[thick] (7,6) rectangle (8,7);
      \node at (5.1,6.5) {\small $\cdots$};
      \draw[thick,decorate,decoration={brace,amplitude=5pt}](0,7.5) -- (8,7.5);
      \draw[thick,decorate,decoration={brace,amplitude=5pt}](-0.5,0) -- (-0.5,6);    
      \node at (4,8.85) {\small $d - i$};
      \node at (-1.8,3) {\small $i$};
    \end{tikzpicture} 
    \vspace{0.2cm}
\end{equation*}
for $i = 0, 1, \ldots, d-1$. Concretely, in Section~\ref{sec:scrollar.invs} we will prove: 
\begin{proposition}\label{prop:scrollar.invs.of.hooks}
Consider a simply branched degree $d \geq 2$ cover $\varphi : C \to \PP^1$ over a field $k$ with $\charac k = 0$ or $\charac k > d$ with scrollar invariants $e_1, e_2, \ldots, e_{d-1}$, and let $i \in \{0, 1, \ldots, d - 1\}$. The multi-set of scrollar invariants of the partition $(d-i, 1^i)$ with respect to $\varphi$ is
\[ \left\{ \,
\left. \sum_{\ell\in S} e_\ell  \, \right| \, \text{$S$ is an $i$-element subset of $\{1, 2, \ldots, d - 1\}$} \, \right\}. 
\]
\end{proposition}
\noindent For $i = 1$, corresponding to the standard representation $V_{(d-1,1)}$, we indeed recover the scrollar invariants of $C$ with respect to $\varphi$. For $i = 0$, corresponding to the trivial representation $V_{(d)}$, we find the unique scrollar invariant $0$. For $i = d-1$, corresponding to the sign representation $V_{(1^d)}$, we find the unique scrollar invariant $e_1 + e_2 + \ldots + e_{d-1} = g + d - 1$.

\subsection{} \label{ssec:introschreyer} The main result of this article is a concrete and surprising interpretation for 
 the multi-set of scrollar invariants of the partition $\lambda_{i+1} = (d-i-1, 2, 1^{i-1})$
   \begin{equation} \label{eq:youngdiagramforb}
    \begin{tikzpicture}[scale=0.3,baseline=(current  bounding  box.center)]
      \draw[thick] (0,0) rectangle (1,1);
      \draw[thick] (0,1) rectangle (1,2);
      \node at (0.5,3.33) {\small $\vdots$};
      \draw[thick] (0,4) rectangle (1,5);

      \draw[thick] (0,5) rectangle (1,6);
      \draw[thick] (1,5) rectangle (2,6);
      
      \draw[thick] (0,6) rectangle (1,7);
      \draw[thick] (1,6) rectangle (2,7);
      \draw[thick] (2,6) rectangle (3,7);
      \draw[thick] (3,6) rectangle (4,7);      
      \draw[thick] (6,6) rectangle (7,7);
      \draw[thick] (7,6) rectangle (8,7);
      \node at (5.1,6.5) {\small $\cdots$};
      \draw[thick,decorate,decoration={brace,amplitude=5pt}](0,7.5) -- (8,7.5);
      \draw[thick,decorate,decoration={brace,amplitude=5pt}](-0.5,0) -- (-0.5,5);    
      \node at (4,8.85) {\small $d - i - 1$};
      \node at (-2.65,2.5) {\small $i - 1$};
    \end{tikzpicture} 
    \vspace{0.2cm}
\end{equation}
for any $i = 1, 2, \ldots, d-3$:
 in Section~\ref{sec:schreyer.invs.scrollar} we show that this multi-set equals the splitting type 
\begin{equation} \label{eq:schreyerinvs} 
\left\{ \, b_j^{(i)} \, \left| \, j = 1, \ldots, \beta_i \, \right. \right\}  \qquad \text{with } \beta_i = \frac{d}{i+1}(d-2-i){d-2 \choose i-1}
\end{equation}
of 
the $i$th syzygy bundle in the relative canonical resolution of $C$ with respect to $\varphi$, 
as introduced by Casnati--Ekedahl~\cite{casnati_ekedahl}, who built on work of Schreyer~\cite{schreyer};
see~\ref{ssec:schreyer_inf} below for more details.
We will occasionally refer to the elements of this splitting type as ``Schreyer invariants of $C$ with respect to $\varphi$".
\begin{theorem} \label{thm:schreyerisscrollar}
Consider a simply branched degree $d \geq 4$ cover $\varphi : C \to \PP^1$ over a field $k$ with $\charac k = 0$ or $\charac k > d$,
and let $i \in \{ 1, \ldots, d - 3 \}$. The multi-set of scrollar invariants of the partition
$(d-i-1, 2, 1^{i-1})$ with respect to $\varphi$
is equal to the splitting type
of the $i$th syzygy bundle of the relative canonical resolution of $C$ with respect to $\varphi$.
\end{theorem}
\noindent Symbolically: for $\lambda = \lambda_{i+1}$ the partition from~\eqref{eq:youngdiagramforb} we have
 \[ e_{\lambda, j} = b_j^{(i)} , \qquad j=1, 2, \ldots, \dim V_{(d-i-1,2,1^{i-1})} \]
 for a suitable ordering of the $e_{\lambda,j}$'s.
 As a sanity check, the reader is invited to verify that $\dim V_{(d-i-1,2,1^{i-1})}$ indeed equals $\beta_i$,
 using the hook length formula. In~\ref{ssec:31111} we will slightly relax the simple branching assumption to the condition that all non-trivial ramification is of the form $(2, 1^{d-2})$ or $(3, 1^{d-3})$. 
 %This yields Proposition~\ref{prop:schreyerrelaxed}.
  
As a consequence to Theorem~\ref{thm:schreyerisscrollar} we find that the splitting types of the syzygy bundles turn out to consist of scrollar invariants, namely of the Galois closure $\overline{\varphi} : \overline{C} \to \PP^1$. Thus, they too are ``geometric" in Minkowski's sense.
We know of one prior observation of this kind: for $d = 4$ it was noted that the scrollar invariants of the degree $3$ cover obtained from $\varphi$ through Recillas' trigonal construction are given by 
 $b_1$ and $b_2$, 
with $\{b_1,b_2\}$ the splitting type of the first syzygy bundle of $C$ with respect to $\varphi$. This observation seems due to Casnati~\cite[Def.\,6.3-6.4]{casnati}, although  we refer to Deopurkar--Patel~\cite[Prop.\,4.6]{deopurkar_patel} for a more explicit mention. As will become clear, our main result can be viewed as a vast generalization of this.

Beyond the families $(d-i, 1^{i})$ and $(d-i-1, 2, 1^{i-1})$, we 
did not succeed in finding partitions whose corresponding scrollar invariants 
can be related to known data and we in fact believe that they are all genuinely new.
The first such partitions appear in degree $d=6$, namely $ (2^3)$ and $ (3^2)$,
%   \begin{equation*} 
%   \begin{array}[t]{ccccc} 
%    \begin{tikzpicture}[scale=0.3,baseline=(current  bounding  box.center)]
%      \draw[thick] (0,0) rectangle (1,1);
%      \draw[thick] (0,1) rectangle (1,2);
%      \draw[thick] (1,0) rectangle (2,1);
%      \draw[thick] (1,1) rectangle (2,2);
%      \draw[thick] (2,0) rectangle (3,1);
%      \draw[thick] (2,1) rectangle (3,2);
%    \end{tikzpicture} \vspace{0.5cm}
%    & \hspace{1cm} & \begin{array}{c} \vspace{-0.3cm} \\ \text{and} \\  \end{array} & \hspace{1cm} &
%    \begin{tikzpicture}[scale=0.3,baseline=(current  bounding  box.center)]
%      \draw[thick] (0,0) rectangle (1,1);
%      \draw[thick] (0,1) rectangle (1,2);
%      \draw[thick] (0,2) rectangle (1,3);
%      \draw[thick] (1,0) rectangle (2,1);
%      \draw[thick] (1,1) rectangle (2,2);
%      \draw[thick] (1,2) rectangle (2,3);
%    \end{tikzpicture} \\ \end{array}
%    \vspace{-0.3cm}
%\end{equation*}
corresponding to invariants $a_1, a_2, \ldots, a_5$
and their duals $g+5 - a_1, g+5 - a_2, \ldots, g + 5 - a_5$, which 
seem unrelated to both the scrollar invariants and the Schreyer invariants. 
%of $C$ with respect to $\varphi$.

\subsection{} \label{ssec:introresolvent} 
For any given subgroup $H \subseteq S_d$, 
we can look at the subfield $L^H \subseteq L$ that is fixed by $H$.
 The corresponding degree $[S_d : H] = d! / \lvert H \rvert$ covering
\[ \res_H \varphi : \res_H C \to \PP^1 \]
is called the ``resolvent of $\varphi$ with respect to $H$". 
For $H = \{ \id \}$ we recover the Galois closure $\overline{\varphi} : \overline{C} \to \PP^1$: 
recall that its 
scrollar invariants are obtained by taking the union of the multi-sets of scrollar invariants with respect to $\lambda$, over all non-trivial partitions $\lambda$, where each multi-set is to be considered with multiplicity $\dim V_\lambda$.
For general $H$, this remains true but the multiplicities change: 
\begin{theorem} \label{thm:scrollar.invariants.resolvent}
Consider a simply branched degree $d \geq 2$ cover $\varphi : C \to \PP^1$ over a field $k$ with $\charac k = 0$ or $\charac k > d$, along with a subgroup $H \subseteq S_d$.
The scrollar invariants of $\res_H C$ with respect to $\res_H \varphi$
are found by taking the union, over all non-trivial partitions $\lambda \vdash d$, of the multi-sets
of scrollar invariants of $\lambda$ with respect to $\varphi$, where each multi-set is to be considered with multiplicity 
 \begin{equation*}   
  \mult(V_\lambda, \Ind^{S_d}_H \mathbf{1}).
\end{equation*}
Here $\Ind^{S_d}_H \mathbf{1}$ denotes the representation of $S_d$ that is induced by the trivial representation of $H$ (i.e., it is the permutation representation of $S_d / H$).
\end{theorem}
\noindent In other words, the decomposition of $(\res_H \varphi)_\ast \mathcal{O}_{\res_H C }$ is obtained by taking $\mult(V_\lambda, \Ind^{S_d}_H \mathbf{1})$ 
horizontal slices of the block~\eqref{eq:multiscrollarVB} corresponding to $\lambda$, for each partition $\lambda \vdash d$.
A proof can be found in Section~\ref{sec:scrollar.invs}. Note that some multiplicities may be zero, in which case the corresponding scrollar invariants do not appear. E.g., for $d=4$ 
and $D_4 = \langle (1\,2), (1 \, 3 \, 2 \, 4) \rangle$ the dihedral group of order $8$ one has $\Ind^{S_4}_{D_4} \mathbf{1} \cong V_{(4)} \oplus V_{(2^2)} $,
which in combination with Theorem~\ref{thm:schreyerisscrollar} shows that the scrollar invariants of $\res_{D_4} C$ with respect to $\res_{D_4} \varphi$ are given by $b_1$ and $b_2$. %~\eqref{eq:scrollarrecillas}. 
 This is not a coincidence, as the resolvent with respect to $D_4$ (also known as ``Lagrange's cubic resolvent") is nothing but the degree $3$ covering found through Recillas' trigonal construction~\cite[\S8.6]{vangeemen}. 
 
 As another exemplary corollary to Theorems~\ref{thm:schreyerisscrollar} and~\ref{thm:scrollar.invariants.resolvent} we state:
\begin{theorem} \label{thm:S2Sd-2}
Consider a simply branched degree $d \geq 4$ cover $\varphi : C \to \PP^1$ over a field $k$ with $\charac k = 0$ or $\charac k > d$. Let $H$ be the Young subgroup $S_2 \times S_{d-2}$ of $S_d$. 
Then the multi-set of scrollar invariants of $\res_H C$ with respect to $\res_H \varphi$ is obtained by taking the union of the multi-sets
\begin{itemize}
  \item $\{e_1, e_2, \ldots, e_{d-1} \}$, the scrollar invariants of $C$ with respect to $\varphi$, and 
  \item 
$\{ b_1 , b_2, \ldots, b_{d(d-3)/2} \}$,
the splitting type of the first syzygy bundle of $C$ with respect to $\varphi$.
\end{itemize}
\end{theorem} 
\noindent The (very short) proof can be read in~\ref{ssec:SdSd-2proof}.

\subsection{Syzygies from Galois representations.} \label{ssec:veryintro} 
The main auxiliary tool behind Theorem~\ref{thm:schreyerisscrollar} is an explicit, purely Galois-theoretic way of constructing minimal graded free resolutions of
certain configurations of $d \geq 4$ points in $\PP^{d-2}$, that were introduced by Bhargava in
the context of ring parametrizations~\cite[\S2]{bhargavaquinticrings}. These configurations arise by considering a degree $d$ extension $F \subseteq K$ of fields with $\charac F = 0$ or $\charac F > d$, along with a basis
$\alpha_0 = 1, \alpha_1, \ldots, \alpha_{d-1}$ of $K$ over $F$.
Denote by $L$ the Galois closure; we assume for convenience that $\Gal(L/F)$ is the full symmetric group $S_d$. Write $\sigma_1 = \id, \sigma_2, \ldots, \sigma_d$ for the embeddings $K \hookrightarrow L$ that fix $F$ element-wise,
take the dual basis 
$\alpha_0^\ast, \alpha_1^\ast, \ldots, \alpha_{d-1}^\ast$ with respect to $\Tr_{L/F}$, and define
\[ \alpha_i^{\ast (j)} = \sigma_j(\alpha_i^\ast) \]
for each $i = 0, 1, \ldots, d-1$ and $j = 1, \ldots, d$.
Then the requested points in $\PP^{d-2}$ are
\begin{equation} \label{eq:asspoints} 
[\alpha_1^{* (1)} : \ldots : \alpha_{d-1}^{* (1)}], \ [\alpha_1^{* (2)} : \ldots : \alpha_{d-1}^{* (2)}], \  \ldots, \ [\alpha_1^{* (d)} : \ldots : \alpha_{d-1}^{* (d)}].
\end{equation}
No $d-1$ of these points lie on a hyperplane, so they are ``in general position". 

Thus, from~\cite[(4.2)]{schreyer} we know that any minimal graded free resolution of their joint coordinate ring must have
\begin{equation*} %\label{eq:bettitable}
    \begin{array}{c|cccccc}
        & 0 & 1 & 2 & \ldots & d-3 & d-4 \\
    \hline 
    0 & 1 & 0 & 0 & \ldots & 0 & 0 \\
    1 & 0 & \beta_1 & \beta_2 & \ldots & \beta_{d-3} & 0 \\
    2 & 0 & 0    & 0       & \ldots & 0 & 1 
    \end{array} 
\end{equation*}
as its Betti table.
In our minimal graded free resolution, the details of which can be found in Section~\ref{sec:minfreerep}, the $i$th syzygy module arises from the isotypic subrepresentation  
$W_{\lambda_{i+1}} \subseteq L$ corresponding to the partition $\lambda_{i+1} = (d-i-1,2,1^{i-1})$, where as before we view $L$ as the regular representation of $S_d$ through the normal basis theorem.

The connection between the partitions $\lambda_{i+1}$ and syzygies of $d$ general points in $\PP^{d-2}$ is not a new observation: this was studied by Wilson~\cite[\S5]{wilsonphd}. 
In his discussion, the syzygy modules are constructed from
the Specht modules $V_{\lambda_{i+1}}$; we tend to think of these as ``vertical slices" of the corresponding isotypic components $W_{\lambda_{i+1}}$. 
Our new Galois-theoretic construction is somehow orthogonal to this 
and uses ``horizontal slices", which are not representations. 
%In particular, our resolution is not equivariant; 
Nonetheless, as we will see, they better suit our needs.

\subsection{} \label{ssec:intro_relative_embedding}
In Section~\ref{sec:schreyer.invs.scrollar} we will explain how Bhargava's point configuration shows up very naturally when studying the
geometric generic fiber of our cover $\varphi : C \to \PP^1$, henceforth assumed to be of degree $d \geq 4$.
In more detail, from Casnati--Ekedahl~\cite[Thm.\,2.1]{casnati_ekedahl} we know that $\varphi$ decomposes as
\[ C \stackrel{\iota}{\longhookrightarrow} \PP(\mathcal{E}) \stackrel{\pi}{\longrightarrow} \PP^1, \qquad 
\mathcal{E} = \OO_{\PP^1}(e_1) \oplus \OO_{\PP^1}(e_2) \oplus \ldots \oplus \OO_{\PP^1}(e_{d-1}),
 \]
with $\pi$ the natural $\PP^{d-2}$-bundle map and $\iota$ the ``relative canonical embedding".\footnote{
The standard (i.e., absolute) canonical map $C \to \PP^{g-1}$ is obtained from $\iota$ by composing it with the ``tautological map" 
\[ \kappa : \PP(\mathcal{E}) \stackrel{\cong}{\rightarrow} \PP(\mathcal{E}(-2)) \to \PP^{e_1 + \ldots + e_{d-1} - d} = \PP^{g-1}, \] 
the image of which is a rational normal scroll; see~\cite[\S1]{eisenbudharris}. If $e_1 > 2$ then $\kappa$ is an embedding, in which case 
one can reverse the construction and recover $\PP(\mathcal{E})$ from the canonical model of $C$ as the union
of the linear spans inside $\PP^{g-1}$ of the fibers of $\varphi$. Each such linear span is indeed a $\PP^{d-2}$,
by the geometric Riemann--Roch theorem. This is Schreyer's original approach from~\cite{schreyer}. \label{footnote:tautological}
}  
By identifying $C$ with its relative canonical image, we can view each geometric fiber of $\varphi$, including the geometric generic fiber, as a configuration of $d$ points in $\PP^{d-2}$; an illustration of the case $d = 4$ can be found in Figure~\ref{fig:constructionofscroll}.
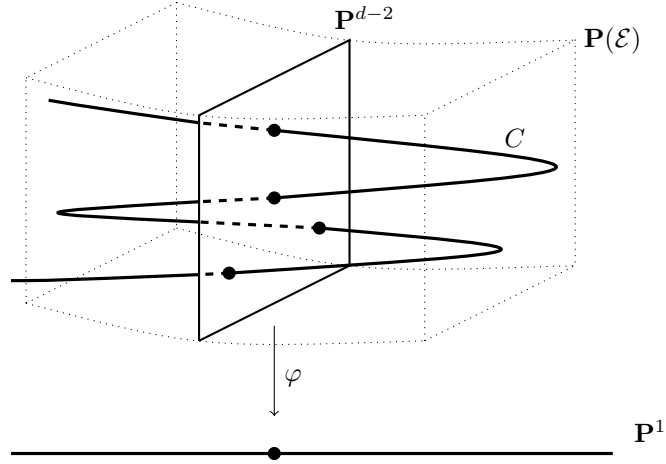
\begin{figure}[ht] 
\begin{center}
 \begin{tikzpicture}
   \draw [thick] (0,0) -- (2,1) -- (2,4) -- (0,3) -- (0,0);
   \node at (2.2,4.3) {\small $\PP^{d-2}$};
   \draw [fill=black,black] (1,2.8) circle (0.08);
   \draw [fill=black,black] (1,1.9) circle (0.08);
   \draw [fill=black,black] (1.6,1.5) circle (0.08);   
   \draw [fill=black,black] (0.4,0.9) circle (0.08);   
   \draw [->] (1,0.2) -- (1,-1);
   \draw [very thick] (-2.5,-1.5) -- (5.5,-1.5);
   \node at (6,-1.2) {\small $\PP^1$};
   \node at (1.25,-0.5) {\small $\varphi$};
   \draw [fill=black,black] (1,-1.5) circle (0.08);
   \draw [very thick] (1, 2.8) .. controls (6,2.3) .. (1, 1.9);
   \draw [very thick] (1.6, 1.5) .. controls (5,1.2) .. (0.4, 0.9);      
   \draw [very thick, dashed] (1, 2.8) -- (0, 2.9);
   \draw [very thick] (0,2.9) .. controls (-1.4,3.1) .. (-2, 3.2);
   \draw [very thick, dashed] (1, 1.9) -- (0, 1.85);
   \draw [very thick, dashed] (1.6, 1.5) -- (0, 1.58);
   \draw [very thick] (0, 1.85) .. controls (-2.5,1.7) .. (0, 1.58);
   \draw [very thick, dashed] (0.4, 0.9) -- (0, 0.88);
   \draw [very thick] (0, 0.88) .. controls (-2, 0.8) .. (-2.5, 0.8);
   \draw [thin, dotted] (3,0) -- (5,1) -- (5,4) -- (3,3) -- (3,0);
   \draw [thin, dotted] (-2.3,0.5) -- (-0.3,1.5) -- (-0.3,4.5) -- (-2.3,3.5) -- (-2.3,0.5);
   \draw [thin, dotted] (-2.3,0.5) .. controls (0,-0.1) .. (3,0);
   \draw [thin, dotted] (-0.3,1.5) .. controls (2,0.9) .. (5,1);
   \draw [thin, dotted] (-0.3,4.5) .. controls (2,3.9) .. (5,4);
   \draw [thin, dotted] (-2.3,3.5) .. controls (0,2.9) .. (3,3);
   \node at (5.5,4) {\small $\PP(\mathcal{E})$};
   \node at (4.2,2.7) {\small $C$};
 \end{tikzpicture}
\end{center}
\caption{\small Configuration of $d$ points in $\PP^{d-2}$ attached to $\varphi$.}
\label{fig:constructionofscroll}
\end{figure}
The field of definition of any point in the support of the geometric generic fiber can be identified with $k(C)$. It is then always possible
to equip that point with projective coordinates $[\alpha_1^\ast : \alpha_2^\ast : \ldots : \alpha_{d-1}^\ast]$ satisfying $\Tr_{k(C)/k(t)}(\alpha_i^\ast) = 0$ for all $i = 1, \ldots, d-1$. By dualizing, we obtain a basis $1, \alpha_1, \ldots, \alpha_{d-1}$ of $k(C)$ over $k(t)$ from which the geometric generic fiber of $\varphi$ is recovered through Bhargava's construction.
(We will actually need a slightly stronger fact, namely that $1, \alpha_1, \ldots, \alpha_{d-1}$ can be arranged to be a so-called ``reduced basis" of $k(C)$ over $k(t)$; see~\ref{ssec:bhargava_is_relativecanonical} for further details.)

\subsection{} \label{ssec:schreyer_inf} 
Casnati--Ekedahl, building on 
Schreyer, further showed that a minimal graded free resolution of the geometric generic fiber can be completed to a minimal resolution of $C$ relative to $\PP(\mathcal{E})$:
\begin{multline}  
 0 \rightarrow \mathcal{O}_{\PP(\mathcal{E})} (-dH + (g-d-1)R) \rightarrow 
 \\
\bigoplus_{j=1}^{\beta_{d-3}} \mathcal{O}_{\PP(\mathcal{E})}(-(d-2)H + b_j^{(d-3)}R) \rightarrow \bigoplus_{j=1}^{\beta_{d-4}} \mathcal{O}_{\PP(\mathcal{E})}(-(d-3)H + b_j^{(d-4)}R)  \rightarrow   \label{Schresolution} \\
\cdots \rightarrow \bigoplus_{j=1}^{\beta_1} \mathcal{O}_{\PP(\mathcal{E})}(-2H + b_j^{(1)}R) \rightarrow \mathcal{O}_{\PP(\mathcal{E})} \rightarrow \mathcal{O}_C \rightarrow 0.
\end{multline}
Here $R = [\pi^\ast \mathcal{O}_{\PP^1}(1)]$ and $H = [\mathcal{O}_{\PP(\mathcal{E})}(1)]$ denote the $\PP^{d-2}$-ruling and the class of ``hyperplane sections",\footnote{More precisely: $\mathcal{O}_{\PP(\mathcal{E})}(1) = j^\ast \mathcal{O}_{\PP^{g + 2d - 3}}(1)$ with $j : \PP(\mathcal{E})\to \PP^{e_1 + \ldots + e_{d-1} + d-2}$ the ``tautological map" from~\cite[\S1]{eisenbudharris}, but now associated with $\mathcal{E}$ rather than with $\mathcal{E}(-2)$ as was the case in Footnote~\ref{footnote:tautological} (this difference is the source of the ambiguity mentioned in~\ref{ssec:differentschreyer}).} respectively, known to form a $\ZZ$-basis of $\Pic(\PP(\mathcal{E}))$; see~\cite{eisenbudharris,schreyer}. 
 This introduces the Schreyer invariants from~\eqref{eq:schreyerinvs}. 
%We refer to the elements of a splitting type as ``Schreyer invariants" of $C$ with respect to $\varphi$. 
At this point, it should come as no surprise to the reader that our strategy to prove Theorem~\ref{thm:schreyerisscrollar} will be to apply this construction to the minimal graded free resolution discussed in~\ref{ssec:veryintro}. 
% Under the assumption that $\varphi$ is simply branched, i.e.\ all non-trivial ramification is
%of type $(2, 1^{d-2})$, the main by-product of our resolution from~\ref{ssec:veryintro} will be that all 
%Schreyer invariants are also scrollar invariants, up to some shift by a small fixed integer. 
%As an aside,
%let us point out that the scrollar invariants $e_i$ are the function-field analogues of $\log \lambda_i$, with $\lambda_1 \leq \lambda_2 \leq \ldots \leq \lambda_{d-1}$ the successive minima of the Minkowski lattice attached to a degree-$d$ number field; thus, our result implies that the Schreyer invariants are ``geometric", in Minkowski's sense.

%The existing literature contains one manifestation of this surprising fact: if $d = 4$ and $\varphi$ admits no ramification of type $(2^2)$ or $(4)$, then the scrollar invariants of the degree $3$ covering obtained from $\varphi$ through Recillas' trigonal construction~\cite{recillas} are given by 
%\begin{equation} \label{eq:scrollarrecillas}
%b_1^{(1)} + 2 \qquad \text{and} \qquad b_2^{(1)} + 2, \end{equation}
%see Casnati~\cite[Def.\,6.3-6.4]{casnati} or Deopurkar--Patel~\cite[Prop.\,4.6]{deopurkar_patel}. 
%By adopting a function-field theoretic point of view, 
% this allows for a vast generalization, as we now discuss in more detail.

\subsection{} \label{ssec:differentschreyer} Recall from~\ref{ssec:differentscrollar} that the scrollar invariants are being defined inconsistently in the existing literature. There is a corresponding ambiguity for the Schreyer invariants: several references, including Schreyer's original treatment~\cite{schreyer}, instead define them as
\[ 
  b_i^{(j)} - 2i - 2.
\]
The reason for the shifts is that working with $\mathcal{E}(-2)$ makes it more natural to use the generator $H - 2R \in \Pic(\PP(\mathcal{E}))$ rather than $H$. Thus, when interpreting our results for the invariants as they were initially introduced by Schreyer, the shifts must be taken into account. 
%The normalizations chosen in the current paper, both for the scrollar invariants and the Schreyer invariants, result in the most natural versions of Proposition~\ref{prop:scrollar.invs.of.hooks}
%and Theorem~\ref{thm:schreyerisscrollar}.

\subsection{Further applications and remarks.}
Our work has three immediate further applications, which are discussed in Section~\ref{sec:applications}. 
Firstly, Theorem~\ref{thm:scrollar.invariants.resolvent} in combination with Proposition~\ref{prop:scrollar.invs.of.hooks} and Theorem~\ref{thm:schreyerisscrollar} gives a way of constructing many new examples of multi-sets of integers that are realizable as the multi-set of
scrollar invariants of some $\PP^1$-cover. Most of these examples are highly non-balanced, i.e., there are large gaps between the scrollar invariants, so they live in a different regime from the ones provided by the existing literature, see e.g.~\cite{ballico,linearpencils,coppens,coppensmartens}.
Secondly, since the Schreyer invariants are scrollar, they are 
subject to the ``Maroni bound" coming from the Riemann--Roch theorem, which leads to non-trivial upper and lower bounds that seem unreported. Thirdly, from Theorem~\ref{thm:scrollar.invariants.resolvent} we see that Gassmann equivalent subgroups of $S_d$ (see~\ref{ssec:gassmannequiv}) give rise to resolvent covers having coinciding multi-sets of scrollar invariants. The number-theoretic counterpart of this statement reads that arithmetically equivalent number fields have Minkowski lattices with similar-sized successive minima: this fact was recently proved by the second-listed author~\cite{floris_minima}, taking inspiration from the current work.

\subsection{} \label{ssec:balanced} A multi-set $\Sigma$ of integers is ``balanced" if $|\max \Sigma - \min \Sigma \, | \leq 1$. 
Consider the Hurwitz space $\mathcal{H}_{d,g}$ of simply branched degree $d$ covers $\varphi : C \to \PP^1$ by curves of genus $g$.
Ballico~\cite{ballico} has proved that the multi-set of scrollar invariants of a sufficiently general 
element of $\mathcal{H}_{d,g}$ is balanced. 
Bujokas and Patel~\cite[Conj.\,A]{bujokas_patel} have conjectured
that the same is true for the splitting types of the syzygy bundles of a relative minimal resolution, provided that $g$ is sufficiently large with respect to $d$.\footnote{Work by Bopp and Hoff~\cite{bopphoff} shows the necessity of this assumption.} They prove this for the first syzygy bundle subject to the bound $g \geq (d-2)^2$, and for all higher syzygy bundles assuming $g \equiv 1 \bmod d$. 

More generally, for any partition $\lambda \vdash d$ one can wonder about the generic behaviour of its scrollar invariants  on $\mathcal{H}_{d,g}$. By Proposition \ref{prop:scrollar.invs.of.hooks} one cannot expect balancedness for all $\lambda$, even if $g$ is large enough.\footnote{We thank Aaron Landesman for pointing this out to us.} We ask:

\begin{problem}
Let $d \geq 2$ and $g \geq 0$ be integers and let $\lambda \vdash d$ be a non-trivial partition. Consider an algebraically closed field $k$ with $\charac k = 0$ or $\charac k > d$. Under what conditions on $d, g, \lambda$ can we conclude that the scrollar invariants of $\lambda$ with respect to  
a general element $\varphi : C \to \PP^1$ of $\mathcal{H}_{d, g}$ over $k$ are balanced?
%which partitions $\lambda\vdash d$ are the scrollar invariants of $\lambda$ with respect to a general element $\varphi: C\to \PP^1$ of $\HH_{d,g}$ balanced?
\end{problem}

%More generally, one can ask what the scrollar invariants of a partition $\lambda\vdash d$ are with respect to a general element $\varphi: C\to \PP^1$ of $\HH_{d,g}$.

A weaker question is answered in a recent preprint by Landesman and Litt~\cite[Ex.\,1.3.7]{LandesmanLitt}, who prove that the scrollar invariants of any $\lambda \vdash d$ with respect to a sufficiently general element $\varphi: C \to \PP^1$ of $\mathcal{H}_{d,g}$ are always ``consecutive": if $e_{\lambda, 1}\leq \ldots \leq e_{\lambda, \dim V_\lambda}$ are the scrollar invariants of $\lambda$ with respect to $\varphi$ then $e_{\lambda, i+1} - e_{\lambda, i} \leq 1$ for $i=1, \ldots, \dim V_\lambda -1$.

\subsection{Acknowledgements} The first-listed author is supported by the European Research Council (ERC) with grant nr.\ 101020788 Adv-ERC\-ISOCRYPT, by CyberSecurity Research Flanders with ref.\ VR20192203, and by Research Council KU
Leuven with grant nr.\ C14/18/067. The second-listed author is supported by the Research Foundation -- Flanders
(FWO) with grant nr.\ 11F1921N. The third-listed author is supported by the National Natural Science Foundation
of China with grant nr.\ 12071371. We have benefited from conversations with 
Alex Bartel, Marc Coppens, Lifan Guan, Florian Hess, Michael Hoff, Aaron Landesman, Alexander Lemmens, Dongwen Liu, Wenbo Niu, Frank-Olaf Schreyer, Takashi Taniguchi, Frederik Vercauteren and Yigeng Zhao, all of whom we thank for this.
We have also benefited from an inspiring ``Research in Pairs" stay at the Mathematisches Forschungsinstitut Oberwolfach in 2021. Finally, we owe thanks to an anonymous referee for many suggestions to improve the exposition.

\section{Auxiliary facts from representation theory} \label{sec:representation.theory}

\subsection{} \label{ssec:notation} We begin with some notation. Fix an integer $d \geq 4$. For a partition $\lambda = (d_r, \ldots, d_1)$ of $d$ we interchangeably write
\[ S_{d_1} \times S_{d_2} \times \cdots \times S_{d_r} \qquad \text{and} \qquad S_\lambda\] 
for the Young subgroup $\Sym\{1, \ldots, d_1\} \times \Sym\{d_1 + 1, \ldots, d_2\} \times \cdots \times \Sym\{d - d_r + 1, \ldots, d\}$ of $S_d = \Sym\{1, \ldots, d\}$, obtained by concatenating cycles. 
If no confusion is possible, then for $d' < d$ we will view $S_{d'}$ as a subgroup of $S_d$ by identifying it with $S_1 \times S_1 \times \cdots \times S_1 \times S_{d'}$. For $i \in  \{ 2, \ldots, d-2 \}$ we write $\lambda_i$ to denote the partition $(d-i, 2, 1^{i-2})$. We extend this notation by letting
$\lambda_0 = (d)$, $\lambda_1 = (d-1, 1)$ and $\lambda_d = (1^d)$. Note that we do not define $\lambda_{d-1}$.
If $R$ is a $\ZZ$-graded ring (e.g., a polynomial ring), then for any $i \in \ZZ$ we write $R_i$ to denote its homogeneous degree $i$ part.

\subsection{} We state two basic facts on representations of finite groups, where we work over an arbitrary field $F$ with $\charac F = 0$ or $\charac F > d$. We stick to $S_d$, but modulo further assumptions on $F$ (being algebraically closed of characteristic $0$ is always sufficient) the direct generalizations of these results hold for any finite group and are well-known to specialists. 

%\todo[inline]{Yongqiang}

\begin{lemma}\label{lem: induced and fixed space} %\todo{guess we should be able to find a reference for this one?}
Let $V$ be an irreducible representation of $S_d$ and let $H \subseteq S_d$ be a subgroup. Then
$\dim V^H = \mult (V, \Ind_H^{S_d} \mathbf{1})$.
\end{lemma}

\begin{proof}
 We have
 \begin{align*}
   \dim V^H & = \mult( \mathbf{1}, \Res^{S_d}_H V) 
             = \dim \Hom_H (\mathbf{1}, \Res^{S_d}_H V) \\
            & = \dim \Hom_{S_d} (\Ind^{S_d}_H \mathbf{1}, V) 
             =  \mult (V, \Ind_H^{S_d} \mathbf{1})
 \end{align*}
 where the third equality follows from Frobenius reciprocity.
\end{proof}

\begin{lemma}\label{lem: projecting to a single W}
Consider a partition $\lambda \vdash d$. Then there exists a unique $\rho \in Z(F[S_d])$ such that $\rho(V_\mu)=0$ for every partition $\mu \neq \lambda$ and such that $\rho$ induces the identity map on $V_\lambda$.
\end{lemma}

\noindent (Here $Z(F[S_d])$ denotes the center of the group ring $F[S_d]$.)

\begin{proof}
This follows from~\cite[Thm.\,8]{serreRepr}.
%Denote by $\mathcal{C}_1, \ldots, \mathcal{C}_n$ the conjugacy classes of $S_d$. Define
%\[
%c_i = \sum_{g\in \mathcal{C}_i} g
%\]
%and recall from~\cite[\S1.10]{sagan} that $c_1, \ldots, c_n$ form a basis for $Z(F[S_d])$. The $c_i$ induce $S_d$-equivariant maps on every representation of $S_d$, because $c_i$ is contained in the center of $F[S_d]$. For $\mu$ a partition of $d$, we obtain by Schur's lemma that $c_i$ acts by scalar multiplication on $V_\mu$. In fact
%\[
%c_i: V_\mu\to V_\mu: x\mapsto \frac{\chi_\mu(\mathcal{C}_i)\cdot | \mathcal{C}_i | }{\dim V_\mu},
%\]
%where $\chi_\mu$ is the character of $V_\mu$. Define the matrix $M_{\mu, i} = \frac{\chi_\mu(\mathcal{C}_i)\cdot | \mathcal{C}_i | }{\dim V_\mu}$ and note that this is simply the character table of $S_d$ multiplied on the left by $D_1 = \diag (1/\dim V_\mu)_\mu$ and on the right by $D_2 = \diag (| \mathcal{C}_i|)_i$. By the orthogonality relations, the determinant of the character table is contained in $F^\times$. But so are the determinants of $D_1$ and $D_2$. Hence $\det M \in F^\times$ and by solving a linear system we obtain a unique $\rho\in Z(F[S_d])$ with the desired properties.
\end{proof}

\subsection{} More specific to $S_d$, we have:
\begin{lemma}\label{lem: multiplicity fixed}
For any $2\leq i\leq d-2$ we have $\dim \left(V_{\lambda_i}^{S_{\lambda_i}}\right)=1$.
\end{lemma}

\begin{proof}
We know that $\mult (V_{\lambda_i}, \Ind_{S_{\lambda_i}}^{S_d} \mathbf{1}) = 1$, see e.g.~\cite[Cor.\,2.4.7]{sagan}, so
this follows from Lemma~\ref{lem: induced and fixed space}. 
\end{proof}

We also need facts on how certain tensor products decompose into irreducibles:

\begin{lemma} \label{lem: tensorstandard}
 Consider a partition $\lambda \vdash d$. Then
  \[ V_\lambda \otimes V_{(d-1,1)} \cong \bigoplus_{\mu \vdash d} V_\mu^{c_{\mu \lambda} - \delta_{\mu \lambda}}, \]
  where $c_{\mu \lambda}$ equals the number of ways of transforming $\mu$ into $\lambda$ by removing a box and adding a box; here $\delta_{\mu \lambda}$ denotes the Kronecker delta.
\end{lemma}

\begin{proof}
This can be found in~\cite[p.\,257-258]{hamermesh}.
\end{proof}

\noindent It is understood that the box removals are valid, in the sense that they result in the Young diagram of a partition of $d-1$. Writing $\lambda = (d_r, \ldots, d_1)$, we note that Lemma~\ref{lem: tensorstandard} admits the rephrasing
\[ V_\lambda \otimes V_{(d-1,1)} \cong V_\lambda^{| \{d_1, \ldots, d_r \}| - 1} \oplus \bigoplus_{\mu} V_\mu, \]
with $\mu$ ranging over all partitions whose Young diagram can be obtained from that of $\lambda$ by removing one box and adding \emph{another} box. Warning: here $|\{d_1, \ldots, d_r\}|$ denotes the cardinality as a set, rather than as a multi-set.

\begin{lemma} \label{lem: tensornotsostandard}
  Consider a partition $\lambda \vdash d$. Then
  \[ V_{\lambda} \otimes V_{(d-2,2)} \cong \bigoplus_{\mu \vdash d} V_\mu^{-c_{\mu \lambda} + \frac{1}{2} (d_{\mu \lambda} + e_{\mu \lambda} - e'_{\mu \lambda} ) }, \]
  where $c_{\mu \lambda}$, $d_{\mu \lambda}$, $e_{\mu \lambda}$, $e'_{\mu \lambda}$ denote
  the number of ways of transforming $\mu$ into $\lambda$ by 
  \begin{itemize}
    \item removing a box and adding a box,
    \item consecutively removing two boxes and consecutively adding two boxes,
    \item removing two horizontally adjacent boxes and adding two horizontally adjacent boxes, or removing two vertically adjacent boxes and adding two vertically adjacent boxes,
    \item removing two horizontally adjacent boxes and adding two vertically adjacent boxes, or removing two vertically adjacent boxes and adding two horizontally adjacent boxes,
  \end{itemize}
  respectively.
\end{lemma}

\begin{proof}
This follows along~\cite[p.\,258-259]{hamermesh} or from a double application of Lemma~\ref{lem: tensorstandard}.
\end{proof}

Using that
\begin{equation} \label{eq: decompfirststep} 
  \Sym^2 V_{(d-1,1)} \ \cong \ V_{(d)} \oplus V_{(d-1,1)} \oplus V_{(d-2,2)},
\end{equation}
see~\cite[Prob.\,4.19]{fultonharris}, we get the following corollaries:
\begin{corollary} \label{cor: multiplicity double next}
We have
\[ \mult(V_{\lambda_{i+2}}, V_{\lambda_i}\otimes \Sym^2 V_{(d-1, 1)} ) = \left\{ \begin{array}{ll}
0 & \text{for all $2 \leq i \leq d-4$,} \\
1 & \text{if $i = d-2$.} \\ \end{array} \right. \]
\end{corollary}
\begin{proof}
By Lemma~\ref{lem: tensorstandard} we only need to look at $V_{\lambda_{i+2}} \otimes V_{(d-2,2)}$, which we can handle with Lemma~\ref{lem: tensornotsostandard}.
Assume $2 \leq i \leq d-4$. Clearly $c_{\lambda_{i+2} \lambda_i} = 0$. There is only one way to transform
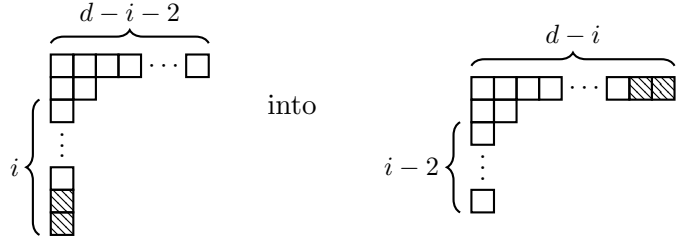
\begin{equation*}
    \begin{tikzpicture}[scale=0.3,baseline=(current  bounding  box.center)]
      \draw[pattern=north west lines, thick] (0,-1) rectangle (1,0);
      \draw[pattern=north west lines, thick] (0,0) rectangle (1,1);
      \draw[thick] (0,1) rectangle (1,2);
      \node at (0.5,3.33) {\small $\vdots$};
      \draw[thick] (0,4) rectangle (1,5);

      \draw[thick] (0,5) rectangle (1,6);
      \draw[thick] (1,5) rectangle (2,6);
      
      \draw[thick] (0,6) rectangle (1,7);
      \draw[thick] (1,6) rectangle (2,7);
      \draw[thick] (2,6) rectangle (3,7);
      \draw[thick] (3,6) rectangle (4,7);      
      \draw[thick] (6,6) rectangle (7,7);
      \node at (5.1,6.5) {\small $\cdots$};
      \draw[thick,decorate,decoration={brace,amplitude=5pt}](0,7.5) -- (7,7.5);
      \draw[thick,decorate,decoration={brace,amplitude=5pt}](-0.5,-1) -- (-0.5,5);    
      \node at (3.5,8.85) {\small $d - i - 2$};
      \node at (-1.6,2) {\small $i$};
    \end{tikzpicture}
    \qquad \text{into} \qquad 
    \begin{tikzpicture}[scale=0.3,baseline=(current  bounding  box.center)]
      \draw[thick] (0,1) rectangle (1,2);
      \node at (0.5,3.33) {\small $\vdots$};
      \draw[thick] (0,4) rectangle (1,5);

      \draw[thick] (0,5) rectangle (1,6);
      \draw[thick] (1,5) rectangle (2,6);
      
      \draw[thick] (0,6) rectangle (1,7);
      \draw[thick] (1,6) rectangle (2,7);
      \draw[thick] (2,6) rectangle (3,7);
      \draw[thick] (3,6) rectangle (4,7);      
      \draw[thick] (6,6) rectangle (7,7);
      \draw[pattern=north west lines, thick] (7,6) rectangle (8,7);
      \draw[pattern=north west lines, thick] (8,6) rectangle (9,7); 
      \node at (5.1,6.5) {\small $\cdots$};
      \draw[thick,decorate,decoration={brace,amplitude=5pt}](0,7.5) -- (9,7.5);
      \draw[thick,decorate,decoration={brace,amplitude=5pt}](-0.5,1) -- (-0.5,5);    
      \node at (4.5,8.85) {\small $d - i$};
      \node at (-2.65,3) {\small $i - 2$};
    \end{tikzpicture} 
    \vspace{0.2cm}
\end{equation*}
by consecutive removal of two boxes and consecutive addition of two boxes, so $d_{\lambda_{i+2} \lambda_i} = 1$.
However, these boxes are vertically resp.\ horizontally adjacent, so $e_{\lambda_{i+2} \lambda_i} = 0$ and 
$e'_{\lambda_{i+2} \lambda_i} = 1$. The case $i = d-2$ follows similarly.
\end{proof}

We also need particular statements about $\Sym^3 V_{(d-1,1)}$ for which we give ad-hoc proofs:
\begin{lemma} \label{lem: sym3}
Assume $d \geq 5$. Then 
\[ \mult(V_{\lambda_3}, \Sym^3 V_{(d-1,1)})  = 
\mult(V_{\lambda_d}, V_{\lambda_{d-3}} \otimes \Sym^3 V_{(d-1,1)}) = 0.\]
\end{lemma}

\begin{proof}
As for the first vanishing, 
the Specht module $V_{\lambda_3}$ can be naturally realized inside the polynomial ring $F[z_1, z_2, \ldots, z_d]$ equipped with the natural $S_d$-action (where the $z_i$ are variables),
where it is generated by polynomials of degree $4 > 3$, see~\cite[Prob.\,4.47]{fultonharris}.
From~\cite[Prop.\,5]{thiery} we see that $V_{\lambda_3}$ cannot appear in the decomposition of
$F[z_1, z_2, \ldots, z_d]_3$, which contains $\Sym^3 V_{(d-1,1)}$ as a subrepresentation, so the desired conclusion follows.

As for the second vanishing, 
tensoring $\Sym^3 V_{(d-1,1)}$ with $V_{\lambda_{d-3}}$ produces the sign representation $V_{(1, \ldots, 1)}$ as a component if and only if tensoring with the dual $V_{(d-3, 2, 1)}$ of $V_{\lambda_3}$ produces the trivial representation as a component. This can only happen if $V_{(d-3, 2, 1)}$ is a subrepresentation of $\Sym^3 V_{(d-1,1)}$, which it is not, by the first part.
\end{proof}

Finally, we will also make use of 
\begin{equation} \label{eq: adhocmults} 
  \mult (V_{(1^3)}, \Sym^3 V_{(2,1)}) = 1 \quad \text{and} \quad
 \mult (V_{(1^4)}, \Sym^4 V_{(3, 1)}) = 0,
\end{equation}
which follow very easily along the previous lines of thought (or by explicit computation).

\section{A minimal free resolution from Galois theory} \label{sec:minfreerep}

%In this section we construct our minimal free resolution using Galois theory and representation theory as discussed in~\ref{ssec:veryintro}.

\subsection{} We resume with the notation and assumptions from~\ref{ssec:veryintro}, where moreover
we assume (without loss of generality) that $\Tr_{L/F}(\alpha_i) = 0$ for $i = 1, \ldots, d-1$. 
Note that Bhargava's point configuration~\eqref{eq:asspoints}  is closed under the action of $\Gal(L/F)$. 
Therefore, as an algebraic set, it is defined over $F$.
Write $R = F[x_1, \ldots, x_{d-1}]$ and let $I \subseteq R$ be the ideal of this algebraic set.
The main result of Section~\ref{sec:minfreerep} is a new and explicit minimal free resolution of the coordinate ring $R/I$ as a graded $R$-module.
%The goal of this section is the construction of a minimal graded free resolution of Bhargava's point configuration~\eqref{eq:asspoints} from representation theory. 
As mentioned, our approach is Galois-theoretic and therefore quite different from that of Wilson~\cite[\S5]{wilsonphd}. Instead, there is common ground 
with the approach of Behnke~\cite{behnke} and it may be possible to deduce several of the statements below from his work (we did not succeed in doing so). We note that this section is technical and the reader may want to skip it upon a first reading.

\subsection{}  \label{ssec:resumenotation} Choosing $\alpha$ such that $K = F(\alpha)$, we can identify $\Gal(L/F)$ with 
\[ \Sym \left\{ \alpha^{(1)} = \sigma_1(\alpha) =  \alpha,  \alpha^{(2)} = \sigma_2(\alpha),  \ldots,  \alpha^{(d)} = \sigma_d(\alpha)  \right\}, \]
which in turn is identified with $S_d = \Sym \{1, 2, \ldots, d\}$ 
by writing indices rather than field elements, e.g., $(1\,2)$ refers to the field automorphism swapping $\alpha^{(1)}$ and $\alpha^{(2)}$ and fixing the other $\alpha^{(i)}$'s.
As discussed in the introduction, we view $L$ as the regular representation of $S_d$ along with its isotypic components $W_\lambda$, with $\lambda$ running over the partitions of $d$. For example, $W_{(d)}$
and $W_{(1^d)}$ are the one-dimensional subspaces generated by $1$ and
\[ \delta = \prod_{1 \leq i < j \leq d} (\alpha^{(i)} - \alpha^{(j)}), \] 
respectively. A less degenerate example is $W_{(d-1,1)}$, which 
is the $(d-1)^2$-dimensional subspace generated by the trace zero elements
\[ \begin{array}{cccc} \alpha_1^{(1)}, & \alpha_2^{(1)}, & \cdots & \alpha_{d-1}^{(1)}, \\
\alpha_1^{(2)}, & \alpha_2^{(2)}, & \cdots & \alpha_{d-1}^{(2)}, \\
\vdots & \vdots & \ddots & \vdots \\
\alpha_1^{(d-1)}, & \alpha_2^{(d-1)}, & \cdots & \alpha_{d-1}^{(d-1)}, \\ \end{array} \]
where $\alpha_i^{(j)} = \sigma_j(\alpha_i)$.
Each of the vertical columns spans a $V_{(d-1,1)}$.

We fix a similar kind of basis for $W_{\lambda_i}$, for $i = 2,  \ldots, d-2$. 
Namely, from Lemma~\ref{lem: multiplicity fixed} we know that each irreducible subrepresentation of $W_{\lambda_i}$ has
a one-dimensional intersection with $L^{S_{\lambda_i}}$. Thus we can find linearly independent elements
\[ \omega_1^i, \ \ \omega_2^i, \ \ \ldots, \ \ \omega_{\beta_{i - 1}}^i \ \ \in W_{\lambda_i} \]
that are fixed by $S_{\lambda_i}$. This leaves each $\omega^i_j$ with $\beta_{i - 1}$ linearly independent conjugates, which together span a $V_{\lambda_i}$. Taking the union over all $j$ then produces our basis of $W_{\lambda_i}$.

\subsection{} \label{ssec:shapeofresolution}
 %It again relies on representation theory, but our approach is Galois-theoretic and therefore quite different from that of Wilson.
%Most notably, in our case the $S_d$-module structure comes from the action of $\Gal(L/F)$. 
For $2 \leq i \leq d-2$ we consider the 
$F$-vector space $V_i$ obtained from $W_{\lambda_i}$ by intersecting it with $L^{S_{\lambda_i}}$, where $\lambda_i$ is as defined in~\ref{ssec:notation}; in other words
\[ V_i = \Span \{ \omega_1^i, \omega_2^i, \ldots, \omega_{\beta_{i - 1}}^i \}. \]
We stress that these $V_i$'s are \emph{not} subrepresentations of $L$. 
Instead, we think of $V_i$ as a ``horizontal slice" of our isotypic component $W_{\lambda_i}$, which 
can be recovered from $V_i$ by taking its closure under the action of $S_d$ (in other words, $W_{\lambda_i}$ is the smallest subrepresentation of $L$ containing $V_i$). 
Likewise, we define $V_0 = F$ and $V_d = \Span\{\delta\} = F \delta$.
Our resolution will take the form
\begin{equation} \label{eq:resolution.Bhargava.points}
\begin{tikzpicture}[scale=0.23,baseline=(current  bounding  box.center)]
      \draw[thick] (1,-3) rectangle (2,-2);
      \draw[thick] (1,-2) rectangle (2,-1);
      \draw[thick] (1,-1) rectangle (2,0);
      \draw[thick] (1,0) rectangle (2,1);
      \draw[thick] (1,1) rectangle (2,2);
      \node at (1.5,3.42) {\tiny $\vdots$};
      \draw[thick] (1,4) rectangle (2,5);
      \draw[thick] (1,5) rectangle (2,6);
      \draw[thick] (1,6) rectangle (2,7);
      
      \draw[thick] (12.5,-1) rectangle (13.5,0);
      \draw[thick] (12.5,0) rectangle (13.5,1);
      \draw[thick] (12.5,1) rectangle (13.5,2);
      \node at (13,3.42) {\tiny $\vdots$};
      \draw[thick] (12.5,4) rectangle (13.5,5);

      \draw[thick] (12.5,5) rectangle (13.5,6);
      \draw[thick] (13.5,5) rectangle (14.5,6);
      
      \draw[thick] (12.5,6) rectangle (13.5,7);
      \draw[thick] (13.5,6) rectangle (14.5,7);  

      \draw[thick] (28.2,0) rectangle (29.2,1);
      \draw[thick] (28.2,1) rectangle (29.2,2);
      \node at (28.7,3.42) {\tiny $\vdots$};
      \draw[thick] (28.2,4) rectangle (29.2,5);

      \draw[thick] (28.2,5) rectangle (29.2,6);
      \draw[thick] (29.2,5) rectangle (30.2,6);
      
      \draw[thick] (28.2,6) rectangle (29.2,7);
      \draw[thick] (29.2,6) rectangle (30.2,7);  
      \draw[thick] (30.2,6) rectangle (31.2,7);    
      
      \node at (21.3,8.7) {$0 \to V_d^\ast \otimes R(-d) \to V_{d-2}^\ast \otimes R(-d+2) \to V_{d-3}^\ast \otimes R(-d+3) \to \ldots$ };
\end{tikzpicture} 
\end{equation} 
\begin{equation*}
\hspace{3cm} \begin{tikzpicture}[scale=0.23,baseline=(current  bounding  box.center)]
      \node at (20, 8.7) {$  \to V_3^\ast \otimes R(-3) \to V_2^\ast \otimes R(-2) \to V_0^\ast \otimes R \to R/I \to 0$,};

      \draw[thick] (2.6,4) rectangle (3.6,5);
      \draw[thick] (2.6,5) rectangle (3.6,6);
      \draw[thick] (3.6,5) rectangle (4.6,6);      
      \draw[thick] (2.6,6) rectangle (3.6,7);
      \draw[thick] (3.6,6) rectangle (4.6,7);
      \draw[thick] (4.6,6) rectangle (5.6,7);
      \draw[thick] (7.6,6) rectangle (8.6,7);
      \draw[thick] (8.6,6) rectangle (9.6,7);
      \node at (6.7,6.5) {\tiny $\cdots$};

      \draw[thick] (14.1,5) rectangle (15.1,6);
      \draw[thick] (15.1,5) rectangle (16.1,6);      
      \draw[thick] (14.1,6) rectangle (15.1,7);
      \draw[thick] (15.1,6) rectangle (16.1,7);
      \draw[thick] (16.1,6) rectangle (17.1,7);
      \draw[thick] (19.1,6) rectangle (20.1,7);
      \draw[thick] (20.1,6) rectangle (21.1,7);
      \draw[thick] (21.1,6) rectangle (22.1,7);
      \node at (18.2,6.5) {\tiny $\cdots$};
   
      \draw[thick] (25.5,6) rectangle (26.5,7);
      \draw[thick] (26.5,6) rectangle (27.5,7);
      \draw[thick] (27.5,6) rectangle (28.5,7);
      \draw[thick] (30.5,6) rectangle (31.5,7);
      \draw[thick] (31.5,6) rectangle (32.5,7);
      \draw[thick] (32.5,6) rectangle (33.5,7);
      \draw[thick] (33.5,6) rectangle (34.5,7);
      \draw[thick] (34.5,6) rectangle (35.5,7);      

      \node at (29.6,6.5) {\tiny $\cdots$};
\end{tikzpicture}
\vspace{0.2cm}
\end{equation*}
%\begin{align}
%0 \to V_d^\ast \otimes R(-d) \to V_{d-2}^\ast \otimes R(-d+2) \to V_{d-3}^\ast \otimes R(-d+3) \to \ldots \label{eq:resolution.Bhargava.points} \\
%\to V_4^\ast \otimes R(-4) \to V_3^\ast \otimes R(-3) \to V_2^\ast \otimes R(-2) \to V_0^\ast \otimes R \to R/I \to 0, \notag
%\end{align}
where $V_i^\ast = \Hom_F(V_i, F)$ denotes the dual of $V_i$; at each step of the resolution we have depicted the Young diagram of the corresponding partition.
%Thus our resolution indeed realizes the correct Betti numbers. 
Our syzygy modules do not come equipped with an $S_d$-module structure, in particular one is free to drop the one-dimensional factors $V_0^\ast$ and $V_d^\ast$ if wanted, but these are included to emphasize the self-duality of the resolution.
The space
\[ V_1 = W_{\lambda_1} \cap L^{S_{\lambda_1}} = \Span \{ \alpha_1, \alpha_2, \ldots, \alpha_{d-1} \}, \]
corresponding to the standard representation, seems missing, but will play a key role in
the construction of the morphisms, and in fact 
the polynomial ring $R = F[x_1, \ldots, x_{d-1}]$ will come about as $\Sym V_1^\ast$ (as is also the case in Behnke's resolution from~\cite{behnke}).

\begin{comment}
\begin{remark}
As pointed out to us by the referee and Aaron Landesman, our resolution can be interpreted as the Casnati--Ekedahl resolution over the base $BS_d$. Since ??, we give a more ad-hoc treatment of this resolution. The $V_i$ can also be defined using the Mehta--Seshadri correspondence~\cite{MS} (see also~\cite[Sec.\,2]{landesmanlitt2}). \todo[inline]{Write more... Also say something about this Mehta--Seshadri correspondence.}
\end{remark}
\end{comment}

%\subsection{} We recall that any two configurations of $d$ points in $\PP^{d-2}$, say having coordinates in $F^\text{alg.cl.}$, 
%with no $d-1$ of them lying on a hyperplane, are projectively equivalent. Thus, even though the point set~\eqref{eq:asspoints} seems very specific, when composed with a suitable linear change of coordinates, our minimal free resolution applies to arbitrary sets of $d$ points in general position.

\subsection{Construction.} \label{ssec:firststep} We now explain how the morphisms are constructed.
We begin with the first step of the resolution~\eqref{eq:resolution.Bhargava.points}.
Start from the decomposition~\eqref{eq: decompfirststep} of $\Sym^2 V_{(d-1,1)}$
into irreducible subrepresentations.
Let $y_1, \ldots, y_{d-1}$ be an $F$-basis for $V_{(d-1, 1)}$ such that $y_1$ is fixed by $S_{d-1}$, and all other $y_i$ are conjugate to $y_1$ in a way that is compatible with the action of $S_d$ on $L$,
i.e., for all $i = 1, \ldots, d-1$ we have $\sigma_i(y) = y_i$. 
By Lemma~\ref{lem: multiplicity fixed}
%Since $\mult (V_{(d-2, 2)}, \Ind_{S_2\times S_{d-2}}^{S_d} \mathbf{1})=1$, it follows from Lemma \ref{lem: induced and fixed space} that 
there is, up to scalar multiplication, a unique element in $V_{(d-2,2)}$ that is fixed by
$S_2 \times S_{d-2}$. Under the isomorphism~\eqref{eq: decompfirststep} this corresponds to an element
\[ p^1 = \sum_{m,n = 1}^{d-1} p^1_{mn} y_m \otimes y_n \in \Sym^2 V_{(d-1,1)}  \]
for certain $p_{mn}^1 \in F$ where, without loss of generality, we may assume that $p^1_{mn} = p^1_{nm}$ for
all $m, n$.
% is fixed by $S_2 \times S_{d-2}$ and such that its conjugates generate a representation isomorphic
%to $V_{(d-2,2)}$. 
We use this element to construct a map
\[ 
\psi_1 : \Sym^2 V_1 \to V_2 : \alpha \otimes \beta \mapsto \sum_{m,n=1}^{d-1} p^1_{mn} \alpha^{(m)} \beta^{(n)}  
\]
where as usual $\alpha^{(m)} = \sigma_m(\alpha)$ and $\beta^{(n)} = \sigma_n(\beta)$.
%where we become more precise on the aforementioned compatibility requirement: it is assumed that any Galois automorphism $\sigma$ mapping $\alpha = \alpha^{(1)} \in V_1$ to $\alpha^{(m)}$ also maps $y_1$ to $y_m$. 
Through dualization we obtain a map $\psi_1^\ast : V_2^\ast \to \Sym^2 V_1^\ast$ whose codomain, 
after identifying $V_1^\ast$ with $R_1$, can be viewed as $R_2$. Thus this yields a map $V_2^*\otimes R(-2) \to V_0^\ast \otimes R$, as desired.

\subsection{} \label{ssec:explicitquadrics} Before discussing the next steps, we show that $\psi_1$ can be made quite concrete.
The Specht module $V_{(d-2,2)}$ can be naturally realized inside
$F[z_1,  \ldots, z_d]_2$.
Explicitly, it is the subspace generated by $(z_1 - z_2)(z_3 - z_4)$ and all its conjugates~\cite[Prob.\,4.47]{fultonharris}. 
By Lemma~\ref{lem: multiplicity fixed},
inside $F[z_1,  \ldots, z_d]_2 \cong V_{(d)}^2 \oplus V_{(d-1,1)}^2 \oplus V_{(d-2,2)}$,
%Since $\mult (V_{(d-2, 2)}, \Ind_{S_2\times S_{d-2}}^{S_d} \mathbf{1})=1$, it follows from Lemma \ref{lem: induced and fixed space} that 
this subspace contains up to scalar multiplication a unique polynomial that is fixed by $S_2\times S_{d-2}$. It is easily made explicit:
\[
p^1 = \sum_{\tau \in S_{d-2}} (z_1 - z_{\tau(3)})(z_2 - z_{\tau(4)})
\]
%\[
%p^1(z) = z_1z_2 + \frac{1}{d-3}\sum_{3\leq i<j\leq d} z_iz_j - \frac{1}{d-1}\sum_{1\leq i<j\leq d}z_iz_j.
%\]
(abusingly, we again call this polynomial $p^1$; here we recall from~\ref{ssec:notation} that we 
view $S_{d-2}$ as the subgroup of $S_d$ fixing $1$ and $2$).
This gives rise to a map
\[
V_1\to V_2: \alpha\mapsto p^1(\alpha^{(1)}, \ldots, \alpha^{(d)})
\]
which is nothing but the quadratic map corresponding to the symmetric bilinear map $\psi_1$ from above; let us call it $\psi_1'$. With respect to our bases $\alpha_1, \ldots, \alpha_{d-1}$ of $V_1$ and $\omega_1^2, \ldots, \omega^2_{\beta_1}$ of $V_2$, it may be represented as
\begin{equation*} %\label{eq:reprofpsi1}
\psi_1'\left(\sum_{i= j}^{d-1} a_j\alpha_j\right) = \sum_{\ell = 1}^{\beta_1} \sum_{j,k = 1}^{d-1} Q^\ell_{j,k}a_ja_k \omega^2_\ell,
\end{equation*}
for certain $Q^\ell_{j,k}\in F$, where we can assume $Q^\ell_{jk} = Q^\ell_{kj}$ for all $j,k$. Thus $\psi_1$ defines $\beta_1$ quadratic forms $Q^\ell \in R$. One checks that
\[ \psi_1^\ast : V_2^\ast \to \Sym^2 V_1^\ast : \omega_\ell^{2 \ast} \mapsto \sum_{j,k=1}^{d-1} Q_{jk}^\ell \alpha_j^\ast \otimes \alpha_k^\ast, \]
therefore the image of $V_2^\ast \otimes R(-2) \to R$ is the ideal generated by these $Q^\ell$'s. 
We will soon prove that it equals $I$. We end by noting that, for $d = 4,5$,
the quadratic forms $Q^\ell$ can also be obtained by applying Bhargava's parametrizations from~\cite{bhargavaquarticrings,bhargavaquinticrings} to
the extension $K / F$.

\subsection{} \label{ssec:nextsteps}
Now assume that $2 \leq i \leq d-3$. Let $y_1, \ldots, y_{d-1}$ be as before
and take an $F$-basis 
$w^i_1, \ldots, w^i_{\beta_{i-1}}$ of $V_{\lambda_i}$ such that $w_1^i$ is fixed by $S_{\lambda_i}$, and all other $w^i_j$'s are conjugate to $w^i_1$.
 By Lemma~\ref{lem: multiplicity fixed} and~Lemma~\ref{lem: tensorstandard} we have that 
\[
\mult(V_{\lambda_{i+1}}, V_{\lambda_i} \otimes V_{(d-1,1)})=1, \qquad \dim\left(V_{\lambda_{i+1}}^{S_{\lambda_{i+1}}}\right) = 1.
\]
Hence there is, up to scalar multiplication, a unique element
\[
p^i = \sum_{m = 1}^{\beta_{i-1}} \sum_{n = 1}^{d-1} p^i_{mn}w^i_m\otimes y_n \in V_{\lambda_i}\otimes V_{(d-1, 1)}
\]
which is fixed by $S_{\lambda_{i+1}}$ and such that its conjugates generate a representation isomorphic to $V_{\lambda_{i+1}}$. We use this $p^i$ to construct a map
\[
\psi_i: V_i\otimes V_1\to V_{i+1}: \omega \otimes \alpha \mapsto  \sum_{m = 1}^{\beta_{i-1}} \sum_{n = 1}^{d-1} p_{mn}^i \omega^{(m)}\alpha^{(n)},
\]
where the conjugation of $\omega$ is labelled compatibly with that of $w^i_1$, that is,
$\omega^{(m)} = \sigma(\omega)$ for any $\sigma$ mapping $w_1^i$ to $w_m^i$. 
By identifying $V_1^\ast$ with $R_1$ as before, the dual map $\psi_i^\ast : V_{i+1}^\ast \to V_i^\ast \otimes V_1^\ast$
is naturally converted into a map $V_{i+1}^\ast \otimes R(-i-1) \to V_i^\ast \otimes R(-i)$.

\subsection{} \label{ssec:laststepconstruction}
We have now constructed our resolution~\eqref{eq:resolution.Bhargava.points} except for the last step. For this, consider the representation $V_{\lambda_{d-2}}\otimes \Sym^2 V_{(d-1,1)}$ and note that, by Corollary~\ref{cor: multiplicity double next}, it contains a unique subrepresentation isomorphic to $V_{(1, \ldots, 1)}$. As above, consider a basis $w_1^{d-2}, \ldots, w_{\beta_{d-3}}^{d-2}$ 
of $V_{\lambda_{d-2}}$ such that $w_1^{d-2}$ is fixed by $S_{\lambda_{d-2}}$ and the other $w_j^{d-2}$'s are conjugate to it.
Then there exists an element
\[
p^{d-2} = \sum_{i = 1}^{\beta_{d-3}} \sum_{j,\ell = 1}^{d-1} p_{ij\ell}^{d-2} w_i^{d-2}\otimes (y_j\otimes y_\ell),
\]
unique up to scalar multiplication, on which $S_d$ acts as the sign representation. 
We can assume that $p_{ij\ell}^{d-2} = p_{i\ell j}^{d-2}$, leading to a linear map
\[
\psi_{d-2}: V_{d-2}\otimes \Sym^2 V_1 \to V_d.
\]
Upon dualizing, tensoring with $R$ and identifying $V_1^*$ with $R_1$, this yields a morphism of graded $R$-modules $\psi_{d-2}^*: V_d^* \otimes R(-d) \to V_{d-2}^*\otimes R(-d+2)$.
%We obtain our desired resolution 
%\begin{equation}\label{eq:resolution}
%0\to V_d^* \otimes S(-d) \to V_{d-2}^*\otimes S(-d+2)\to \ldots\to V_3^*\otimes S(-3)\xrightarrow{\psi_2^*} V_2^*\otimes S(-2)\xrightarrow{\psi_1^*} S\to S/I\to 0.
%\end{equation}

\subsection{A chain complex.} Our next goal is to prove that the sequence~\eqref{eq:resolution.Bhargava.points} is a complex.
We first discuss what this means in terms of the maps $\psi_i$.

\begin{lemma}\label{lem:complex.equivalence}
Assuming $d \geq 5$, the sequence~\eqref{eq:resolution.Bhargava.points} is a chain complex if and only if
\begin{enumerate}
\item the quadrics $Q^\ell$ vanish on the points~\eqref{eq:asspoints},
\item we have $\psi_2(\psi_1(\alpha\otimes \alpha)\otimes \alpha) = 0$ for any $\alpha\in V_1$, 
\item we have $\psi_{i+1}(\psi_i(\omega\otimes \alpha)\otimes \alpha) = 0$ for any $2\leq i\leq d-4, \alpha\in V_1$ and $\omega\in V_i$,
\item we have $\psi_{d-2}(\psi_{d-3}(\omega\otimes \alpha)\otimes (\alpha\otimes \alpha)) = 0$ for any $\alpha\in V_1$ and $\omega\in V_{d-3}$.
\end{enumerate}
\end{lemma}
\begin{proof}
The first step
$V_2^*\otimes R(-2) \to V_0^*\otimes R \to R/I$
being a complex is equivalent to the quadrics $Q^\ell$ being contained in the ideal $I$, i.e.\ the quadrics must vanish on~\eqref{eq:asspoints}.
The subsequent steps are handled using a direct computation with bases. 
%To this end,
%for $i = 3, \ldots, d-2$ we fix bases
%\[ \omega_1^i, \ldots, \omega_{\beta_{i-1}}^i \text{ of } V_i \]
%and we reconsider
%our bases $\alpha_1, \ldots, \alpha_{d-1}$ of $V_1$ and $\omega_1^2, \ldots, \omega_{\beta_1}^2$ of $V_2$.
%in which the sygyzies will make an explicit appearance. 

For the second step of the resolution, 
we may write the maps $\psi_1$ and $\psi_2$ as
\[
\psi_1(\alpha_j\otimes \alpha_k) = \sum_{\ell = 1}^{\beta_1} Q_{jk}^\ell \omega_\ell^2, \quad \psi_2(\omega_\ell^2\otimes \alpha_n) = \sum_{m = 1}^{\beta_2} L_{\ell n}^m \omega_m^3,
\]
for certain $Q_{jk}^\ell, L_{\ell n}^m$ in $F$. In terms of the dual bases with respect to $\Tr_{L/F}$, the maps $\psi_1^*$ and $\psi_2^*$ satisfy
\[
\psi_1^*(\omega_\ell^{2*}) = \sum_{j,k = 1}^{d-1} Q_{jk}^\ell \alpha_j^*\otimes \alpha_k^*, \quad \psi_2^*(\omega_m^{3*}) = \sum_{\ell = 1}^{\beta_1} \sum_{n = 1}^{d-1} L_{\ell n}^m \omega_\ell^{2*} \otimes \alpha_n^*.
\]
Let $q(x)$ be homogeneous in $R$, then in the sequence of maps a computation shows that
\begin{equation}\label{eq:psi1*(psi2*)}
\psi_1^*(\psi_2^*(\omega_m^{3*}\otimes q(x))) = q(x)\sum_{\ell = 1}^{\beta_1} \left( \sum_{n = 1}^{d-1} L_{\ell n}^m x_n\right) \left( \sum_{j,k = 1}^{d-1} Q_{jk}^\ell x_jx_k\right).
\end{equation}
On the other hand, letting $\alpha=\sum_r a_r\alpha_r \in V_1$ for $a_r\in F$, one computes that
\begin{equation}\label{eq:psi1(psi2)}
\psi_2(\psi_1(\alpha\otimes \alpha)\otimes \alpha) = \sum_{m = 1}^{\beta_2} \left( \sum_{\ell = 1}^{\beta_1} \left(\sum_{n = 1}^{d-1} L_{\ell n}^m a_n\right) \left( \sum_{j,k = 1}^{d-1} Q_{jk}^\ell a_j a_k\right) \right) \omega_m^3.
\end{equation}
We indeed see that~\eqref{eq:psi1(psi2)} is zero for all $\alpha \in V_1$ if and only if~\eqref{eq:psi1*(psi2*)} is zero for all $m = 1, \ldots, \beta_2$ and $q(x) \in R$. The ``only if" part relies on the fact that
a non-zero polynomial in $R$ cannot vanish on all of $F^{d-1}$, because $F$ is infinite: indeed, by assumption it admits an $S_d$-extension, so it cannot be finite.\footnote{But even for a finite field $F$
it is true that a non-zero polynomial of degree $< \charac F$ cannot vanish everywhere.}

The middle steps are handled similarly. Explicitly, for $2\leq i\leq d-4$ we may write the maps $\psi_i$ and $\psi_{i+1}$ as
\[
\psi_i(\omega_j^i\otimes \alpha_k)= \sum_{\ell = 1}^{\beta_i} L_{jk}^\ell \omega_\ell^{i+1}, 
\quad \psi_{i+1}(\omega_\ell^{i+1}\otimes \alpha_n) = \sum_{m = 1}^{\beta_{i+1}} L'^m_{\ell n} \omega_m^{i+2},
\]
for certain $L_{jk}^\ell, L_{\ell n}^m \in F$.
On the dual bases this gives
\[
\psi_i^*(\omega_\ell^{i+1*}) = \sum_{j = 1}^{\beta_{i-1}} \sum_{k = 1}^{d-1} L_{jk}^\ell \omega_j^{i*}\otimes \alpha_k^*, \quad \psi_{i+1}^*(\omega_m^{i+2*}) = \sum_{\ell = 1}^{\beta_i} \sum_{n = 1}^{d-1} L'^m_{\ell n} \omega_\ell^{i+1*}\otimes \alpha_n^*.
\]
For $q(x)\in R$, one computes that
\begin{equation}\label{eq:psi_i*(psi_i+1*}
\psi_i^*(\psi_{i+1}^*(\omega_m^{i+2*}\otimes q(x)) = \sum_{j= 1}^{\beta_{i-1}} \omega_j^{i*}\otimes \left( q(x)\sum_{\ell = 1}^{\beta_i} \left( \sum_{n=1}^{d-1} L'^m_{\ell n} x_n\right) \left(\sum_{k=1}^{d-1} L_{jk}^\ell x_k\right) \right).
\end{equation}
On the other hand, if we let $\alpha = \sum_r a_r\alpha_r$ for $a_r\in F$, then
\begin{equation}\label{eq:psi_i+1(psi_i)}
\psi_{i+1}(\psi_i(\omega_j^i\otimes \alpha)\otimes \alpha) = \sum_{m = 1}^{\beta_{i+1}} 
\left( \sum_{\ell = 1}^{\beta_i} \left( \sum_{n = 1}^{d-1} L'^m_{\ell n} a_n\right) \left( \sum_{k=1}^{d-1} L_{j k}^\ell a_k\right)\right) \omega_m^{i+2}.
\end{equation}
Again, it is clear that~\eqref{eq:psi_i+1(psi_i)} is zero for all $j = 1, \ldots, \beta_{i-1}$ and all $\alpha \in V_1$ if and only if~\eqref{eq:psi_i*(psi_i+1*} is zero for all $m = 1, \ldots, \beta_{i+1}$ and all $q(x) \in R$.

We omit the details of the final step, which again can be dealt with analogously.
\end{proof}

 The reader may have observed that the syzygies of our resolution make an explicit 
appearance in the above proof, more precisely in~\eqref{eq:psi1*(psi2*)} and~\eqref{eq:psi_i*(psi_i+1*}.

\subsection{} \label{ssec:proofthatquadricsvanish}
So we turn to proving the statements (1--4) from Lemma~\ref{lem:complex.equivalence}. Towards proving that the quadrics vanish as wanted, we give an explicit description of the coordinates
of our points~\eqref{eq:asspoints}; this again follows~\cite[\S2]{bhargavaquinticrings}. 
Consider the matrix
\begin{equation} \label{eq:discmatrix}
D = \begin{pmatrix}
1 & 1 & \hdots & 1 \\
\alpha_1^{(1)} & \alpha_{1}^{(2)} & \hdots & \alpha_1^{(d)} \\
\vdots & \vdots & \ddots & \vdots \\
\alpha_{d-1}^{(1)} & \alpha_{d-1}^{(2)} & \hdots & \alpha_{d-1}^{(d)}
\end{pmatrix},
\end{equation}
and denote by $D_{j,i}$ the minor of $D$ corresponding to $\alpha_j^{(i)}$, i.e.\ it is $(-1)^{i+j}$ times the determinant of $D$ with the $j$-th row and $i$th column removed. Then we have 
$\alpha_j^{*(i)} = D_{j+1, i}/\det D$.

\begin{theorem}\label{thm:quadrics.vanish}
The quadrics $Q^\ell$ vanish on the points~\eqref{eq:asspoints}, i.e.\ for any $\ell$ we have
\[
\sum_{i,j = 1}^{d-1} Q^\ell_{ij}\alpha_i^*\alpha_j^*=0.
\]
\end{theorem}
\begin{proof}
 We claim that for $p=1, \ldots, d$ we have 
\begin{equation} \label{eq:vanishingidentity}
\sum_{m,n=0}^{d-1}\psi_1(\alpha_m^* \otimes \alpha_n^*)\alpha_m^{(p)}\alpha_n^{(p)} = 0.
\end{equation}
Indeed, using the definition of $\psi_1$ we can expand this as 
\[ \sum_{\substack{i, j = 1 \\ i \neq j}}^d c_{ij} \sum_{m, n = 0}^{d-1} D_{m+1, i} D_{n+1, j}\alpha_m^{(p)}\alpha_n^{(p)} \]
for certain coefficients $c_{ij} \in F$ in which we have absorbed a denominator
$\det D^2$ (which is a non-zero element of $F$, being the discriminant of $1, \alpha_1, \ldots, \alpha_{d-1}$ with respect to $K/F$).
%\[
%\alpha_m^{*(i)} \alpha_n^{*(j)} \alpha_m^{(p)}\alpha_n^{(p)} = \frac{1}{\det D^2} 
%\]
%where $i\neq j$. (Here we note that $\det D^2$ is the discriminant of , in particular it is an element of %$F^\ast$.) 
In each term, at least one of $i,j$ is not equal to $p$. Say $i \neq p$, then the corresponding summand can be written as  
\[
 c_{ij} \cdot \sum_{n = 1}^{d-1} D_{n+1, j}\alpha_n^{(p)} \cdot \sum_{m = 1}^{d-1} D_{m+1, i}\alpha_m^{(p)} = 0,
\]
and one sees that the last factor is $0$ because it is the determinant of the matrix obtained from $D$ by replacing the $i$th column with a copy of the $p$th column. The case $j \neq p$ is analogous, hence the claim follows.
%Thus
%\[
%\sum_{m,n=0}^{d-1}\psi_1(\alpha_m^*, \alpha_n^*)\alpha_m^{(p)}\alpha_n^{(p)} = 0.
%\]

Expand $\alpha_m\alpha_n$ with respect to the basis $1, \alpha_1, \ldots, \alpha_{d-1}$ and let $c_{mn}^q \in F$ denote the coordinate at $\alpha_q$. Then we may rewrite~\eqref{eq:vanishingidentity} as
\[
\sum_{m,n=0}^{d-1}\sum_{q=0}^{d-1} \psi_1(\alpha_m^* \otimes \alpha_n^*)c_{mn}^q \alpha_q^{(p)}=0.
\]
Multiplying by $\alpha_r^{*(p)}$ and summing over $p$ yields
\[
\sum_{m,n=0}^{d-1}c_{mn}^r \psi_1(\alpha_m^* \otimes \alpha_n^*)=0
\]
for all $r=0, \ldots, d-1$. Since 
$\alpha_m^*=\sum_{i=1}^{d-1} \Tr_{L/F}(\alpha_i^\ast \alpha_m^\ast) \alpha_i$ 
%$\alpha_m^*=\sum_{i=1}^{d-1} \langle \alpha_i^*, \alpha_m^*\rangle \alpha_i$ 
(here we use the fact that the $\alpha_i$ have trace zero), 
%where $\langle \cdot, \cdot \rangle$ is the trace pairing, 
we get for any $\ell$
\[
\sum_{m,n=0}^{d-1}c_{mn}^r\sum_{i,j=1}^{d-1} Q_{ij}^\ell \Tr_{L/F}(\alpha_i^\ast \alpha_m^\ast) \Tr_{L/F}(\alpha_j^\ast \alpha_n^\ast) = 0.
\]
%\[
%\sum_{m,n=0}^{d-1}c_{mn}^r\sum_{i,j=1}^{d-1} Q_{ij}^\ell\langle \alpha_i^*, \alpha_m^*\rangle \langle\alpha_j^*, \alpha_n^*\rangle = 0.
%\]
Multiplying by $\alpha_r$ and summing over $r$ finally yields the desired vanishing.
%\[
%\sum_{i,j=1}^{d-1}Q_{ij}^\ell\alpha_i^*\alpha_j^* = 0. \qedhere
%\]
\end{proof}

\subsection{} As for the other statements in Lemma~\ref{lem:complex.equivalence}, we have:

%By Theorem \ref{thm:quadrics.vanish}, the last step of the resolution is indeed a complex, but it remains to show that it is exact and minimal. This will be done in~\ref{ssec:exactness}.

\begin{lemma} \label{lem:2iscomplex}
If $d \geq 5$, then for $\alpha\in V_1$, we have that $\psi_2(\psi_1(\alpha\otimes \alpha)\otimes \alpha) = 0$.
\end{lemma}
\begin{proof}
Let $\eta = \psi_2(\psi_1(\alpha\otimes \alpha)\otimes \alpha)$. By construction of $\psi_2$, $\eta$ lies in $W_{\lambda_3}$. On the other hand, $\eta$ is also an element of the subrepresentation $U=\Span\{\alpha^{(j)}\alpha^{(k)}\alpha^{(\ell)}\mid 1\leq j,k,\ell\leq d\}$ of $L$. Now $U$ is the image of the equivariant map
\[
\Sym^3 V_{(d-1,1)} \to L: y_j\otimes y_k\otimes y_\ell \mapsto \alpha^{(j)}\alpha^{(k)}\alpha^{(\ell)}.
\]
However, by Lemma~\ref{lem: sym3}, $\Sym^3 V_{(d-1,1)}$ does not contain the irreducible representation $V_{\lambda_3}$ and so indeed $\eta = 0$.
\end{proof}

\begin{lemma} \label{lem:3iscomplex}
For $2\leq i\leq d-4, \alpha\in V_1$ and $\omega\in V_i$, we have  $\psi_{i+1}( \psi_i(\omega\otimes \alpha) \otimes \alpha)) = 0$.
\end{lemma}
\begin{proof}
This is similar to the previous lemma. Let $\eta = \psi_{i+1}( \psi_i(\omega\otimes \alpha) \otimes \alpha))$. Then $\eta$ lives in $W_{\lambda_{i+2}}$, by construction of $\psi_{i+1}$. However, $\eta$ is also contained in the subrepresentation $U = \Span \{\sigma(\omega) \alpha^{(j)}\alpha^{(k)} \mid \sigma\in S_d, 1\leq j,k\leq d\}$. The space $U$ is the image of the map
\[
V_{\lambda_i}\otimes \Sym^2 V_{(d-1,1)} \to L: \sigma(w_1^i)\otimes (y_j\otimes y_k)\mapsto \sigma(\omega)\alpha^{(j)}\alpha^{(k)}.
\]
However, by Corollary~\ref{cor: multiplicity double next}, $V_{\lambda_i}\otimes \Sym^2 V_{(d-1,1)}$ does not contain the representation $V_{\lambda_{i+2}}$, proving the desired result.
\end{proof}

\begin{lemma} \label{lem:4iscomplex}
If $d \geq 5$ then we have $\psi_{d-2}(\psi_{d-3}(\omega\otimes \alpha)\otimes (\alpha\otimes \alpha)) = 0$ for all $\alpha\in V_1$ and $\omega\in V_{d-3}$. 
\end{lemma}
\begin{proof}
The proof idea is the same as in the previous two lemmas. This time one uses that $V_{\lambda_{d-3}}\otimes \Sym^3 V_{(d-1,1)}$ does not contain the sign representation $V_{(1,\ldots , 1)}$, by Lemma~\ref{lem: sym3}.
\end{proof}

\subsection{} The case $d = 4$ is not covered by our treatment so far. We leave it to the reader to check that the requirements 
from Lemma~\ref{lem:complex.equivalence} should be replaced by
\begin{enumerate}
  \item the quadrics $Q^\ell$ vanish on the points~\eqref{eq:asspoints},
  \item we have $\psi_{2}(\psi_{1}(\alpha\otimes \alpha)\otimes (\alpha\otimes \alpha)) = 0$ for any $\alpha\in V_1$.
\end{enumerate}
The first requirement is covered by Theorem~\ref{thm:quadrics.vanish}, while the second requirement can be checked as in Lemma~\ref{lem:4iscomplex}, with $\mult(V_{(1^4)}, \Sym^4 V_{(3,1)}) = 0$ from~\eqref{eq: adhocmults} as the key representation-theoretic ingredient.

%We have now proven that our constructed resolution is indeed a chain complex. So it remains to prove that it is exact and minimal, which will be done in the next section.

%\subsection*{Exactness and minimality.}

\subsection{Exactness and minimality.} We now prove that our chain complex is indeed a minimal free resolution. 
This is equivalent with proving surjectivity of the maps $\psi_i$:

\begin{lemma}
If the linear maps 
$\psi_1  :  \Sym^2 V_1 \to V_2$, $\psi_i  :  V_i\otimes V_1\to V_{i+1}$ for $2\leq i\leq d-3$,
and $\psi_{d-2} :  V_{d-2}\otimes \Sym^2 V_1\to V_d$ 
 are all surjective, then the complex~\eqref{eq:resolution.Bhargava.points} is a minimal free resolution.
\end{lemma}
\begin{proof}
From~\cite[(4.2)]{schreyer} we know that $I$ is generated by $\beta_1$ linearly independent quadrics.
Thanks to Theorem~\ref{thm:quadrics.vanish} we also know that the $\beta_1$ quadrics $Q^\ell$ are elements of $I$. Thus they form a minimal generating set if and only if they are linearly independent. But this is equivalent to $\psi_1$ being surjective.
Assuming surjectivity of $\psi_1$, again from~\cite[(4.2)]{schreyer} we then know that the $R$-module of syzygies between these $Q^\ell$'s is generated by $\beta_2$ linearly independent linear syzygies. 
But we have just showed that $\psi_2^\ast$ produces $\beta_2$ linear syzygies. Thus they form a minimal generating set if and only if they are linearly independent. In turn, this is equivalent to the surjectivity of $\psi_2$. 
An inductive application of this argument concludes the proof.
\end{proof}

%$R/I$ has a minimal free resolution of the form
%\[
%0\to R(-d) \to R(-(d-2))^{\beta_{d-3}} \to \ldots \to R(-3)^{\beta_2} \to R(-2)^{\beta_1} \to R \to R/I\to 0.
%\]
%For the map $\psi_1$, note that surjectivity is equivalent to the quadrics $Q^\ell$ as constructed above to be linearly independent. From the above resolution, we know that the ideal $I$ is generated by $\beta_1$ quadrics. Therefore, the quadrics $Q^\ell$ form a minimal generating set for $I$ and the complex is exact at $V_2^*\otimes R(-2)^{\beta_1} \to V_0^*\otimes R \to R/I$.

\subsection{} To prove surjectivity we need the following technical lemma;
here, by the product $W \cdot W'$ of $W, W' \subseteq L$ we mean the subspace generated by all $ww'$ for $w \in W$ and $w' \in W'$.

\begin{lemma}\label{lem:containment.Wlambda}
We have $W_{\lambda_2}\subseteq W_{(d-1,1)}^2$, $W_{\lambda_{i+1}}\subseteq W_{\lambda_i}\cdot W_{(d-1, 1)}$ for $2\leq i\leq d-3$, and $W_{(1^d)} \subseteq  W_{\lambda_{d-2}}
\cdot W_{(d-1,1)}^2 $.
\end{lemma}
\begin{proof}
%Consider our basis $\alpha_0 = 1, \alpha_1, \ldots, \alpha_{d-1}$ of $K$ over $F$, and
%recall that $\alpha_1, \ldots, \alpha_{d-1}$ form a basis for $V_1$, since they were chosen to have trace zero.  
%For any $\eta\in W_{\lambda_{i+1}}$, the set $\operatorname{\sigma (\eta)}_{\sigma \in S_d} \cap L^{S_{\lambda_{i+1}}}$ has dimension one. Hence 
Denoting by $\Sym^2 W_{(d-1,1)}$ the subspace of $L$ generated by all elements of the form
$\alpha^{(m)} \beta^{(n)} + \alpha^{(n)} \beta^{(m)}$
with $\alpha, \beta \in V_1$ and $m, n \in \{1, \ldots, d\}$, we will first show that
\begin{equation} \label{eq:weirdinclusion}
 W_{\lambda_{i+1}} \subseteq W_{(d-1,1)}^{i-1} \Sym^2 W_{(d-1,1)}
\end{equation}
for all $i = 1, \ldots, d-3$. Since $\Sym^2 W_{(d-1,1)} \subseteq W^2_{(d-1,1)}$, this will settle
the first inclusion.

Because of the $S_d$-action, it is enough to prove that $V_{i+1}$ is contained in the right-hand side of~\eqref{eq:weirdinclusion}.
In turn, it suffices to prove this for
$L^{S_{\lambda_{i+1}}} \supseteq V_{i+1}$. Since this field is generated by 
$\alpha^{(d-i)}+\alpha^{(d-i+1)}$, $\alpha^{(d-i+2)}$, $\alpha^{(d-i+3)}$, \ldots, $\alpha^{(d)}$ for some primitive element $\alpha \in K$, any element is an $F$-linear combination of elements of the form
\[
(\alpha^{(d-i)} + \alpha^{(d-i+1)})^{f_1} \alpha^{(d-i+2)f_2} \cdots \alpha^{(d)f_i}.
\]
Every appearance of $\alpha^{(m)f_j}$  can be rewritten as an $F$-linear combination of the elements 
$\alpha_0^{(m)} = 1, \alpha_1^{(m)}, \ldots, \alpha_{d-1}^{(m)}$.
After doing this, we find that every element of $L^{S_{\lambda_{i+1}}}$ is an $F$-linear combination of elements of the form
\[
\left( \alpha_{k_0}^{(d-i)}\alpha_{k_1}^{(d-i+1)} + \alpha_{k_0}^{(d-i + 1)}\alpha_{k_1}^{(d-i)} \right) \alpha_{k_2}^{(d-i+2)} \cdots \alpha_{k_i}^{(d)},
\]
for $0\leq k_j\leq d-1$. This shows that $V_{i+1}$ is contained in 
\begin{equation} \label{eq:weirderinclusion} 
  (W_{(d)} + W_{(d-1,1)})^{i-1} \Sym^2 (W_{(d)} + W_{(d-1,1)}).
\end{equation}
We claim that $\Sym^2 W_{(d-1,1)}$ contains $W_{(d)}$ and $W_{(d-1,1)}$. This claim readily implies that $\Sym^2 (W_{(d)} + W_{(d-1,1)}) = \Sym^2 W_{(d-1,1)}$ and also that 
\[ (W_{(d)} + W_{(d-1,1)}) \Sym^2 W_{(d-1,1)} = W_{(d-1,1)} \Sym^2 W_{(d-1,1)}, \] so that~\eqref{eq:weirderinclusion} equals $W_{(d-1,1)}^{i-1} \Sym^2 W_{(d-1,1)}$, thus settling~\eqref{eq:weirdinclusion}.

In order to prove the claim, it suffices to show that $\Sym^2 W_{(d-1,1)}$ contains $K = V_0 + V_1$. Pick a
non-zero element $\beta_1 \in K \cap \Sym^2 W_{(d-1,1)}$ that is not contained in $F$; e.g.\ one of $\alpha_1^2, \alpha_1\alpha_2$ will do. By replacing it with 
\[ \beta_1 - \frac{1}{d!} \Tr_{L/F}(\beta_1) = \beta_1 - \frac{1}{d} \sum_{j=1}^d \beta_1^{(j)} \]
if needed, we can assume that it has trace zero, so that it belongs to $V_1$. Extending this to a basis $\beta_1, \ldots, \beta_{d-1}$ of $V_1$, it is easy to see that $\beta_1^2, \beta_1\beta_2, \ldots, \beta_1 \beta_{d-1}, \beta_1 \in K \cap \Sym^2 W_{(d-1,1)}$ are linearly independent over $F$, hence they must generate $K$, as wanted. 

Having settled the case $i=1$, we now assume $2 \leq i \leq d-3$. Decomposing 
\[ V_{(d-1,1)}^{\otimes i-2} \otimes \Sym^2 V_{(d-1,1)} = \bigoplus_{\lambda} V_\lambda^{c_{i, \lambda}}  \] 
we see from~\eqref{eq:weirdinclusion} that
\[
W_{\lambda_{i+1}}\subseteq W_{(d-1,1)}^{i-1} \Sym^2 W_{(d-1,1)} \subseteq \sum_{\substack{\lambda \text{ with} \\ c_{i,\lambda} > 0}} W_\lambda \cdot W_{(d-1,1)}.
\]
By decomposing $\Sym^2 V_{(d-1,1)}$ as in~\eqref{eq: decompfirststep} and using Lemma~\ref{lem: tensorstandard} and Lemma~\ref{lem: tensornotsostandard}, one sees that $c_{i,\lambda} >0$ only if the Young diagram corresponding to $\lambda$ has at most $i$ boxes not in its first row and $\lambda \neq (d-i, 1^i)$. Consequently, again using Lemma~\ref{lem: tensorstandard}, we see that the only partition $\lambda$ with $c_{i,\lambda}>0$ such that $V_\lambda \otimes V_{(d-1,1)}$ contains $V_{\lambda_{i+1}}$ is $\lambda=\lambda_i$. We conclude that 
\[
W_{\lambda_{i+1}} \subseteq W_{\lambda_i} \cdot W_{(d-1,1)},
\]
as wanted.

For the last inclusion we give an ad-hoc proof. Recall that $W_{(1^d)}$ is generated by
the Vandermonde determinant
\[ \left| \begin{array}{cccc} 1 & \alpha^{(1)} &  \hdots & \alpha^{(1)d-1} \\ 
%1 & \alpha^{(2)} &  \hdots & \alpha^{(2)d-1} \\
\vdots & \vdots & \ddots & \vdots \\
1 & \alpha^{(d)} & \hdots & \alpha^{(d)d-1} \\
\end{array} \right| \]
for some primitive element $\alpha \in K$.
Expanding this determinant with respect to the last two columns yields a linear combination of
\begin{equation*} %\label{eq: expandeddisc} 
 \left| \begin{array}{cccc} 1 & \alpha^{(1)} &  \hdots & \alpha^{(1)d-3} \\ 
%1 & \alpha^{(2)} &  \hdots & \alpha^{(2)d-3} \\
\vdots & \vdots & \ddots & \vdots \\
1 & \alpha^{(d-2)} & \hdots & \alpha^{(d-2)d-3} \\
\end{array} \right| \cdot (\alpha^{(d)} - \alpha^{(d-1)}) \cdot (\alpha^{(d-1)} \alpha^{(d)})^{d-1}
\end{equation*}
and conjugates thereof. By~\cite[Prob.\,4.47]{fultonharris} the product of the first two factors is in the image of the Specht module $V_{\lambda_{d-2}}$ under the natural map $F[z_1, \ldots, z_d] \to L : z_i \mapsto \alpha^{(i)}$, so it is in $W_{\lambda_{d-2}}$. 
By mimicking the argumentation from the start of this proof, the last factor is seen to be in 
\[ (W_{(d)} + W_{(d-1, 1)})^2 = W_{(d-1,1)}^2, \] where the equality holds because $W_{(d-1, 1)}^2 \supseteq \Sym^2 W_{(d-1,1)}$ contains $W_{(d)}$ and $W_{(d-1,1)}$.
\end{proof}

\subsection{} We are now ready to conclude:
\begin{lemma}
The maps $\psi_i$ are all surjective, so that the complex~\eqref{eq:resolution.Bhargava.points} is a minimal free resolution.
\end{lemma}
\begin{proof}
%Throughout the proof, we use the bases as provided by Lemma \ref{lem:psi.nonzero}.
We first prove the surjectivity of $\psi_1 : \Sym^2 V_1 \to V_2$. 
By Lemma \ref{lem:containment.Wlambda} any $\gamma \in V_2$ can be written as
\[
\gamma = \sum_{\substack{1\leq m,n\leq d \\ 1\leq j\leq k \leq d-1 }} c_{jk}^{mn} \alpha_j^{(m)} \alpha_k^{(n)}
\]
for $c_{jk}^{mn}\in F$.
For every pair $j,k$ define
$D_{jk} = \sum_{1\leq m,n\leq d} c_{jk}^{mn} \alpha_j^{(m)} \alpha_k^{(n)}$.
We show that we may take the individual $D_{jk}$'s to be in $V_2$ as well. For this, let $\rho\in Z(F[S_d])$ be the element obtained by applying Lemma~\ref{lem: projecting to a single W} to the partition $(d-2,2)$. Now consider for every $j,k$ the element
\[
D_{jk}' = \frac{1}{2(d-2)!}\sum_{\tau\in S_2\times S_{d-2}} \tau(\rho(D_{jk})).
\]
We still have $\gamma = \sum_{1 \leq j \leq k \leq d-1} D'_{jk}$ with $D'_{jk} \in \Span \{ \alpha_j^{(m)}\alpha_k^{(n)} \mid 1\leq m,n\leq d \}$, and each $D'_{jk}$ is contained in $L^{S_2\times S_{d-2}} \cap W_{(d-2,2)} = V_2$, as wanted. 

We now prove that every $D_{jk}'$ is in the image of $\psi_1$. If $D_{jk}'=0$, there is nothing to prove. So we assume that $D_{jk}'\neq 0$. Consider the equivariant map
\[
\phi_{jk}: V_{(d-1,1)}\otimes V_{(d-1,1)} \to L: y_m\otimes y_n \mapsto \alpha_j^{(m)}\otimes \alpha_k^{(n)}
\]
whose image is precisely $\Span \{ \alpha_j^{(m)}\alpha_k^{(n)} \mid 1\leq m,n\leq d \}$. From Lemma~\ref{lem: tensorstandard} we see that $V_{(d-1,1)}\otimes V_{(d-1,1)}$ contains a unique copy of $V_{(d-2,2)}$, which contains the element $p^1$ from~\ref{ssec:firststep}. 
The image of $\phi_{jk}$ also contains $V_{(d-2,2)}$ as a subrepresentation because $D_{jk}'\neq 0$. So Schur's lemma implies that $\phi_{jk}$ cannot map $p^1$ to zero. Since both $\phi_{jk}(p^1)$
and $D_{jk}'$ are fixed under the action of $S_2 \times S_{d-2}$, we find from Lemma~\ref{lem: multiplicity fixed} that $D_{jk}' = d_{jk} \phi_{jk}(p^1) = \psi_1( d_{jk} \alpha_j\otimes \alpha_k)$ for some $d_{jk}\in F^\times$. But then
\[ \gamma = \psi_1 \left( \sum_{1 \leq i \leq j \leq d-1} d_{jk} \alpha_j \otimes \alpha_k \right), \]
i.e.\ $\psi_1$ is surjective.

For the other maps, the argument is similar. Let $2\leq i\leq d-3$ and take $\gamma \in V_{i+1}$. We wish to prove that $\gamma$ is in the image of $\psi_i : V_i \otimes V_1 \to V_{i+1}$. By Lemma~\ref{lem:containment.Wlambda} we may write $\gamma$ as 
\[
\gamma = \sum_{\substack{1\leq j\leq \beta_{i-1} \\ 1\leq k\leq d-1 \\ \sigma\in S_d, 1\leq n\leq d}} c_{jk}^{\sigma,n} \sigma(\omega_j^i)\alpha_k^{(n)},
\]
for $c_{jk}^{\sigma, n}\in F$. For every pair $j,k$ define
\[
D_{jk} = \sum_{\substack{\sigma\in S_d \\ 1\leq n\leq d}} c_{jk}^{\sigma,n} \sigma(\omega_j^i)\alpha_k^{(n)}.
\]
Let $\rho\in Z(F[S_d])$ come from Lemma~\ref{lem: projecting to a single W} applied to $\lambda_{i+1}$ and define
\[
D_{jk}' = \frac{1}{2(d-i-1)!}\sum_{\tau\in S_{\lambda_{i+1}} } \tau(\rho(D_{jk})).
\]
We still have $\gamma = \sum_{1 \leq j \leq \beta_{i-1}, 1 \leq k \leq d-1} D'_{jk}$ with $D'_{jk} \in \Span \{ \sigma(\omega_j^i)\alpha_k^{(n)} \mid \sigma\in S_d, 1\leq n\leq d \}$, and each $D'_{jk}$ is contained in $L^{S_{\lambda_{i+1}}} \cap W_{\lambda_{i+1}} = V_{i+1}$. 
We prove that every $D_{jk}'$ is in the image of $\psi_i$, from which the surjectivity of $\psi_i$ follows. We may assume that $D_{jk}'\neq 0$, for else there is nothing to prove. Consider the equivariant map
\[
\phi_{jk}: V_{\lambda_i}\otimes V_{(d-1,1)} \to L: \sigma(w^i_1)\otimes y_n \mapsto \sigma(\omega_j^i)\alpha_k^{(n)}.
\]
The image of this map is precisely $\Span\{\sigma(\omega_j^i) \alpha_k^{(n)} \mid \sigma\in S_d, 1\leq n\leq d\}$. Lemma \ref{lem: tensorstandard}, Schur's lemma and the fact that $D_{jk}'\neq 0$ yields as before that $D_{jk}' = d_{jk}\phi_{jk}(p^i)  = 
\psi_1 (d_{jk} \omega_j^i \otimes \alpha_k)$
for some $d_{jk}\in F^\times$, as wanted.

The surjectivity of the last map $\psi_{d-2}$ can be proved similarly, the main representation-theoretic ingredient now being that $V_{\lambda_{d-2}} \otimes \Sym^2 V_{(d-1,1)}$ contains a unique component $V_{(1^d)}$ by Corollary~\ref{cor: multiplicity double next}. We
omit further details.
\end{proof}

\subsection{Three points in $\PP^1$} \label{ssec:threepoints}
The remainder of Section~\ref{sec:minfreerep} discusses two stand-alone observations, which can be skipped by the reader eager to move forward.
First, we consider the ideal $I$ of the $3$ points in $\PP^1$ associated with a basis $1, \alpha_1, \alpha_2$ of a cubic $S_3$-extension $K/F$ as in~\ref{ssec:veryintro}, where we assume $\Tr_{L/F}(\alpha_i) = 0$ for $i = 1, 2$. This ideal is generated by one cubic polynomial in $R$ rather than
by quadrics, hence the resolution takes the form
\[
0\to R(-3)\to R \to R/I\to 0.
\]
This exceptional behaviour also has a representation-theoretic explanation: for $d = 3$ the decomposition~\eqref{eq: decompfirststep} fails, instead we have $\Sym^2 V_{(2,1)} \cong V_{(3)} \oplus V_{(2,1)}$, so we cannot construct the quadratic map from~\ref{ssec:firststep}. 
Here, the correct representation to look at is $\Sym^3 V_{(2,1)}$, which has a unique
component $V_{(1^3)}$ in view of~\eqref{eq: adhocmults}. Let 
\[
 p = \sum_{j,k,\ell = 1}^2 p_{jk\ell} y_j\otimes y_k\otimes y_\ell \in \Sym^3 V_{(2,1)}
\]
be a generator of this component, expanded with respect to some basis $y_1, y_2$ of $V_{(2,1)}$, this defines the cubic map
\[
\psi: V_1 \to V_3: \alpha \mapsto p(\alpha^{(1)}, \alpha^{(2)}, \alpha^{(3)})
\]
from $V_1 = W_{(2,1)} \cap L^{S_2}$ to $V_3 = W_{(1,1,1)}$,
which is easy to make explicit: 
%$\psi$ simply maps an element $\alpha$ to the discriminant of , i.e.\ for $\alpha\in V_1$
\[
\psi(\alpha) = (\alpha^{(1)} - \alpha^{(2)})(\alpha^{(1)} - \alpha^{(3)})(\alpha^{(2)} - \alpha^{(3)}).
\]
This naturally induces a linear map $\Sym^3 V_1\to V_3$, mapping
$\alpha_j\otimes \alpha_k \otimes \alpha_\ell$ to $C_{jk\ell}\delta$, for some 
$C_{jk\ell} \in F$
and some fixed non-zero $\delta \in V_3$, corresponding to a cubic form 
\[
C(x_1, x_2) = \sum_{j,k,\ell = 1}^2 C_{jk\ell}x_j x_k x_\ell
\]
which again can be seen to vanish on the points~\eqref{eq:asspoints}. The resulting complex
$0\to V_3^*\otimes R(-3) \to V_0^*\otimes R \to R/I\to 0$
is a minimal free resolution of $R/I$.
We note that the cubic form $C$ can also be obtained by applying the Delone--Faddeev parametrization~\cite[Prop.\,4.2]{gangrosssavin} to the cubic extension $K/F$.

\subsection{Self-duality} \label{ssec:duality}
As a second side trip, we note that there is a ``dual" repre\-sentation-theoretic construction of a minimal free resolution
of our coordinate ring $R/I$, in which the spaces $V_0^\ast$, $V_2^\ast$, $V_3^\ast$, \ldots, $V_{d-2}^\ast$, $V_d^\ast$
appear in opposite order. This construction links the well-known self-duality for minimal free resolutions of $d$ general points in $\PP^{d-2}$~\cite{behnke,schreyer} to the duality for representations of $S_d$, obtained by
 tensoring with the sign representation $V_{\lambda_d} = V_{(1^d)}$ (i.e.\ by transposing Young diagrams). %We briefly explain how this works.

\subsection{} For the first step, tensoring the decomposition of $V_{(d-2,2)}$ from~\eqref{eq: decompfirststep} with the sign representation yields  
\[ V_{(1^d)} \otimes \Sym^2 V_{(d-1,1)} \cong V_{(1^d)} \oplus V_{(2, 1^{d-2})} \oplus V_{(2^2, 1^{d-4})}, \] 
so there is a unique component $V_{\lambda_{d-2}}$. Moreover, Lemma \ref{lem: multiplicity fixed} shows that $V_{\lambda_{d-2}}^{S_{\lambda_{d-2}}}$ has dimension $1$. Denote by $u$ a generator of $V_{(1,\ldots, 1)}$. Let 
\[
q^1 = \sum_{m,n = 1}^{d-1} q_{mn}^1 u\otimes (y_m \otimes y_n)\in V_{(1,\ldots, 1)}\otimes \Sym^2 V_{(d-1,1)}
\]
be the unique element (up to multiplication by a scalar from $F^\times$) which is fixed under $S_{\lambda_{d-2}}$ and such that its conjugates generate the representation $V_{\lambda_{d-2}}$. 
We can assume that $q_{mn}^1 = q_{nm}^1$ and use this element to define a quadratic map
\[
\varphi_1: V_{d}\otimes V_1 \to V_{d-2}: \delta \otimes \alpha\mapsto q^1(\delta, \alpha^{(1)}, \ldots, \alpha^{(d)}),
\]
which naturally induces a linear map $V_d\otimes \Sym^2 V_1\to V_{d-2}$, also denoted by $\varphi_1$. Dualizing yields the first step of the resolution for $I$
\[
V_{d-2}^*\otimes R(-2) \xrightarrow{\varphi_1^*} V_d^*\otimes R\to R/I\to 0.
\]

The subsequent steps are similar. Let $2\leq i\leq d-3$, then from Lemma~\ref{lem: tensorstandard} 
we see that $V_{\lambda_{d-i-1}}$ appears once in $V_{\lambda_{d-i}}\otimes V_{(d-1,1)}$. 
By Lemma~\ref{lem: multiplicity fixed}
%Since $V_{\lambda_{d-i-1}}^{S_{\lambda_{d-i-1}}}$ has dimension 1, 
there is, up to scalar multiplication, a unique
\[
q^i = \sum_{m = 1}^{\beta_{d-i-1}} \sum_{n = 1}^{d-1} q_{mn}^i w_m^{d-i} \otimes y_n \in V_{\lambda_{d-i}}\otimes V_{(d-1,1)},
\]
that is fixed by $S_{\lambda_{d-i-1}}$ and whose conjugates generate a $V_{\lambda_{d-i-1}}$.\footnote{Notice that for $3 \leq i \leq d - 3$ the elements $p^{d-i}$ from~\ref{ssec:nextsteps} are also taken from $V_{\lambda_{d-i}} \otimes V_{(d-1,1)}$, but they are fixed under $S_{\lambda_{d-i + 1}}$ and generate a $V_{\lambda_{d- i +1}}$, rather than being fixed under $S_{\lambda_{d-i-1}}$ and generating a $V_{\lambda_{d - i - 1}}$.} In the same way as before, this induces a linear map
\[
\varphi_i: V_{d-i}\otimes V_1\to V_{d-i-1}
\]
which, when dualized, yields a next step in the resolution. 

To conclude, using Lemma~\ref{lem: tensorstandard} and Lemma~\ref{lem: tensornotsostandard} along with~\eqref{eq: decompfirststep} one finds a unique trivial component in the representation $V_{(d-2,2)} \otimes \Sym^2 V_{(d-1,1)}$, resulting in a linear map $\varphi_{d-2}: V_2\otimes \Sym^2 V_1\to V_0$ whose dual is the last step of the resolution. 

%Note that, compared to TODO(ref: eq to original resolution), the irreducible representations corresponding to the terms in this resolution appear in the opposite order.

\subsection{} \label{ssec:lastofresolution} The resulting sequence%\todo{this part could use some extra details}
\begin{multline*}
0\to V_0^*\otimes R(-d)\xrightarrow{\varphi_{d-2}^*}  V_2^*\otimes R(-d+2) \xrightarrow{\varphi_{d-3}^*} \ldots  \\
  \ldots \xrightarrow{\varphi_3^*} V_{d-3}^*\otimes R(-3) \xrightarrow{\varphi_2^*} V_{d-2}^* \otimes R(-2) \xrightarrow{\varphi_1^*} V_d^*\otimes R\to R/I \to 0
\end{multline*}
is indeed a minimal graded free resolution because 
it is explicitly isomorphic to our sequence from~\ref{ssec:shapeofresolution}. 
We sketch this for steps $2 \leq i \leq d-3$; the analysis for $i = 1$ and $i=d-2$ is completely similar. 
Choose elements $c_\sigma^i \in F$, where $\sigma$ ranges over $S_d$, with the following property: if $w \in V_{\lambda_i}$ is fixed by $S_{\lambda_i}$, then $\sum_{\sigma \in S_d} c^i_\sigma \sigma(u \otimes w)$ is non-zero and fixed under $S_{\lambda_{d-i}}$. Note that such coefficients indeed exist in view of Lemma~\ref{lem: multiplicity fixed}, because the $\sigma(u \otimes w)$'s span a $V_{(1, \ldots, 1)} \otimes V_{\lambda_i} \cong V_{\lambda_{d-i}}$. 
It follows, for each $m = 1, \ldots, \beta_{i-1} = \beta_{d-i-1}$, 
that $\sum_{\sigma\in S_d} c_\sigma^i \sigma(\delta \omega_m^{i}) \in W_{\lambda_{d-i}}$ is fixed under $S_{\lambda_{d-i}}$ and non-zero; consequently 
\[ \varpi^i_m = \sum_{\sigma\in S_d} c_\sigma^i \sigma(\delta \omega_m^{i}), \quad m = 1, \ldots, \beta_{d-i-1} \]
is a basis for $V_{d-i}$. We leave it to the reader to verify that the map $\psi_i$ expressed with respect to the bases $\omega_m^i$ and $\omega_m^{i+1}$ coincides with the map $\varphi_i$ expressed with respect to the bases $\varpi_m^i$ and $\varpi_m^{i+1}$.

%We work relative to this basis in the dual resolution and relate the maps $\varphi_i$ and $\psi_i$. 
 
%To this end, let $p^i \in V_{\lambda_i} \otimes V_{(d-1,1)}$ be as in~\ref{ssec:nextsteps} and recall that it is fixed under $S_{\lambda_{i+1}}$ and that its conjugates generate a representation isomorphic to $V_{\lambda_{i+1}}$, therefore
%\[
%\sum_{\sigma \in S_d} c_\sigma^{i+1} \sigma(u \otimes p^i) = u\otimes \sum_{\sigma} \sgn (\sigma) c_\sigma^{i+1} \sigma(p^i) \in (V_{(1, \ldots, 1)} \otimes V_{\lambda_i}) \otimes V_{(d-1,1)}
%\]
%is a valid instance of $q^i$.
%Writing 
%\[ \psi_i(\omega_\ell^i\otimes \alpha_j) = \sum_{m=1}^{\beta_i} L_{\ell j}^m \omega_m^{i+1} \] for certain $L_{\ell j}^m\in F$, we find by direct computation that $\varphi_i (\varpi^i_\ell \otimes \alpha_j)$ equals
%\[
% \varphi_i\left( \sum_{\sigma \in S_d} c_\sigma^i \sigma(\delta \omega_\ell^i)\otimes \alpha_j\right) = \sum_{m = 1}^{\beta_i} L_{\ell j}^m \left( \sum_{\sigma \in S_d} c_\sigma^{i+1} \sigma(\delta \omega_m^{i+1}) \right) = \sum_{m = 1}^{\beta_i} L_{\ell j}^m \varpi_m^{i+1},
%\]
%as wanted. 
%Thus, with respect to our chosen bases, the maps in the new resolution are exactly the same as the maps in the original resolution.

\section{Scrollar invariants of representations and resolvents}\label{sec:scrollar.invs}

%In this section we introduce our generalized scrollar invariants associated to irreducible representations of the symmetric group, and prove various results about them. In particular we prove Proposition~\ref{prop:scrollar.invs.of.hooks} on the scrollar invariants associated to hook partitions as well as Theorem~\ref{thm:scrollar.invariants.resolvent} about the scrollar invariants of resolvent covers. %In Proposition~\ref{prop:scrollar.invs.duality} we prove a duality statement, expressing the scrollar invariants of a partition in terms of those of its dual partition. 

%Recall that $\varphi: C\to \PP^1$ is a degree $d$ morphism over a field of characteristic zero or larger than $d$, where $C$ is a smooth projective genus $g$ curve.

\subsection{} We now return to our simply branched degree $d$ cover $\varphi : C \to \PP^1$ over a field $k$ satisfying $\charac k = 0$ or $\charac k > d$, with $C$
a curve of genus $g$. In this section we 
explain the decomposition~\eqref{eq:multiscrollarVB} that underlies Definition~\ref{def_scrollar_lambda}, introducing the scrollar invariants of $C$ with respect to a partition
$\lambda \vdash d$: this is done in~\ref{ssec:defscrollpart}. We then proceed with studying some of their first properties.
Recall that this notion generalizes that of the scrollar invariants $e_1, e_2, \ldots, e_{d-1}$ of $C$ with respect to $\varphi$, which arise as the scrollar invariants with respect to the partition $(d-1,1)$. Several other examples are discussed, which are all subsumed by Proposition~\ref{prop:scrollar.invs.of.hooks} on the scrollar invariants of hooks: its proof can be found in~\ref{ssec:scrollar.invs.of.hooks}.
In~\ref{ssec:volanddual} we prove our ``volume formula" generalizing the sum formula $e_1 + e_2 + \ldots + e_{d-1} = g + d - 1$, and we also present an explicit formula relating the scrollar invariants of $\lambda$ to those of the dual partition $\lambda^\ast$. 
The main result of this section is Theorem~\ref{thm:scrollar.invariants.resolvent}, expressing the scrollar invariants of resolvent covers in terms of scrollar invariants with respect to certain partitions: a proof can be found in~\ref{ssec:resolvent.curves}. 
Our treatment is based on explicit function field arithmetic, in the style of Hess~\cite{hessRR}; alternatively, it should be possible to develop much of the following material using the formalism of parabolic vector bundles and the Mehta--Seshadri correspondence~\cite{MS} mentioned in~\ref{ssec:MS}.

%Let us point out that one may alternatively define the scrollar invariants by appealing to the Metha--Sesahdri correspondence again~\cite{MS}. With this perspective, one can give alternative proofs of many of the following proofs using the theory of parabolic vector bundles.

\subsection{Reduced bases.} \label{ssec:reducedbasisdef}
 Consider the extension
$k(t) = k(\PP^1) \subseteq k(C) = K$
induced by $\varphi$. The simple branching assumption will not play an important role until~\ref{ssec:genusformula}. Following Hess, see~\cite{hessRR} and~\cite[pp.\,43-52]{hessnotes}, we can interpret the decomposition~\eqref{pushfwd} in terms of this extension. 
%This viewpoint will be useful in defining the scrollar invariants for representations of $\Gal(L / k(t))$, where $L$ denotes the Galois closure of $K / k(t)$.
Denote by $\OO_K$ resp.\ $\OO_{K}^\infty$ the integral closure of $k[t]$ resp.\ $k[t^{-1}]$ in $K$. Geometrically, these rings correspond to two affine patches of the curve $C$, one above $\AA^1 = \PP^1 \setminus \{ \infty \}$ and the other above $\PP^1\setminus \{0\}$.

\begin{theorem}\label{thm:Hess.reducedbasis}
There exists a $k[t]$-basis $1, \alpha_1, \ldots, \alpha_{d-1}$ of $\OO_K$, together with unique integers $1\leq r_1\leq \ldots \leq r_{d-1}$ such that $1, t^{-r_1}\alpha_1, \ldots, t^{-r_{d-1}}\alpha_{d-1}$ is a $k[t^{-1}]$-basis of $\OO_{K}^\infty$.
\end{theorem}

\begin{proof}
This follows from~\cite[Cor.\,4.3]{hessRR}; see also~\cite[p.\,46]{hessnotes}.
\end{proof}

\noindent The integers $r_i$
describe how $\varphi_\ast \mathcal{O}_C$ splits, that is, $r_i = e_i$ for $i=1, \ldots, d-1$.
An accompanying basis $1, \alpha_1, \ldots, \alpha_{d-1}$ as above is called a ``reduced basis". 
%\todo[inline]{Mention something about computing reduced bases? (via eg magma)}
%because $\alpha_0 \in (k[C]_0 \cap k[C]_\infty) \setminus \{0\} = k \setminus \{0\}$ it can be assumed that $\alpha_0 = 1$.

\subsection{} \label{ssec:minkowski_analogy} As explained in~\cite[\S7]{hessRR} and already touched upon in~\ref{ssec:introdefscrollar}, the notion of a reduced basis is the function-field theoretic analogue of that of a Minkowski-reduced $\ZZ$-basis $1, v_1, \ldots, v_{d-1}$ of the ring of integers of a degree $d$ number field $E$, and under this correspondence the scrollar invariants $e_i$ can be seen as the equivalents of $\log \, \lVert v_i \rVert$, with $\lVert \cdot \rVert$ denoting the $\ell_2$-norm under the Minkowski embedding. Given the analogy between $g$ and $\log \sqrt{|\Delta_E|}$, where $\Delta_E$ is the discriminant of $E$, it is interesting to view the property $e_1 + e_2 + \ldots + e_{d-1} = g + d - 1$ against Minkowski's second theorem
\[ \lVert v_1 \rVert \cdot \lVert v_2 \rVert \cdots \lVert v_{d-1} \rVert \sim_d \sqrt{|\Delta_E|}, \]
and to view the Maroni bound $e_{d-1} \leq (2g+2d-2)/d$ against the observation
$\lVert v_{d-1} \rVert = O_d(|\Delta_E|^{1/d})$
due to Peikert--Rosen~\cite[Lem.\,5.4]{peikertrosen}, see also~\cite[Thm.\,1.6]{2torsionclassgroup}.

\subsection{}
For the purposes of this paper, it is convenient to assume that $\Gal(L/k(t))$ is the full symmetric group $S_d$, although the observations below
apply more generally. 
Let $\OO_{L}$ resp.\ $\OO_{L}^\infty$ be the integral closure of $k[t]$ resp.\ $k[t^{-1}]$ in $L$. For any $k(t)$-vector space $V \subseteq L$, say of dimension $n$, we define 
\[ \OO_V = \OO_L\cap V, \qquad \OO_{V}^\infty = \OO_{L}^\infty\cap V. \] 
Note that $\OO_V$ is a free $k[t]$-submodule of $L$ of rank $n$, and similarly for $\OO_{V}^\infty$.
%One can think of $\OO_V$ as a lattice of rank $\dim V$ inside $\OO_L\subset L$.
By \cite{hessRR} there are unique integers $0\leq r_1\leq \ldots \leq r_n$, together with a $k[t]$-basis $v_1, \ldots, v_n$ of $\OO_V$ such that
\[
t^{-r_1}v_1, \ldots, t^{-r_n}v_n
\]
form a $k[t^{-1}]$-basis of $\OO_{V}^\infty$. We call the $r_i$ the ``scrollar invariants" of $V$ and an accompanying basis $v_1, \ldots, v_n$ is called a ``reduced basis". The scrollar invariant $0$ is realized if and only if $v_1 \in k$. Note that the scrollar invariants of $K$ are just $\{ 0 \} \cup \{ e_1, \ldots, e_{d-1} \}$, where $e_i$ are the scrollar invariants of $C$ with respect to $\varphi$.

We prove two auxiliary lemmas:

\begin{lemma} \label{lem:scrollarsofV}
Let $V \subseteq L$ be a $k(t)$-subspace with scrollar invariants $r_1, \ldots, r_n$ and reduced basis $v_1, \ldots, v_n$. For every integer $r$ define
\[ V^r = \{ v \in \OO_V \mid t^{-r } v \in \OO_{V}^\infty \} \]
and let $j_r$ be maximal such that $r_{j_r} \leq r$.
Then $V^r$ is a $k$-vector space with basis 
$\{t^i v_j \}_{j=1, \ldots, j_r, i = 0, \ldots r - r_j}$, and
$k(t) V^r$ is a $k(t)$-vector space with basis $\{ v_1, \ldots, v_{j_r} \}$.
\end{lemma}
\begin{proof}
It is clear that $\{ t^i v_j \}_{j = 1, \ldots, n, i = 0, 1, 2, \ldots}$ is a $k$-basis for $\OO_V$. When expanding $v\in \OO_V$ with respect to this basis, the requirement $t^{-r}v\in \OO_{V}^\infty$ is easily seen to be equivalent to the vanishing of the coordinates at $t^iv_j$ with $i + j_r - r > 0$, from which the first claim follows. The second claim is immediate from the first one.
\end{proof}

The next lemma is frequently useful in proving that a candidate-reduced basis is indeed a reduced basis. 

\begin{lemma}\label{lem:reducedbasis.guess}
Let $V$ be a $k(t)$-vector subspace of $L$ with scrollar invariants $r_1\leq \ldots \leq r_n$. Suppose that $v_1', \ldots, v_n'\in \OO_V$ form a $k(t)$-basis for $V$ and that there are integers $r'_1\leq \ldots \leq r'_n$ such that 
\[
r_1 + \ldots + r_n = r_1' + \ldots + r_n'
\]
and such that $t^{-r_i'}v_i'\in \OO_{V}^\infty$. Then 
$r_i = r_i'$ for all $i$, and $v_1', \ldots, v_n'$ is a reduced basis for $V$. 
\end{lemma}
\begin{proof}
Let $v_1, \ldots, v_n$ be a reduced basis of $V$. Define the matrices
\begin{align*}
D_1 &= \diag(t^{r_1}, \ldots, t^{r_{n}}),\\
D_2 &= \diag(t^{r_1'}, \ldots, t^{r_{n}'}).
\end{align*}
Note that $\det D_1 = \det D_2 = t^{ r_1+\ldots + r_n }$. Let $B$ be the change of basis matrix from $\{v_i'\}_i$ to $\{v_i\}_i$. Since the $v_i'$ are integral over $k[t]$, this matrix has entries in $k[t]$. We also have $\det B \neq 0$ since the $v_i'$ form a $k(t)$-basis of $V$. We make a similar reasoning above infinity. The change of basis matrix from $\{t^{-r_i'}v_i'\}_i$ to $\{t^{-r_i}v_i\}_i$ is given by $D_2 B D_1^{-1}$. It has entries in $k[t^{-1}]$ since $t^{-r_i'}v_i'$ is integral over $k[t^{-1}]$. It follows that 
$\det (D_2BD_1^{-1}) =  \det B \in k[t]\cap k[t^{-1}] = k$.
Since $\det B$ is non-zero, we conclude that $\{v_i'\}_i$ is a reduced basis for $V$.
\end{proof}

\subsection{Scrollar invariants of representations and partitions.} \label{ssec:defscrollpart} We now introduce our scrollar invariants associated to irreducible representations. Recall that we can view $L$ as the regular representation of $S_d$.
We split $L$ into isotypic components $W_\lambda \cong 
V_\lambda^{\dim V_\lambda}$ as in~\ref{ssec:informalmainresult}, one for each partition $\lambda \vdash d$. The patches $\mathcal{O}_{W_{\lambda}}$ and $\mathcal{O}_{W_{\lambda}}^\infty$ then glue together to the vector bundle $\mathcal{W}_\lambda$ from Section~\ref{ssec:informalmainresult}. 
The next corollaries to Lemma~\ref{lem:scrollarsofV} establish the decomposition~\eqref{eq:multiscrollarVB} and lie at the heart of everything
that follows:

\begin{corollary}
If $V\subseteq L$ is an irreducible subrepresentation then all scrollar invariants of $V$ are equal to each other.
\end{corollary}
\begin{proof}
Let $v_1, \ldots, v_n$ be a reduced basis for $V$ and let $r_1\leq \ldots \leq r_n$ be the corresponding scrollar invariants. The $k$-vector space $V^{r_1}$ from Lemma~\ref{lem:scrollarsofV} has dimension $j_{r_1}$. But it contains all the conjugates of $v_1$, and by the irreducibility of $V$ we can find $n$ conjugates which are linearly independent over $k(t)$, so definitely over $k$. But then we must have $j_{r_1} = n$, i.e.\ $r_1 = r_2 = \ldots = r_n$.
\end{proof}

\begin{corollary} \label{cor:scrollar.invariants.Wlambda}
The scrollar invariants of $W_\lambda\subseteq L$ form a multi-set
\[ 
e_{\lambda,1}, e_{\lambda,1}, \ldots, e_{\lambda,1}, \quad e_{\lambda,2}, e_{\lambda,2}, \ldots, e_{\lambda,2}, \quad \ldots  \quad e_{\lambda, \dim V_\lambda}, e_{\lambda, \dim V_\lambda}, \ldots, e_{\lambda, \dim V_\lambda}
\]
with every block containing $\dim V_{\lambda}$ copies of the same entry.
\end{corollary}
\begin{proof} 
Write $n = \dim V_{\lambda}$, 
let 
$v_{11}, \ldots, v_{1n}, v_{21}, \ldots, v_{2n},\ldots, v_{n1}, \ldots, v_{nn}$
 be a reduced basis of $W_{\lambda}$, and let 
\[ 0 \leq r_{11}\leq \ldots \leq r_{1n} \leq r_{21} \leq \ldots \leq r_{2n} \leq  \ldots \leq r_{n1} \leq \ldots \leq r_{nn} \]
be the corresponding scrollar invariants. Using the notation from Lemma~\ref{lem:scrollarsofV}, we know that $\dim k(t) W_\lambda^{-3} = 0$ and $\dim k(t) W_\lambda^{r_{nn}} = n^2$. By Lemma~\ref{lem:scrollarsofV} it suffices to show that whenever
\begin{equation} \label{eq:ktWstrict}
   k(t)W_\lambda^{r-1} \subsetneq k(t) W_\lambda^r
\end{equation}
for some $r = 0, \ldots, r_{nn}$, the dimensions differ by a multiple of $n$. By assumption there exists a basis element $v_{ij}$ with
scrollar invariant exactly $r$. The conjugates of $v_{ij}$ span a space $V_{ij} \cong V_\lambda$ of dimension $n$ which is contained in $k(t) W_{\lambda}^r$ and which intersects $k(t)W_{\lambda}^{r-1}$ trivially. We can repeat this argument with $k(t)W_{\lambda}^{r-1}$ replaced by $k(t)W_{\lambda}^{r-1} + V_{ij}$ until~\eqref{eq:ktWstrict} becomes an equality, at which moment the dimension has indeed increased with a multiple of $n$. 
\end{proof}

\subsection{Resolvent curves and their scrollar invariants.}\label{ssec:resolvent.curves}
For every subgroup $H \subseteq S_d$ we call the fixed field $L^H$ the ``resolvent" of $K$ with respect to $H$. Geometrically, this corresponds to a curve, denoted by $\res_H C$, equipped with a morphism $\res_H \varphi : \res_H C \to \PP^1$ of degree $[S_d:H]$. Its isomorphism class as a $\PP^1$-cover is only dependent on the conjugacy class of $H$. We now prove Theorem~\ref{thm:scrollar.invariants.resolvent}, 
expressing the scrollar invariants of $L^H$ in terms of those of $L$. 

\begin{proof}[Proof of Theorem~\ref{thm:scrollar.invariants.resolvent}]
Fix a partition $\lambda$ and write $n = \dim V_\lambda$ and $r = \mult(V_\lambda, \Ind^{S_d}_H \mathbf{1})$. By the proof of Corollary~\ref{cor:scrollar.invariants.Wlambda} we can write
$W_\lambda = V_1 \oplus \ldots \oplus V_n$
as an internal direct sum of irreducible subrepresentations $V_i \cong V_\lambda$, in such a way that a reduced basis of $W_\lambda$ is obtained from reduced bases $\{v_{i1}, \ldots, v_{in} \}$ of the $V_i$'s by simply taking the union. As before, we denote by $e_{\lambda, i}$ the unique scrollar invariant of $V_i$, which appears with multiplicity $n$.
From Lemma~\ref{lem: induced and fixed space} we see that the fixed subspaces $V_i^H$ have 
dimension $r$, and by applying the lemma over $k$ rather than $k(t)$, we find that $V_i^H$ admits a $k(t)$-basis consisting of $k$-linear combinations of 
the $v_{ij}$'s. It is immediate that this basis is reduced, with again $e_{\lambda, i}$ as the unique corresponding scrollar invariant, now appearing with multiplicity $r$. The union of these reduced bases is
a reduced basis for $W_\lambda^H = V_1^H \oplus \ldots \oplus V_n^H$ (e.g., because it extends to a reduced basis for $W_\lambda$), so the scrollar invariants of $W_\lambda$ are obtained by considering $r$ horizontal slices in~\eqref{eq:multiscrollarVB}. 

We claim that taking the union over all partitions $\lambda \vdash d$ produces a reduced basis of $L^H$. For every partition $\lambda$ of $d$, let $\{\omega_{\lambda, i}\}_i$ be the reduced basis of $W_\lambda^H$ just constructed. (For some $\lambda$, this set may be empty.) This is clearly a $k(t)$-basis for $L^H$. Let us first prove that it also concerns a $k[t]$-basis of $\OO_{L^H}$. Let $\omega \in \OO_{L^H}$ and write
\[
\omega = \sum_{\text{partitions }\lambda} \sum_i a_{\lambda, i}\omega_{\lambda,i},
\]
for some $a_{\lambda, i}\in k(t)$; 
it suffices to prove that $a_{\lambda, i}\in k[t]$ for any $\lambda, i$. 
Fix $\lambda$ and let $\rho\in Z(k[S_d])$ be the corresponding element from Lemma \ref{lem: projecting to a single W}. As $\rho$ is defined over $k$, it maps $\OO_L$ to itself and 
\[
\rho(\omega) = \sum_i a_{\lambda, i}\omega_{\lambda,i} \in \OO_L \cap W_\lambda^{H} = \OO_{W_\lambda}^{H}.
\]
But the set $\{\omega_{\lambda,i} \}_i$ is a basis for $\OO_{W_\lambda}$ and so $a_{\lambda, i}\in k[t]$ for all $i$. An identical argument proves that the set $\{ t^{-e^\lambda_i} \omega_{\lambda,i}\}_{\lambda, i}$ is a $k[t^{-1}]$-basis for $\OO_{L}^\infty$, with $e^\lambda_i$ the scrollar invariant associated with the basis element $\omega_{\lambda,i}$.
%This allows us to conclude.
%Fix an irreducible representation $V_\lambda$ of $S_d$. Let $\omega^\lambda_i$, $i=1, \ldots, \dim V_\lambda$ be in $\OO_{W_\lambda}$ such that the elements $t^{-e_{\lambda, i}-2}\omega_i^\lambda$ are all integral over $k[t^{-1}]$. Then the set $\{\sigma(\omega_i^\lambda)\mid \sigma \in S_d\}$ automatically contains a reduced basis for $\OO_{W_\lambda}$. By Lemma \ref{lem: induced and fixed space} over $k$, the space $V_{\lambda}^{H}$ has dimension $j_\lambda := \mult(\lambda, \Ind_H^{S_d} \mathbf{1})$. Hence, for every $i=1, \ldots, \dim V_\lambda$, there exist $j_\lambda$ linearly independent elements $\tilde{\omega}_{i,1}^\lambda, \ldots, \tilde{\omega}_{i,j_\lambda}^\lambda$ fixed under $H$, each of which is a $k$-linear combination of the elements $\sigma(\omega_i^\lambda)$. Since we are taking $k$-linear combinations, we have that the $\tilde{\omega}_{i,j}^\lambda$ are all integral over $k[t]$. Even more, the complete set $\{\tilde{\omega}_{i,j}^\lambda \mid 1\leq i\leq \dim V_\lambda, 1\leq j\leq j_\lambda\}$ may be extended to a reduced basis for $W_\lambda$, where the scrollar invariant of $\tilde{\omega}^\lambda_{i,j}$ is $e_{\lambda, i}$.
%Now the elements $\tilde{\omega}_{i,j}^\lambda$ over all $i,j$ and $\lambda$ taken together form a reduced basis for $\OO_{L^H}$. Indeed, the argument is completely analogous to that of Theorem~\ref{thm:unionofreducedbases}, so we omit further details.
\end{proof}

\subsection{} \label{ssec:firstexamplespart}
Some first examples of partitions and their scrollar invariants are:
the trivial partition $(d)$, with scrollar invariant $\{ 0 \}$, and the partition $(d-1, 1)$, with scrollar invariants $\{ e_1, e_2, \ldots, e_{d-1} \}$.
To see the latter claim: one can view our given cover $\varphi : C \to \PP^1$ 
as its own resolvent with respect to $S_{d-1}$, and then the claim follows from Theorem~\ref{thm:scrollar.invariants.resolvent}
along with 
\[ \Ind_{S_{d-1}}^{S_d} \mathbf{1} \cong V_{(d)} \oplus V_{(d-1,1)}.\] Alternatively, if $1, \alpha_1, \ldots, \alpha_{d-1}$ is a reduced basis for $\OO_K$ then after a replacement of $\alpha_i$ with $\alpha_i - \frac{1}{d}\Tr_{K/k(t)}(\alpha_i)$ we may assume that $\Tr_{K/k(t)}(\alpha_i) = 0$. Then the $\alpha_1, \ldots, \alpha_{d-1}$ form a reduced basis for $V_1 = W_{(d-1,1)}\cap K$ from which it follows that $(d-1,1)$ has scrollar invariants $e_1, \ldots, e_{d-1}$ in $L$. 

Another basic case
is the partition $(1^d)$, with scrollar invariant $\{ g' + 1 \}$, where $g'$ denotes the genus of $\res_{A_d} C$.
This again follows from Theorem~\ref{thm:scrollar.invariants.resolvent}, now using $\Ind_{A_d}^{S_d} \mathbf{1} \cong V_{(d)} \oplus V_{(1^d)}$. From Theorem~\ref{thm:genusresolvents} below it will follow that $g' = g + d - 2$
as soon as $\varphi$ is simply branched.

\subsection{A genus formula for resolvent curves.} \label{ssec:genusformula}

From this section onward we assume that the morphism $\varphi: C\to \PP^1$ is simply branched, i.e.\ all non-trivial ramification is of the form $(2, 1^{d-2})$; recall that this implies that $\Gal(K/k(t)) \cong S_d$. Under this assumption, we can express the genus
of $\res_H \varphi$ in terms of that of $C$. 
%For more general ramification, one has to look at singular resolvents as well, see section \ref{ssec:assumptions} for a more complete discussion. 
This may be known to specialists,\footnote{E.g., a special case is covered by~\cite[pp.\,111-113]{eisenbudelkies} and a number theoretic analogue is discussed in mathoverflow question~\href{https://mathoverflow.net/questions/6674/how-do-i-calculate-the-discriminant-of-a-galois-closure-and-its-other-subfields}{6674}.} but we could not find it explicitly in the existing literature, so let us include a proof.

\begin{theorem} \label{thm:genusresolvents} 
Let $\varphi : C \to \PP^1$ be a simply branched morphism of degree $d \geq 2$ over a field $k$ with $\charac k = 0$ or $\charac k > 2$ and let $H$ be a proper subgroup of $S_d$. The morphism $\res_H \varphi  : \res_H C  \to \PP^1$ is branched over exactly the same points as $\varphi$. Each such point ramifies with pattern $(2^{p(H)}, 1^{[S_d : H] - 2p(H)})$, where
%For each such point the contribution to the degree of the ramification divisor is scaled up by  
\[  p(H) = (d-2)! \cdot \frac{ | \{ \text{transpositions } \sigma \in S_d \, | \, \sigma \notin H \} |  }{|H|}. \]
In particular, the genus of $\res_H C $ is $p(H) (g + d - 1) + 1 - [S_d : H]$.
\end{theorem}

\noindent Note that the excluded case $H = S_d$ corresponds to the identity morphism $\PP^1 \to \PP^1$. Here $p(H) = 0$ so the theorem remains valid, except for the second sentence. At the other extreme, 
for $H = \{ \id \}$ which corresponds to the Galois closure $\overline{\varphi} : \overline{C} \to \PP^1$, we have $p(H) = d!/2$. In this case the theorem says that all ramification patterns of $\res_{\{\id\}} \varphi$ are $(2^{d! / 2})$.

\begin{lemma} \label{lem:genusresolventslemma}
Let $H$ be a subgroup of $S_d$ and let $\alpha\in L$ have stabilizer $H$ under the Galois action. Let $\sigma \in S_d$ be a transposition. Then the number of $\omega$ in the orbit of $\alpha$ for which $\sigma ( \omega) \neq \omega$ equals $2p(H)$.
\end{lemma}
\begin{proof}
The map $S_d \to \{ \text{transpositions in } S_d\} : \tau \mapsto \tau^{-1} \sigma \tau$ is surjective and $2(d-2)!$-to-$1$. Therefore the number of $\tau \in S_d$ such that $\tau^{-1} \sigma \tau \notin H$ equals $2 \cdot |H| \cdot p(H)$. Thus the number of cosets $\tau H \in S_d / H$ for which $\sigma  (\tau ( \alpha)) \neq \tau (\alpha)$ equals $2p(H)$, from which the lemma follows.
\end{proof}

\begin{proof}[Proof of Theorem~\ref{thm:genusresolvents}] By Riemann--Hurwitz it suffices to prove that $\res_H \varphi$ is branched over the same points as $\varphi$ with the stated ramification patterns. Because these properties are local, we can work around $0 \in \PP^1$ without loss of generality. 

Let $\mathfrak{d}_{K}\subseteq k[t]$ denote the discriminant ideal of $K/k(t)$, i.e., the principal ideal generated by the field discriminant
$\Delta_{K/k(t)}(\alpha_1, \alpha_2, \ldots, \alpha_d)$,
where $\alpha_1, \alpha_2, \ldots, \alpha_d$ denotes any $k[t]$-basis of $\OO_K$. Similarly consider the integral closure 
$\OO_{L^H}$ of $k[t]$ inside $L^H$ along with its discriminant ideal $\mathfrak{d}_{L^H} \subseteq k[t]$. 
We will first show that
\begin{equation} \label{powerofdisc}
 \mathfrak{d}_{L^H} = \mathfrak{d}_{K}^{p(H)}.
\end{equation}
For this, we rely on a discriminant formula due to Lenstra, Pila and Pomerance~\cite[Thm.\,4.4]{hyperelliptic}, which in our case states that
\begin{equation} \label{lenstraformula}
 \mathfrak{d}_{L^H}^{d!} = \prod_{\sigma \in S_d \setminus \{\id \}} \Norm_{L/k(t)}(\mathfrak{I}_\sigma)^{ \left| \{ \, \tau : L^H\hookrightarrow L \, \mid \, \sigma \circ \tau \neq \tau \, \} \right| }
\end{equation}
with $\mathfrak{I}_\sigma$
denoting the ideal generated by all expressions of the form $\sigma(x) - x$ for $x \in \OO_L$. Note that a non-zero prime ideal $\mathfrak{P} \subseteq \OO_L$ divides $\mathfrak{I}_\sigma$ if and only if $\sigma$ is in the inertia group of $\mathfrak{P}$. Using that $\varphi$ is simply branched, from~\cite[Sz.\,1]{vanderwaerden} we see that this inertia group is either trivial or consists of two elements. Therefore whenever $\sigma$ is not a transposition, the corresponding factor in~\eqref{lenstraformula} contributes trivially. On the other hand, if $\sigma$ is a transposition, then 
the number of embeddings $\tau: L^H \hookrightarrow L$ for which $\sigma \circ \tau \neq \tau$ equals $2p(H)$: this follows from Lemma~\ref{lem:genusresolventslemma} when applied to a primitive element $\rho$ of $L^H$ over $k(t)$, which can be chosen to be a polynomial expression in the conjugates $\alpha^{(1)}, \alpha^{(2)}, \ldots, \alpha^{(d)}$ of a primitive element $\alpha=\alpha^{(1)}$ of $K$ over $k(t)$. We find that
\[ 
\mathfrak{d}_{L^H}^{d!} = \prod_{\substack{\text{transpositions} \\ \sigma \in S_d}} \Norm_{L/k(t)}(\mathfrak{I}_\sigma)^{2p(H)}. 
\]
This formula in combination with the same formula applied to the group $H = S_{d-1}$ of permutations fixing $1$ yields~\eqref{powerofdisc}.

It then follows from~\cite[III, Thm.\,2.6 \& 2.9]{neukirch} that $\res_H \varphi$ is branched over the same points as $\varphi$. It also follows that, degree-wise, each branch point contributes $p(H)$ to the ramification divisor. Here we have used that $\charac k = 0$ or $\charac k > 2$, which ensures tame ramification. For $H = \{ \id \}$, corresponding to the Galois closure, we know that all ramification indices must be equal to each other. But since $p(H) = d! /2$ and the Galois closure has degree $d!$, we conclude that these ramification indices must in fact be $2$. Consequently, the ramification indices must also be $1$ or $2$ for every resolvent cover, from which the theorem follows.
\end{proof}

%\noindent Observe that the above proof not only implies that the morphism $\res_H \varphi $ ramifies above exactly the same points of $\PP^1$, but also that the corresponding ramification patterns transform from simple branching
% $(2, 1^{d-2})$ into $(2^{p(H)}, 1^{[S_d : H] - 2p(H)})$. For the Galois closure, corresponding to $H = \{ \id \}$, this reads $(2^{d! / 2})$.

\subsection{Volume and duality.} \label{ssec:volanddual}
We prove two general facts about the scrollar invariants of a partition $\lambda$ of $d$, still under the assumption that $\varphi$ is simply branched. Firstly, we prove a closed formula for their sum
\[
\vol_K (\lambda) = e_{\lambda, 1} + e_{\lambda, 2} + \ldots + e_{\lambda, \dim V_\lambda}.
\]
We call this the ``volume" of $\lambda$ with respect to $K / k(t)$.\footnote{This is in view of Minkowski's second theorem and the analogy between the scrollar invariants and the successive minima of the Minkowski lattice attached to a number field~\cite[\S7]{hessRR}; in fact calling it the ``log-volume" would make the analogy more precise.}
This generalizes the well-known formula $g + d - 1$ for the sum of the scrollar invariants $e_1, \ldots, e_{d-1}$ of $C$ with respect to $\varphi$, i.e.\ the scrollar invariants of the partition $(d-1,1)$.

\begin{proposition}[volume formula] \label{prop:genus.irrep}
Assume that $\varphi: C\to \PP^1$ is simply branched. Let $\lambda$ be a partition of $d$. Then
\[
\vol_K (\lambda)  = p(\lambda) \cdot (g+d-1),
\]
where $p(\lambda) = \frac{1}{2}(\dim V_\lambda - \chi_\lambda((1\,2)))$.
\end{proposition}

\noindent The quantity $p(\lambda)$ admits several other characterizations: 
\begin{equation*}
p(\lambda) = \langle V_{(1^2)}, \operatorname{Res}_{S_2}^{S_d} \lambda\rangle_{S_2} = \mult(\lambda, \Ind_{S_2}^{S_d} V_{(1^2)}).
\end{equation*}
One can also verify the inductive formula 
$ p(\lambda) = \sum_{\lambda' < \lambda} p(\lambda')$
for $d \geq 3$, where the sum is over all partitions of $d-1$ which are smaller than $\lambda$ (i.e., which are obtained from $\lambda$ by removing a box from the Young diagram).

\begin{proof}[Proof of Proposition~\ref{prop:genus.irrep}]
We claim that
\begin{equation*}
\sum_{\lambda\in \Ind_H^{S_d} \mathbf{1}} \vol_K(\lambda) = p(H)(g+d-1).
\end{equation*}
for each subgroup $H \subseteq S_d$, where the sum is understood to run over all partitions $\lambda$ of $d$, counted with multiplicity $\mult(V_\lambda, \Ind_H^{S_d} \mathbf{1})$.
Assuming the claim, we note that
\[ p(H) = \frac{1}{2}(\dim \Ind_{H}^{S_d} \mathbf{1} - \chi_{\Ind_H^{S_d} \mathbf{1} }((1\,2))),\] 
therefore 
\[
p(H) = \sum_{\lambda\in \Ind_H^{S_d} \mathbf{1} } p(\lambda),
\]
allowing us to conclude that the linear system of equations
\[
\sum_{\lambda\in \Ind_H^{S_d} \mathbf{1}} X_\lambda = \sum_{\lambda\in \Ind_H^{S_d} \mathbf{1}} p(\lambda)(g+d-1)
\]
(one equation for each $H \subseteq S_d$) in the variables $X_\lambda$ (one variable for each $\lambda \vdash d$) admits the solution $X_\lambda = \vol_K(\lambda)$. But, clearly, also $X_\lambda = p(\lambda)(g + d - 1)$ is a solution, so it suffices to see that the solution is unique.
For this, restrict to those equations for which $H=S_\rho$ is a Young subgroup, with $\rho$ a partition of $d$. By~\cite[p.\,88]{sagan} the resulting coefficient matrix is upper triangular (with respect to a certain ordering on all partitions) and all diagonal elements are $1$. Hence our system indeed has a unique solution, as wanted.

 To prove the claim, we consider a reduced basis 
$\{ \omega_i^\lambda \}_{\lambda, i}$ of $L^H / k(t)$ as in the proof
of Theorem~\ref{thm:scrollar.invariants.resolvent}.
That is: for every fixed partition $\lambda$ of $d$ the elements
$\omega_i^\lambda$,  $i = 1, \ldots, r \dim V_\lambda$
form a reduced basis of $W_\lambda$, and the corresponding scrollar invariants $e_i^\lambda$
are the $e_{\lambda, j}$'s, each counted with multiplicity $r$.
The discriminant ideal
$\mathfrak{d}_{L^H} \subseteq k[t]$
is generated by
$\det T$, where $T = (\Tr_{L/k(t)}(\omega_i^\lambda \omega_j^{\lambda'}))_{\lambda,\lambda',i,j}$. If $\lambda \neq \lambda'$ are two  partitions of $S_d$ then $\Tr_{L/k(t)}(\omega_i^\lambda \omega_j^{\lambda'}) = 0$, as $V_\lambda\otimes V_{\lambda'}$ does not contain a copy of the trivial representation. Thus $T$ falls apart into the block matrices $T_\lambda = (\Tr_{L/k(t)}(\omega_i^\lambda \omega_j^\lambda))_{i,j}$. We obtain that
\[
 \mathfrak{d}_K^{p(H)} = \mathfrak{d}_{L^H}  = (\det T) = \left( \prod_{\lambda} \det T_\lambda \right)  
\]
where the first equality follows from the proof of Theorem~\ref{thm:genusresolvents}; note that this uses the simple branching assumption (also note that this equality remains true for $H = S_d$).
We now repeat this argument above infinity, where we write $\mathfrak{d}_{K}^\infty, \mathfrak{d}_{L^H}^\infty \subseteq k[t^{-1}]$
to denote the discriminant ideals of $\mathcal{O}_{K}^\infty, \mathcal{O}_{L^H}^\infty$.
 Since the $\omega_i^\lambda$'s form a reduced basis, the corresponding matrices $T_{\lambda}^\infty$ are equal to 
 \[ (\Tr_{L/k(t)}(t^{-e_i^\lambda-e_j^\lambda} \omega_i^\lambda \omega_j^\lambda))_{i,j}. \] 
% We also have $\mathfrak{d}_{K}^\infty = t^{-2(g+d-1)}\mathfrak{d}_K$, so we find
We similarly find
\[
 (\mathfrak{d}_{K}^\infty)^{p(H)}  = \mathfrak{d}_{L^H}^\infty = \left( \prod_{\lambda } \det T_{\lambda}^\infty \right)
 = \left( \prod_{\lambda }  t^{-2\sum_i e_i^\lambda }  \det T_\lambda \right). \\
 \]
Using that a generator of $\mathfrak{d}_{K}^\infty$ is obtained from a generator of $\mathfrak{d}_K$
by scaling with $t^{-2(g + d - 1)}$ we find that
\[ -2p(H)(g + d - 1) =
 \sum_\lambda -2\sum_i e_i^\lambda  = -2 \sum_{\lambda \in \Ind_H^{S_d} \mathbf{1}} e_{\lambda, j}
 \]
from which the claim follows.
\end{proof}

\subsection{} 
Secondly, we can explicitly relate the scrollar invariants of a partition $\lambda$ to those of the dual partition $\lambda^*$, which is obtained by transposing the corresponding Young diagram. 
%We show that the scrollar invariants of $W_{\lambda}$ and $W_{\lambda^*}$ are also dual in a natural way. 

\begin{proposition}[duality] \label{prop:scrollar.invs.duality} 
Let $\lambda$ be a partition of $d$ and consider its multi-set of scrollar invariants $\{ e_{\lambda, i} \}_i$. Then the multi-set of scrollar invariants of $\lambda^*$ is given by $\{ g + d- 1 - e_{\lambda, i} \}_i$.
\end{proposition}

\begin{proof}
It suffices to show that the scrollar invariants of $W_\lambda$ and $W_{\lambda^\ast}$ are obtained
from each other through the map $e \mapsto g + d - 1 - e$.
Let $\omega_i$, $i=1, \ldots, (\dim V_\lambda)^2$, be a reduced basis for $W_\lambda$ with corresponding scrollar invariants $e_{\lambda, i}$ and let $\delta$ be a generator of $W_{(1^d)}$ as in~\ref{ssec:resumenotation}. By Theorem~\ref{thm:genusresolvents} the scrollar invariant of $W_{(1^d)}$ is $g+d-1$. Let $\DD_{L/k(t)}$ denote the different ideal of the extension $L/k(t)$ and note that 
this is just the principal ideal generated by $\delta$ because all branch points have ramification pattern $(2^{d!/2})$ by Theorem~\ref{thm:genusresolvents}.
%because $K/k(t)$ is simply branched, the discussion following the proof of Theorem \ref{thm:genusresolvents} implies that $\DD_{L/k(t)}$ is the principal ideal generated by $\delta$. 
Let $\omega_i^*$ be the dual basis of $W_\lambda$ with respect to the trace pairing, i.e., $\Tr_{L / k(t)}(\omega_i \omega_j^*) = \delta_{ij}$ where $\delta_{ij}$ is the Kronecker delta. We claim that the elements $\delta\omega_i^*$ form a reduced basis of $W_{\lambda^*}$. To prove this, we use Lemma \ref{lem:reducedbasis.guess}. First of all, note that the elements
\[ \delta \omega_i^* \in \delta \DD^{-1}_{L/k(t)} = \OO_L \]
are integral over $k[t]$.
%\[
%\delta^2 (\omega_i^*)^2 \in \delta^2 (\DD^{-1}_{L/k(t)})^2 = \OO_L,
%\]
%so that the $\delta\omega_i^*$ are integral over $k[t]$. 
Also, it is clear that they form a $k(t)$-basis for $W_{\lambda^*}$. Playing the same game above infinity, so starting with  $t^{-e_{\lambda, i} }\omega_i$, we find the elements $t^{-(g+d-1-e_{\lambda, i})-2}\delta \omega_i^*$, which are all integral over $k[t^{-1}]$. Using Proposition \ref{prop:genus.irrep} for $\lambda$, we compute
\[
\sum_i (g+d-1-e_{\lambda, i})  = \frac{1}{2}(\dim V_\lambda + \chi_\lambda((1\,2)))(g+d-1),
\]
from which the desired result follows, as $\dim V_\lambda = \dim V_{\lambda^*}$ and $\chi_{\lambda^*}((1\,2)) = - \chi_\lambda ((1\,2))$.
\end{proof}

\subsection{Scrollar invariants of hooks.} \label{ssec:scrollar.invs.of.hooks}

Recall from~\ref{ssec:firstexamplespart} that the scrollar invariants of the standard partition $(d-1,1)$ are simply the scrollar invariants $e_1, \ldots, e_{d-1}$ of our input map $\varphi: C\to \PP^1$, and that the sign representation $(1^d)$ has $g + d - 1$ as its unique scrollar invariant. These facts are part of the more general statement from Proposition~\ref{prop:scrollar.invs.of.hooks}, which we now prove:

\begin{proof}[Proof of Proposition~\ref{prop:scrollar.invs.of.hooks}]
Note that indeed
\[ \dim V_\lambda = {d - 1 \choose i} \]
is the number of $i$-element subsets of $\{1, 2, \ldots, d- 1\}$.
Take a reduced basis $\alpha_1, \ldots, \alpha_{d-1}$ of $V_1 = W_{(d-1,1)}\cap L^{S_{d-1}}$, with scrollar invariants $e_1, \ldots, e_{d-1}$. Then 
$\{\alpha_j^{(m)} \mid 1\leq j,m\leq d-1\}$ 
is a reduced basis for $W_{(d-1,1)}$. The determinant of the matrix
\[
D = \begin{pmatrix}
\alpha_1^{(1)} & \alpha_1^{(2)} & \hdots & \alpha_1^{(d-1)} \\
%\alpha_2^{(1)} & \alpha_2^{(2)} & \hdots & \alpha_2^{(d-1)} \\
\vdots & \vdots & \ddots & \vdots \\
\alpha_{d-1}^{(1)} & \alpha_{d-1}^{(2)} & \hdots & \alpha_{d-1}^{(d-1)}
\end{pmatrix}
\]
appears, up to sign, as a summand in the expansion of the determinant of~\eqref{eq:discmatrix} 
along the first row. Since the latter determinant is non-zero and all these summands are conjugate to each other, we see that $\det D \neq 0$.
 We claim that the $i\times i$ minors of $D$ form a reduced basis for $W_\lambda$. First note that acting with $S_d$ on such a minor generates a representation isomorphic to $V_\lambda$, so it is indeed contained in $W_\lambda$. Secondly, we have 
\[ \dim W_\lambda = (\dim V_\lambda)^2 = {d-1 \choose i}^2 \]
minors and they are linearly independent because the $(d-1) \times (d-1)$ matrix whose entries are these minors has determinant 
\[ \det(D)^{\binom{d-2}{i-1}}\neq 0 \]
by the Sylvester--Franke theorem. Thus our minors form a basis of $W_\lambda$, and clearly they are all integral over $k[t]$. Now repeat this construction starting from the basis
$t^{-e_1}\alpha_1, \ldots, t^{-e_{d-1}} \alpha_{d-1}$, in order to obtain another basis of $W_\lambda$ consisting of $i \times i$ minors which are integral over $k[t^{-1}]$. The statement then follows through an application of Lemma \ref{lem:reducedbasis.guess}.
\end{proof}

\subsection{} As a first non-trivial application of the work done above, we can describe the scrollar invariants of $\res_{A_{d-1}} \varphi : \res_{A_{d-1}} C \to \PP^1$, where $A_{d-1}$ denotes the subgroup of $S_d$  consisting of all even permutations fixing $1$.

\begin{corollary}\label{cor:scrollars.Ad-1}
The scrollar invariants of the curve $\res_{A_{d-1}} C$ with respect to $\res_{A_{d-1}} \varphi $ are $e_1, \ldots, e_{d-1}, g+d-1-e_1, \ldots, g+d-1-e_{d-1}$ and $g+d-1$. 
\end{corollary}

\begin{proof}
We have that $
\Ind_{A_{d-1}}^{S_d} \mathbf{1} = V_{(d)}\oplus V_{(d-1,1)} \oplus V_{(2, 1^{d-2})} \oplus V_{(1^d)}$.
Thus the result follows from Proposition~\ref{prop:scrollar.invs.duality} (or from Proposition~\ref{prop:scrollar.invs.of.hooks}), together with Theorem~\ref{thm:scrollar.invariants.resolvent}.
\end{proof}

\noindent We believe that, as soon as $d \geq 4$, the only resolvent covers whose scrollar invariants admit a description completely in terms of the scrollar invariants of $\varphi : C\to \PP^1$ are those with respect to $S_{d-1}, A_{d-1}, A_d$ (and the degenerate case $S_d$), although we do not have a proof for this.

\subsection{The case of $(3,1^{d-3})$ ramification} \label{ssec:31111}

For use in Section~\ref{sec:schreyer.invs.scrollar} we extend our proof of the volume formula for the partition $(d-2,2)$ such that it also covers the case of $(3, 1^{d-3})$ ramification. This allows for a slight generalization of our results. We say that a degree $d \geq 4$ morphism $\varphi: C\to \PP^1$ has ``good ramification" if all non-trivial ramification is of type $(2,1^{d-2})$ or $(3, 1^{d-3})$ and the Galois closure of $\varphi$ has the full symmetric group $S_d$ as its Galois group; the latter property no longer follows automatically.

\begin{lemma}\label{lem:genus.good.ramification}
If $\varphi$ has good ramification, then
\[
\vol_K ((d-2,2))  = p((d-2,2)) (g+d-1).
\]
\end{lemma}
\begin{proof}
We start by modifying the proof of Theorem~\ref{thm:genusresolvents} for the subgroup $H = S_2\times S_{d-2}$, in such a way that it covers the case of good ramification, rather than just simple branching. 
Remember that this proof boiled down to showing that
\[ \left| \{ \, \tau : L^H\hookrightarrow L \, \mid \, \sigma \circ \tau \neq \tau \, \} \right| = p(H) \cdot \left| \{ \, \tau : L^{S_{d-1}}\hookrightarrow L \, \mid \, \sigma \circ \tau \neq \tau \, \} \right| \]
for every transposition $\sigma \in S_d$. It suffices to show that this is also true in the case where $\sigma$ is a $3$-cycle, in which case the right-hand side becomes $3p(H) = 3(d-2)$. Let us rewrite the left-hand side as 
\[ \frac{ | \{\tau\in S_d \mid \tau^{-1}\sigma \tau \notin H \}|}{|H|} = \chi_{\Ind_H^{S_d} \mathbf{1}} (\id) - \chi_{\Ind_H^{S_d} \mathbf{1}} (\sigma). \] 
%The same argument is valid for $S_{d-1}$ and so it is enough to prove that
%\[
%\chi_{\Ind_H^{S_d} \mathbf{1}} (\id) - \chi_{\Ind_H^{S_d} \mathbf{1}} (\sigma) = (d-2)(\chi_{\Ind_{S_{d-1}}^{S_d} \mathbf{1}}(\id) - \chi_{\Ind_{S_{d-1}}^{S_d} \mathbf{1}}(\sigma)).
%\]
%The right hand side is equal to $(d-2)(d-4)$. // DENK DAT DIT NIET KLOPT 
Since we know that $\Ind_H^{S_d} \mathbf{1} \cong V_{(d)}\oplus V_{(d-1,1)}\oplus V_{(d-2,2)}$ it suffices to prove that $\beta_1 - \chi_{(d-2,2)}(\sigma) = 3(d-3)$ for every $3$-cycle $\sigma$. The character $\chi_{(d-2,2)}(\sigma)$ can be computed using the Murnaghan--Nakayama rule~\cite[Thm.\,4.10.2]{sagan} and is found to be equal to 
\[
\chi_{(d-2,2)}(\sigma) = \begin{cases}
-1 & \text{if } d = 4,5, \\
0 & \text{if } d=6, \\
\dim V_{(d-5,2)} & \text{if } d\geq 7.
\end{cases}
\]
A small calculation then proves the desired genus formula for $L^H$. 

To conclude, the proof of Proposition \ref{prop:genus.irrep} shows that 
\[
\sum_{\lambda\in \Ind_H^{S_d} \mathbf{1}}  \vol_K(\lambda)  = p(H)(g+d-1),
\]
so the result follows from the fact that $\vol_K( (d-1,1)) = g+d-1$ and $\Ind_H^{S_d} \mathbf{1} \cong V_{(d)}\oplus V_{(d-1,1)}\oplus V_{(d-2,2)}$.
\end{proof}

In general, we expect that for every partition $\lambda \vdash d$ there is a list of ``allowed" ramification patterns, such that Proposition~\ref{prop:genus.irrep} holds. For $\lambda=(d-1,1)$ any ramification type is allowed, while for $\lambda=(d-2,2)$ we have just proven that $(2, 1^{d-2})$ and $(3,1^{d-3})$ are both good. We expect that no other non-trivial ramification is good for $(d-2,2)$. We will say more about this in~\ref{ssec:assumptions}--\ref{ssec:assumptionssometimesok}.

\section{Schreyer's invariants are scrollar} \label{sec:schreyer.invs.scrollar}

%In this section we reinterpret Schreyer's invariants of a degree $d\geq 4$ cover $C\to \PP^1$ as the scrollar invariants of associated to certain partitions of $d$. In other words, we give a proof of Theorem~\ref{thm:schreyerisscrollar}.

\subsection{}  
This section is devoted to proving Theorem~\ref{thm:schreyerisscrollar}, reinterpreting the Schreyer invariants of a simply branched degree $d \geq 4$ morphism $\varphi : C \to \PP^1$ as scrollar invariants of the partitions $\lambda_{i+1} = (d-i-1, 2, 1^{i-1})$, where $i = 1, \ldots, d-3$. The proof can be found in~\ref{ssec:schreyerscrollarsection}. We begin with gathering some facts on the relative canonical embedding.

\subsection{Defining equations of $C$ inside $\PP(\mathcal{E})$ and their syzygies.}
%Consider a morphism $\varphi : C \to \PP^1$ of degree $d \geq 4$.
%\footnote{If $d=3$ then our curve $C$ arises in $\PP(\mathcal{E})$ as a divisor in the class $3H - (g+2)R$. In other words, it is cut out by a binary cubic form 
%$\sum_{j_1 + j_2 = 3} \varphi_{j_1,j_2} \cdot x_1^{j_1}x_2^{j_2}$
%whose coefficients $\varphi_{j_1, j_2} \in k[s,t]$ are of degree $j_1 e_1 + j_2 e_2 - g-2$.} 
We know from~\eqref{Schresolution} that 
the relative canonical embedding realizes $C$ inside $\PP(\mathcal{E})$ as the intersection of $\beta_1$ 
%$= d(d-3)/2$ 
divisors
\begin{equation} \label{classofDi} 
  D_j \in 2H - b_j^{(1)}R
\end{equation}
%that sum up to 
%\[ (d-3)g -d^2 + 2d + 3,\]
and this is a minimal set of generators. 
Equipping the fibers of our $\PP^{d-2}$-bundle $\pi : \PP(\mathcal{E}) \to \PP^1$ with homogeneous coordinates $x_1, x_2, \ldots, x_{d-1}$ and similarly providing $\PP^1$ with homogeneous coordinates $s,t$, it makes sense to talk about defining equations. 
%; in fact if $S$ is singular it is somewhat better to work with the pull-back of $C$ along $\PP(\mathcal{E}) \to S$. 
Being in the class~\eqref{classofDi} then amounts to being defined by a quadratic form
\begin{equation} \label{eq:quadformgeneral} 
  \sum_{ j_1 + j_2 + \ldots + j_{d-1} = 2 }  \varphi_{j_1, j_2, \ldots, j_{d-1}} \cdot x_1^{j_1} x_2^{j_2} \cdots x_{d-1}^{j_{d-1}} 
\end{equation}
where each coefficient $\varphi_{j_1, j_2, \ldots, j_{d-1}} \in k[s,t]$ is homogeneous of degree 
\begin{equation} \label{degquadform} 
  j_1 e_1 + j_2 e_2 + \ldots + j_{d-1}e_{d-1} - b_j^{(1)} 
\end{equation}
(the coefficient is zero if this quantity is negative). The morphism $\varphi$ is just the restriction of $\pi$ to $C$, i.e., it amounts to projection on the $(s,t)$-coordinates. 

The next steps of the resolution~\eqref{Schresolution} can be made explicit as well, in terms of syzygies.
Concretely, for $1 \leq i \leq d-4$ the map
\[ \bigoplus_{j=1}^{\beta_{i+1}} \mathcal{O}_{\PP(\mathcal{E})}(-(i+2)H + b_j^{(i+1)}R) 
\to \bigoplus_{j=1}^{\beta_i} \mathcal{O}_{\PP(\mathcal{E})}(-(i+1)H + b_j^{(i)}R) \]
can be represented by a $\beta_i \times \beta_{i+1}$ matrix whose entry on row $j_1$ and column $j_2$
is a linear form
$\varphi_1 x_1 + \varphi_2 x_2 + \ldots + \varphi_{d-1} x_{d-1}$
where each $\varphi_\ell \in k[s,t]$ is homogeneous of degree
\[ e_\ell + b_{j_1}^{(i)} - b_{j_2}^{(i+1)} . \]
The last step of the resolution is then again described by $\beta_{d-3} = \beta_1$ quadratic forms as in~\eqref{eq:quadformgeneral}, where now the coefficients
$\varphi_{j_1, j_2, \ldots, j_{d-1}} \in k[s,t]$ are seen to be homogeneous of degree 
\[ 
  j_1 e_1 + j_2 e_2 + \ldots + j_{d-1}e_{d-1} + b_j^{(d-3)} + d + 1 - g, 
\]
which in fact just equals~\eqref{degquadform} for an appropriate ordering of the Schreyer invariants, because of duality: see~\eqref{eq:schreyerduality} below.

%Once again we refer to~\cite{casnati_ekedahl,schreyer} for background and justification.

\subsection{} \label{ssec:bhargava_is_relativecanonical} 
The geometric generic fiber of $\varphi$ is the configuration of $d$ points in $\PP^{d-2}$ over $k(t)$ 
cut out by the quadrics
\[   \sum_{ j_1 + j_2 + \ldots + j_{d-1} = 2 }  \varphi_{j_1, j_2, \ldots, j_{d-1}}(1,t) \cdot x_1^{j_1} x_2^{j_2} \cdots x_{d-1}^{j_{d-1}} \in k(t)[x_1, x_2, \ldots, x_{d-1}] \]
for $j = 1, 2, \ldots, \beta_1$.
Let us become more precise about our claim from~\ref{ssec:intro_relative_embedding}
that this geometric generic fiber is in fact a point configuration of Bhargava type.

\begin{proposition} \label{prop:bhargava_is_casnati_ekedahl}
Consider a morphism $\varphi : C \to \PP^1$ of degree $d \geq 3$ and consider 
a reduced basis $1, \alpha_1, \ldots, \alpha_{d-1}$ of the corresponding function field extension $k(t) \subseteq k(C)$.
The geometric generic fiber of $\varphi$, when viewed as a configuration of $d$ points in $\PP^{d-2}$ through the relative canonical embedding, can be identified with Bhargava's point set 
\[ [\alpha_1^{* (1)} : \ldots : \alpha_{d-1}^{* (1)}], \ [\alpha_1^{* (2)} : \ldots : \alpha_{d-1}^{* (2)}], \  \ldots, \ [\alpha_1^{* (d)} : \ldots : \alpha_{d-1}^{* (d)}] \]
attached to $1, \alpha_1, \ldots, \alpha_{d-1}$ as in~\ref{ssec:veryintro}.
\end{proposition}

\begin{proof}
Recall from~\ref{ssec:reducedbasisdef} that our reduced basis $1, \alpha_1, \ldots, \alpha_{d-1}$ is a $k[t]$-basis
of $\OO_{k(C)}$ such that $1, t^{-e_1 }\alpha_1, \ldots, t^{-e_{d-1} }\alpha_{d-1}$ is a basis of $\OO_{k(C)}^\infty$ over $k[t^{-1}]$, and that this property is tantamount to the decomposition
\[ \varphi_\ast \OO_C \cong \OO_{\PP^1} \oplus \mathcal{E}^\vee, \qquad \mathcal{E} = \OO_{\PP^1}(e_1) \oplus \cdots \oplus \OO_{\PP^1}(e_{d-1}) \]
from~\eqref{pushfwd}. 
Also, remember from~\ref{ssec:firstexamplespart} that the $\alpha_i$ can be assumed to have trace zero, i.e., they form a reduced basis of $V_1$. Then the dual basis is of the form
$1, \alpha_1^*, \ldots, \alpha_{d-1}^*$. We claim that this dual basis plays a similar role
in the decomposition
\begin{equation*} %\label{eq:relpushf} 
  \varphi_\ast \OO_C(K_{C / \PP^1}) \cong \OO_{\PP^1} \oplus \mathcal{E}
\end{equation*} 
with $K_{C / \PP^1}$ the ramification divisor of $\varphi$ (this is a relative canonical divisor). Indeed, 
the explicit formulae $\alpha_i^\ast = \det D_{i+1,1}/\det D$  
from~\ref{ssec:proofthatquadricsvanish} imply that
$1, \alpha_1^\ast, \ldots, \alpha_{d-1}^\ast$ forms a $k[t]$-basis of $\OO_C(K_{C/ \PP^1})(\AA^1)$,
while the counterparts of these formulae above infinity show that $1$, $t^{e_1}\alpha_1^\ast$, \ldots, $t^{e_{d-1}} \alpha_{d-1}^\ast$
is a $k[t^{-1}]$-basis of $\OO_C(K_{C/ \PP^1})(\PP^1 \setminus \{ 0 \})$.\footnote{The same reasoning for the absolute canonical divisor $K_C = K_{C / \PP^1} - 2(t)_\infty = (dt)$ shows that $1, \alpha_1^\ast, \ldots, \alpha_{d-1}^\ast$ and $t^{-2}, t^{e_1 - 2} \alpha_1^\ast, \ldots, t^{e_{d-1} - 2} \alpha_{d-1}^\ast$ are respective bases for $\OO_C(K_C)(\AA^1)$ and $\OO_C(K_C)(\PP^1 \setminus \{0\})$.
This yields the decomposition $\varphi_\ast \OO_C(K_C) \cong \OO_{\PP^1}(-2) \oplus \mathcal{E}(-2)$ mentioned in~\ref{ssec:differentscrollar}. \label{footnote:dt}}

%Thus $\iota$,
%locally above $0$, embeds $C$ using the functions $\alpha_1^\ast$, $\alpha_2^\ast$, \ldots, $\alpha_{d-1}^\ast$, from which the result follows.

%We quickly recall 
%how the relative canonical embedding is constructed, while referring to Schreyer~\cite{schreyer}, Casnati--Ekedahl~\cite[Thm.\,2.1]{casnati_ekedahl} %, Farkas--Kemeny~\cite[\S4]{farkas_kemeny} 
%or the recent account Landesman--Vakil--Wood~\cite[\S3.1]{landesman_vakil_wood} for further details.
%  For any degree $d \geq 3$ morphism $\varphi : C \to \PP^1$ with scrollar invariants $e_1, e_2, \ldots, e_{d-1}$, one can consider $\mathcal{E} = \OO_{\PP^1}(e_1) \oplus \OO_{\PP^1}(e_2) \oplus \ldots \oplus \OO_{\PP^1}(e_{d-1})$ as in~\ref{ssec:intro_relative_embedding}, where we caution the reader that several authors work instead with the ?? of $\varphi$, which is $\mathcal{E}$ twisted by $\OO_{\PP^1}(-2)$. Through dualization, the isomorphism~\eqref{pushfwd} implies
%\begin{equation} \label{eq:decomp_canonical} 
%  \varphi_\ast \OO_C(K_C) \cong \OO_{\PP^1}(-2) \oplus \mathcal{E},
%\end{equation}
%which yields an injection $\varphi^\ast \mathcal{E} \to \OO_C(K_C)$ inducing a closed immersion $\iota : C \to \PP(\mathcal{E})$, which is precisely our relative canonical embedding.\todo{entire proof to be double-checked and adapted to new convention}

The relative canonical embedding is the map $\iota : C \hookrightarrow \PP(\mathcal{E})$
induced by the morphism $\varphi^\ast \mathcal{E} \to \varphi^\ast \varphi_\ast \OO_C(K_{C / \PP^1}) \to \OO_C(K_{C / \PP^1})$ (which turns out to be surjective) coming from 
the natural inclusion $\mathcal{E} \to \varphi_\ast \OO_C(K_{C/ \PP^1})$; see Casnati--Ekedahl~\cite[Thm.\,2.1]{casnati_ekedahl} %, Farkas--Kemeny~\cite[\S4]{farkas_kemeny} 
or the recent treatment due to Landesman--Vakil--Wood~\cite[\S3.1]{landesman_vakil_wood} for details.
In view of the above, in explicit terms this means that on the patch above~$\AA^1$, the morphism $\iota$ is given by
\[ \varphi^{-1}(\AA^1) \, \to \, \AA^1 \times \PP^{d-2} : p \mapsto ( \, \varphi(p), \, [\alpha_1^\ast(p) : \alpha_2^\ast(p) : \ldots : \alpha_{d-1}^\ast(p)] \, ), \]
from which the proposition follows.
\end{proof}

\subsection{Schreyer's invariants are scrollar.} %\label{ssec:schreyerisscrollar}
A key role in the proof of Theorem~\ref{thm:schreyerisscrollar} will be played by the following formula, saying that the volume of the partition $\lambda_{i+1}$ is as it should be:

\begin{lemma}\label{lem:sum.schreyer.invs}
Let $1\leq i\leq d-3$. Then we have that 
\[
\vol_{k(C)} (\lambda_{i+1}) = \sum_{j=1}^{\beta_i} b_j^{(i)}.
\]
\end{lemma}

\begin{proof}
Firstly, we claim that $p(\lambda_{i+1}) = (d-2-i)\binom{d-2}{i-1}$. To prove this claim, define
\[
\beta^{(e)}_i = \frac{e}{i+1}(e-2-i)\binom{e-2}{i-1}.
\]
Using the Murnaghan--Nakayama rule~\cite[Thm.\,4.10.2]{sagan}, one computes that 
\[
\chi_{(d-2,2)}((1\,2)) = \begin{cases}
0 & d=4, \\
1 & d=5, \\
1+\beta_1^{(d-2)} & d\geq 6, 
\end{cases} 
\quad \chi_{(d-3,2,1)}((1\,2)) = \begin{cases}
-1 & d = 5, \\
0 & d=6, \\
\beta_2^{(d-2)} & d\geq 7,
\end{cases} 
\]
while for $3\leq i\leq d-5$ we have $\chi_{\lambda_{i+1}}((1\,2)) = \beta_i^{(d-2)} - \beta_{i-2}^{(d-2)}$. The values of $\chi_{\lambda_{d-3}}((1\,2))$ and $\chi_{\lambda_{d-2}}((1\,2))$ can be determined by duality. The claim follows by an explicit calculation.

On the other hand, by~\cite[Prop.\,2.9]{bopphoff} we have that 
\begin{equation}%\label{eq:sum.schreyer.invs}
\sum_{j=1}^{\beta_i} b_j^{(i)} = (d-2-i){d-2 \choose i-1}(g+d-1),
\end{equation}
where we caution the reader that~\cite[Prop.\,2.9]{bopphoff} uses the different convention for the $b_j^{(i)}$ that was discussed in~\ref{ssec:differentschreyer}.
So the lemma follows from Proposition~\ref{prop:genus.irrep}.
%\[
%\vol_{k(C)}(\lambda_{i+1})  = p(\lambda_{i+1})(g + d - 1) = \sum_{j = 1}^{\beta_i} b_j^{(i)}. \qedhere
%\]
\end{proof}

\subsection{} \label{ssec:schreyerscrollarsection} We are now
ready to prove Theorem~\ref{thm:schreyerisscrollar}.

\begin{proof}[Proof of Theorem~\ref{thm:schreyerisscrollar}]
For $i = 1, \ldots, d-3$, write
$e_1^{(i)} \leq e_2^{(i)} \leq \ldots \leq e_{\beta_i}^{(i)}$
for the scrollar invariants of $\lambda_{i+1}$ with respect to $\varphi$, and likewise sort the Schreyer invariants such that
\[ b_1^{(i)} \leq b_2^{(i)} \leq \ldots \leq b_{\beta_i}^{(i)}.\]
Our aim is to prove that $e_j^{(i)} = b_j^{(i)}$ for all $i$ and $j$.

We follow the notation from Section \ref{sec:minfreerep}. 
In regards to bases, we now assume that $\alpha_1$, \ldots, $\alpha_{d-1}$ is a reduced basis for $V_1$ and similarly that $\omega^i_{1}, \ldots, \omega_{\beta_{i-1}}^i$ is a reduced basis for $V_i$, $i=2, \ldots, d-2$. Note that the scrollar invariants of $V_1$ are simply the scrollar invariants 
$e_1, \ldots, e_{d-1}$ of $C$ with respect to $\varphi$, while those of
$V_i$ are
\[ e_1^{(i-1)}, e_2^{(i-1)}, \ldots, e_{\beta_{i-1}}^{(i-1)} \]
for $ i =2, \ldots, d-2$. This follows, for instance, from the proof of Theorem~\ref{thm:scrollar.invariants.resolvent}.

Applying the construction from~Section~\ref{sec:minfreerep}, in view of Proposition~\ref{prop:bhargava_is_casnati_ekedahl} we obtain
 a minimal graded free resolution 
\[
...\to V_4^*\otimes R(-4)\xrightarrow{\psi^*_3} V_3^*\otimes R(-3)\xrightarrow{\psi_2^*} V_2^*\otimes R(-2)\xrightarrow{\psi_1^*} R\to R/I\to 0.
\]
of the homogeneous coordinate ring of the geometric generic fiber of $\varphi$ as a graded module
over the polynomial ring $R = k(t)[x_1, \ldots, x_{d-1}]$.
We can assume that the polynomials
$p^1, \ldots, p^{d-3}$, used to define
the maps 
\[ \psi_1: \Sym^2 V_1\to V_2, \qquad \psi_i: V_i\otimes V_1\to V_{i+1}, i = 2, 3, \ldots, d-3, \]
%, \qquad \psi_{d-2}: V_{d-2}\otimes \Sym^2 V_1\to V_d \]
have coefficients in $k$: indeed, we can find them using representation theory over $k$, rather than $k(t)$. 
Consequently, the maps $\psi_1, \ldots, \psi_{d-3}$ send tensors of elements that are integral over $k[t]$, resp.\ $k[t^{-1}]$, to elements that are integral over $k[t]$, resp.\ $k[t^{-1}]$. %This implies that our resolution is defined over $k[t]$.

%Therefore, since our chosen bases for the $V_i$ are integral over $k[t]$, this resolution is already defined over $k[t]$. Our resolution, considered over $k[t][x_1, \ldots, x_{d-1}]$, is certainly a complex. Therefore, it must contain the minimal free resolution of $C$ over $k[t][x_1, \ldots, x_{d-1}]$, i.e.\ Schreyer's resolution, as a direct summand. The goal is to obtain degree bounds on the maps in this complex and show that it is already a minimal resolution.

Let us first discuss the implications for $\psi_1$. Because $\alpha_1, \ldots, \alpha_{d-1}$ and $\omega_1^2, \ldots, \omega_{\beta_1}^2$ are $k[t]$-bases
we must have
\[
\psi_1(\alpha_m\otimes \alpha_n) = \sum_{\ell = 1}^{\beta_1} Q_{mn}^\ell \omega^2_\ell,
\]
with $Q_{mn}^\ell\in k[t]$, for all $1 \leq m \leq n \leq d-1$. By linearity, it follows that 
\[
\psi_1(t^{-e_m}\alpha_m \otimes t^{-e_n} \alpha_n) = \sum_{\ell = 1}^{\beta_1} t^{-e_m-e_n+e_\ell^{(1)} } Q_{mn}^\ell (t^{-e_\ell^{(1)}  }\omega_\ell^2).
\]
But $t^{-e_1 } \alpha_1, \ldots, t^{-e_{d-1}}\alpha_{d-1}$ and $t^{-e_1^{(1)} }\omega_1^2, \ldots, t^{-e_{\beta_1}^{(1)} }\omega_{\beta_1}^2$ are $k[t^{-1}]$-bases, so by our reasoning we must also have 
\[ t^{-e_m-e_n+e_\ell^{(1)} } Q_{mn}^\ell \in k[t^{-1}] \]
for all $\ell$, 
which shows that
\[
\deg Q_{mn}^\ell \leq e_m + e_n - e_\ell^{(1)} ,
\]
or in other words that the quadric $Q^\ell$ defines an element of $2H - e_\ell^{(1)}R$ containing $C$.
By the minimality of~\eqref{Schresolution} we must have 
\[ e_\ell^{(1)} \leq b_\ell^{(1)} \]
for all $\ell$. But from Lemma~\ref{lem:sum.schreyer.invs} we already know that
\[
\sum_{\ell = 1}^{\beta_1} e_\ell^{(1)} = \sum_{\ell = 1}^{\beta_1} b_\ell^{(1)},
\]
from which Theorem~\ref{thm:schreyerisscrollar} follows in the case $i = 1$. 
%In the case of good ramification, one uses Lemma \ref{lem:genus.good.ramification} to conclude instead.

The idea for the following steps is similar. Let $2\leq i\leq d-3$. By our choice of bases we can write
\[
\psi_i(\omega^i_m\otimes \alpha_n) = \sum_{\ell = 1}^{\beta_i} L^\ell_{mn}\omega^{i+1}_m,
\]
for certain $L^\ell_{mn}\in k[t]$. Above infinity we find
\[
\psi_i(t^{-e_m^{(i-1)}}\omega^i_m \otimes t^{-e_n}\alpha_n) = \sum_{\ell = 1}^{\beta_i} t^{-e_m^{(i-1)} - e_n + e_\ell^{(i)} } L^\ell_{mn} (t^{-e_\ell^{(i)}}\omega_\ell^{i+1})
\]
and we conclude as above that
\[ t^{-e_m^{(i-1)} - e_n + e_\ell^{(i)} } L^\ell_{mn} \in k[t^{-1}], \]
or in other words that 
\[ \deg L^\ell_{mn} \leq e_m^{(i-1)} + e_n - e_\ell^{(i)}  = b_m^{(i-1)} + e_n - e_\ell^{(i)} ,
\]
where the equality follows by induction on $i$.
Again, the minimality of~\eqref{Schresolution} implies that 
\[ e_\ell^{(i)} \leq b_\ell^{(i)},\] 
which in view of Lemma~\ref{lem:sum.schreyer.invs} allows us to conclude
that $b_\ell^{(i)} = e_\ell^{(i)}$.
%The reasoning for the final step is identical and we omit the details.
\end{proof}

\subsection{} \label{ssec:schreyerduality} 
A well-known consequence to the self-duality of the relative minimal resolution, see~\cite[Thm.\,1.3]{casnati_ekedahl} or~\cite[Cor.\,4.4]{schreyer}, reads that
\begin{equation} \label{eq:schreyerduality} 
 \left\{ \, b_j^{(d-2-i)} \, \right\}_j = \left\{ \, g + d - 1 - b_j^{(i)} \, \right\}_j 
\end{equation}
as multi-sets. Assuming simple branching, one can view this as a special case of Proposition~\ref{prop:scrollar.invs.duality}. Indeed, because the partitions $\lambda_{i+1}$ and $\lambda_{d-i-1}$ are dual to each other, i.e., they are obtained from one another by transposing Young diagrams, this proposition along with
Theorem~\ref{thm:schreyerisscrollar} immediately implies the duality statement.

% HET VOLGENDE NOG WEGLATEN VOOR LANGE VERSIE; PROOF ENVIRONMENT CLASHT MET \OPTIONAL

\subsection{} \label{ssec:SdSd-2proof}
We end this section with a proof of our exemplary Theorem~\ref{thm:S2Sd-2}:
\begin{proof}[Proof of Theorem~\ref{thm:S2Sd-2}]
This follows from the decomposition
\[
\Ind_{S_2\times S_d}^{S_d} \mathbf{1} \cong V_{(d)}\oplus V_{(d-1,1)}\oplus V_{(d-2,2)}.
\]
along with Proposition~\ref{prop:scrollar.invs.of.hooks}, Theorem~\ref{thm:schreyerisscrollar} and
Theorem~\ref{thm:scrollar.invariants.resolvent}.
\end{proof}

\subsection{} \label{ssec:3111theorem} The case $i = 1$ of Theorem~\ref{thm:schreyerisscrollar} 
remains true under the weaker assumption of good ramification, as defined in~\ref{ssec:31111}.

\begin{proposition}[addendum to Theorem~\ref{thm:schreyerisscrollar}] \label{prop:schreyerrelaxed}
Consider a degree $d \geq 4$ cover $C \to \PP^1$ with good ramification.
The multi-set of scrollar invariants of
$(d-2,2)$ with respect to $\varphi$ is given by
\[ \left\{ \, b_j \, \right\}_{j=1, 2, \ldots, \beta_1}, \]
 the splitting type of the first syzygy bundle in the relative minimal
resolution of $C$ with respect to $\varphi$.
\end{proposition}
\begin{proof}
This is an exact copy of the proof of Theorem~\ref{thm:schreyerisscrollar}, except that now one needs to use Lemma~\ref{lem:genus.good.ramification} as a substitute for Proposition~\ref{prop:genus.irrep}
in establishing Lemma~\ref{lem:sum.schreyer.invs}.
\end{proof}

\section{Examples} \label{sec:examples}

%In this section we give several examples of scrollar invariants in small degree, in particular treating the Lagrange cubic resolvent and the Cayley sextic resolvent. We also give several infinite families of resolvent covers and their scrollar invariants. 

\subsection{Low degree examples.} We have made a repository containing explicit descriptions of the scrollar invariants of any resolvent of any simply branched morphism $\varphi : C \to \PP^1$ of degree $d \leq 6$. It can be found at~\url{https://homes.esat.kuleuven.be/~wcastryc/}. Let us highlight three examples: Lagrange's cubic resolvent in degree $d = 4$, Cayley's sextic resolvent in degree $d = 5$, and an ``exotic resolvent" in degree $d = 6$.

\subsection{} \label{ssec:lagrangecubic}
Up to conjugation, $D_4 = \langle (1\,2), (1 \, 3 \, 2 \, 4) \rangle$ is the unique subgroup of $S_4$ 
which is isomorphic to the dihedral group of order $8$. In Galois theory, the resolvent of a quartic polynomial $f(x)$ with respect to $D_4$ is also known as ``Lagrange's cubic resolvent" and is best known as a tool for solving $f(x) = 0$ in terms of radicals. When applied to a simply branched degree $4$ cover of $\PP^1$ by a curve of genus $g$, the resolvent construction results in a degree $3$ cover of $\PP^1$ by a curve of genus $g + 1$, having the following scrollar invariants:
\begin{corollary}[Casnati] \label{cor:casnati}
Consider a degree $4$ simply branched cover $\varphi: C\to \PP^1$ having Schreyer invariants $b_1, b_2$. The scrollar invariants of  $\res_{D_4} \varphi : \res_{D_4} C \to \PP^1$ are equal to the Schreyer invariants of $C$ with respect to $\varphi$. 
\end{corollary}
\begin{proof}
One checks that
\[
\Ind_{D_4}^{S_4} \mathbf{1} \cong V_{(4)}\oplus V_{(2^2)}
\]
so, in the case where $\varphi$ is simply branched, this is an immediate consequence of Theorem~\ref{thm:schreyerisscrollar} and Theorem~\ref{thm:scrollar.invariants.resolvent}. 
%\optional{If $\varphi$ admits $(3,1)$-ramification, then this follows with the aid of Lemma~\ref{lem:genus.good.ramification}, as in the proof of Theorem~\ref{thm:S2Sd-2relaxed}. }
\end{proof}

Both Casnati~\cite[Def.\,6.3-6.4]{casnati} and Deopurkar--Patel~\cite[Prop.\,4.6]{deopurkar_patel} use the language of Recillas' trigonal construction~\cite{recillas}, rather than the cubic resolvent. This construction applies to arbitrary degree $4$ covers and produces smooth and geometrically integral curves as soon as there is no ramification of type $(2^2)$ or $(4)$, even in the case of a smaller Galois group (the only other option being $A_4$). Therefore Casnati's result is slightly stronger than
Corollary~\ref{cor:casnati}. Note that the case of an $S_4$-cover without ramification of type $(2^2)$ or $(4)$ can in fact be settled by invoking Proposition~\ref{prop:schreyerrelaxed} instead of Theorem~\ref{thm:schreyerisscrollar}.
%\optional{, but this is just an artifact of our technical assumptions; see~\ref{ssec:assumptions}--\ref{ssec:assumptionssometimesok} for a discussion}.

\subsection{} \label{ssec:lagrange_vs_recillas} For the sake of exposition, let us briefly recall the connection between Lagrange's cubic resolvent and Recillas' trigonal construction. This can be found in e.g.~\cite[\S8.6]{vangeemen}, \cite{vakiltwelvepts} or~\cite{recillas}. Recall that our curve $C$ 
arises as the complete intersection of two divisors
\[ 
D_1 \in 2H - b_1R \qquad \text{and} \qquad D_2 \in 2H - b_2R 
\]
inside $\PP(\mathcal{E})$.
Every fiber of the bundle map $ \pi : \PP(\EE)\to \PP^1$ is a $\PP^{2}$, in which our curve cuts out $4$ points, counting multiplicities. Given a configuration of $4$ generally positioned points in $\PP^2$, there are three ways in which these can be viewed as a union of two pairs. 
If each time we take the intersection point of the two lines spanned by these pairs then we find a ``dual" configuration of $3$ points in the same $\PP^2$, see Figure~\ref{dualthreepoints}. 
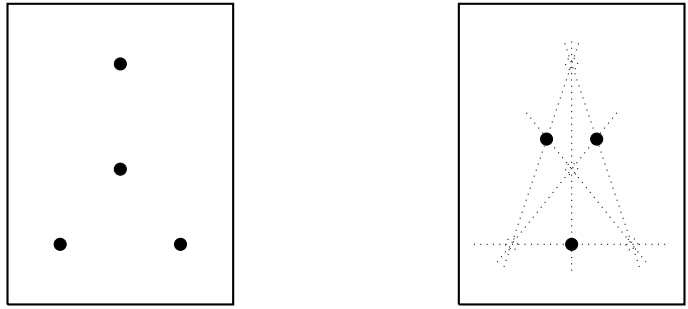
\begin{figure}[ht] 
\begin{center}
 \begin{tikzpicture}
   \draw [thick] (0,0) -- (3,0) -- (3,4) -- (0,4) -- (0,0);
   \draw [fill=black,black] (0.7,0.8) circle (0.08);
   \draw [fill=black,black] (2.3,0.8) circle (0.08);
   \draw [fill=black,black] (1.5,3.2) circle (0.08);
   \draw [fill=black,black] (1.5,1.8) circle (0.08);
   \draw [thick] (6,0) -- (9,0) -- (9,4) -- (6,4) -- (6,0);
   \draw [black,dotted] (6.7,0.8) circle (0.08);
   \draw [black,dotted] (8.3,0.8) circle (0.08);
   \draw [black,dotted] (7.5,3.2) circle (0.08);
   \draw [black,dotted] (7.5,1.8) circle (0.08);
   \draw [dotted] (6.2,0.8) -- (8.8,0.8);
   \draw [dotted] (6.6,0.5) -- (7.6,3.5);
   \draw [dotted] (8.4,0.5) -- (7.4,3.5);
   \draw [dotted] (7.5,3.5) -- (7.5,0.4);
   \draw [dotted] (8.1,2.55) -- (6.5,0.55);
   \draw [dotted] (6.9,2.55) -- (8.5,0.55);
   \draw [fill=black,black] (7.5,0.8) circle (0.08);
   \draw [fill=black,black] (7.166,2.2) circle (0.08);
   \draw [fill=black,black] (7.833,2.2) circle (0.08);
 \end{tikzpicture}
\end{center}
\caption{\small ``Dual" points associated with $4$ points in $\PP^2$ in general position.}
\label{dualthreepoints}
\end{figure}
By applying this procedure, or rather a scheme-theoretic version of it (to cope with multiplicities), to the fibers of $\varphi$ in $\PP(\EE)$ we find a family of point triples that swipe out a new curve $C'$ which naturally comes equipped with a degree $3$ map $\varphi' : C' \to \PP^1$; whence the trigonal construction. If we let $A_1, A_2 \in k[s,t]^{3 \times 3}$ denote symmetric matrices with homogeneous entries corresponding to the ternary quadratic forms defining $D_1$ and $D_2$, then it is easy to see that
\begin{equation} \label{degeneratemembers}
 \det(A_2x_1 + A_1x_2)  = 0
\end{equation}
is a defining binary cubic form for $C'$ in $\PP(\mathcal{E})$. Indeed, this follows by noting that $4$ points in $\PP^2$ in general position define a pencil of quadrics, and that the $3$ dual points are in a natural bijection with the degenerate members of this pencil. 
From expression~\eqref{degeneratemembers} 
one explicitly checks that the scrollar invariants of $\varphi'$ are $b_1, b_2$, e.g., using~\cite[Thm.\,9.1]{linearpencils}. One can also use~\eqref{degeneratemembers} to verify that $C'$ is indeed the geometric counterpart of Lagrange's cubic resolvent, using the formulas from~\cite[p.\,1351]{bhargavaquarticrings}.

\subsection{}
Next, we consider the 
subgroup
\[
\AGL_1(\FF_5) = \langle (1\,2\,3\,4\,5), (1\,2\,4\,3)\rangle \subseteq S_5.
\]
In Galois theory, the resolvent of a quintic polynomial $f(x)$ with respect to $\AGL_1(\FF_5)$ is known as ``Cayley's sextic resolvent"; its main use lies in determining whether the equation $f(x) = 0$ is solvable by radicals~\cite[Cor.\,13.2.11]{coxgalois}. When applied to a simply branched degree $5$ cover of $\PP^1$ by a curve of genus $g$, we obtain a degree $6$ cover of $\PP^1$ by a curve of genus $3g + 7$, with the following scrollar invariants:
\begin{corollary} \label{cor:cayley}
Consider a simply branched degree $5$ cover $\varphi : C \to \PP^1$. 
The degree $6$ cover $\res_{\AGL_1(\FF_5)} \varphi : \res_{\AGL_1(\FF_5)} C \to \PP^1$ has scrollar invariants \[ \{ b_1^{(2)} , b_2^{(2)} , b_3^{(2)} , b_4^{(2)} , b_5^{(2)}  \}, \]
where the $b_j^{(2)}$'s denote the elements of
the splitting type of the second syzygy bundle in the relative minimal resolution of $\varphi$.
\end{corollary}
\begin{proof}
Use $\Ind_{\AGL_1(\FF_5)}^{S_5} \mathbf{1} \cong V_{(5)} \oplus V_{(2^2,1)}$
in combination with Theorems~\ref{thm:schreyerisscrollar} and~\ref{thm:scrollar.invariants.resolvent}.
\end{proof}

\noindent Cayley's sextic resolvent also appears in Bhargava's work on quintic ring parame\-trizations~\cite{bhargavaquinticrings}, and a proof of Corollary~\ref{cor:cayley} can also be deduced from that work. The details of this approach can be found in the master thesis of the second-listed author~\cite[Thm.\,4.13]{vermeulenthesis}, where this was studied in the context of lifting pentagonal curves from finite fields to characteristic zero, for use in Tuitman's point counting algorithm~\cite[\S5]{castryck_lifting_2020}. As in the case of Recillas' trigonal construction, this approach in fact allows for a relaxation of the simple branching assumption. %\optional{; see again~\ref{ssec:assumptions}--\ref{ssec:assumptionssometimesok} for a discussion}.  

\subsection{} \label{ssec:exotic} From degree $6$ onward, 
there exist resolvents whose scrollar invariants we can no longer relate to known data of $\varphi : C \to \PP^1$. For $d=6$, these scrollar invariants arise from the partitions $(2^3)$ and $(3^2)$; note that these partitions are dual to each other. Since $\dim V_{(2^3)}=5$, the scrollar invariants of $(2^3)$ are certain integers $1\leq a_1\leq \ldots \leq a_5$, which sum up to $3g+15$ by Proposition~\ref{prop:genus.irrep}. By Proposition~\ref{prop:scrollar.invs.duality} the scrollar invariants of $(3^2)$ are $g+5-a_1, \ldots, g+5-a_5$, which sum up to $2g+10$.
We call $a_1, a_2, \ldots, a_5$ the ``exotic invariants" of $C$ with respect to $\varphi$. 
The terminology comes from the exotic embedding $S_5\hookrightarrow S_6$ realizing $S_5$ as a transitive subgroup of $S_6$; we denote this subgroup by $S_5'$. It is unique up to conjugation; one 
realization is $\langle (1\,2\,3\,4), (1\,5\,6\,2) \rangle$.
With respect to $S_5'$, the resolvent of a simply branched degree $6$ cover of $\PP^1$ by a curve of genus $g$ is another degree $6$ cover of $\PP^1$, now by a curve of genus $3g+10$.

\begin{corollary}
Consider a simply branched degree $6$ cover $\varphi : C \to \PP^1$. The scrollar invariants
of the degree $6$ cover $\res_{S_5'} C$ with respect to $\res_{S_5'} \varphi$
are given by the exotic invariants of $C$ with respect to $\varphi$.
\end{corollary}
\begin{proof}
This follows from $\Ind_{S_5'}^{S_6} \mathbf{1} \cong  V_{(6)} \oplus V_{(2^3)}$
along with Theorem~\ref{thm:scrollar.invariants.resolvent}.
\end{proof}

We wonder whether a deeper understanding of these exotic invariants could be key towards a better understanding of the Hurwitz spaces $\mathcal{H}_{6,g}$, where the most pressing question is whether they are unirational or not~\cite[\S1]{tanturri_schreyer}. Here are two concrete first problems:

\begin{problem}
Find an alternative interpretation for the exotic invariants $a_i$, directly in terms of the 
morphism $\varphi : C \to \PP^1$.
\end{problem}

\begin{problem} \label{prob:different_exotic}
Find simply branched degree $6$ morphisms $C \to \PP^1$ and $C' \to \PP^1$ having the same scrollar invariants, the same Schreyer invariants, but different exotic invariants.
\end{problem}

\subsection{Some further infinite families.} \label{ssec:infinitefamilies}
In arbitrary degree $d \geq 2$, we have already discussed the resolvent with respect to $A_{d-1}$ in
Corollary~\ref{cor:scrollars.Ad-1}, as well as the resolvent with respect to $S_2 \times S_{d-2}$ in Theorem~\ref{thm:S2Sd-2}. Let us extend this list somewhat further:
\begin{corollary} \label{cor:infinitefamilies}
Consider a simply branched cover $\varphi : C \to \PP^1$ of degree $d \geq 4$. Let $\{ e_1, e_2, \ldots, e_{d-1} \}$ be its scrollar invariants and let $\{ b_1, b_2, \ldots, b_{d(d-1)/3} \}$ be the splitting type of the first syzygy bundle of its relative minimal resolution. Then:
\begin{itemize}
  \item the scrollar invariants of $\res_{S_{d-2}} C$ with respect to $\res_{S_{d-2}} \varphi$ are
  \[ \{e_i\}_i \, \cup \, \{ e_i \}_i \, \cup \, \{ e_i + e_j  \}_{i < j} \, \cup \, \{b_i \}_i, \]
  \item the scrollar invariants of $\res_{S_2 \times A_{d-2}} C$ with respect to $\res_{S_2 \times A_{d-2}} \varphi$ are
  \[ \quad \qquad  \{ e_i \}_i \, \cup \, \{ b_i  \}_i \, \cup \, \{ g + d - 1 - e_i\}_i \, \cup \, \{g + d - 1 - e_i - e_j \}_{i < j}, \]
  \item the scrollar invariants of $\res_{A_{d-2}} C$ with respect to $\res_{A_{d-2}} \varphi$ are
  \begin{multline*} 
  \qquad \ \{e_i\}_i \, \cup \, \{e_i\}_i \, \cup \, \{ e_i + e_j  \}_{i < j} \, \cup \, 
  \{b_i \}_i  \\ \, \cup \, \{ g + d - 1 - e_i\}_i \, \cup \, \{ g + d - 1 - e_i\}_i   \\ \cup \, \{ g + d - 1 - e_i - e_j \}_{i < j} \, \cup \,
  \{ g + d - 1 - b_i\}_i \, \cup \, \{ g + d - 1 \},
  \end{multline*}
\end{itemize}
where the unions are as multi-sets and where $g$ denotes the genus of $C$.
\end{corollary}
\begin{proof}
This follows from the decompositions\footnote{Note that, for $d = 4,5$, some terms may coincide. E.g., for $d = 4$ the terms $V_{(d-1,1)}$ and $V_{(3,1^{d-3})}$ are the same, and then so are the corresponding multi-sets of scrollar invariants $\{e_i\}_i$ and $\{ g + d - 1 - e_i - e_j\}_{i < j}$.}
\begin{align*} 
  \Ind^{S_d}_{S_{d-2}} \mathbf{1} & \cong   V_{(d)} \oplus V_{(d-1,1)}^2 \oplus V_{(d-2, 1^2)} \oplus V_{(d-2,2)},  \\
\Ind^{S_d}_{S_2 \times A_{d-2}} \mathbf{1} & \cong   V_{(d)} \oplus V_{(d-1,1)} \oplus V_{(d-2, 2)} \oplus V_{(3,1^{d-3})} \oplus V_{(2, 1^{d-2})}, \\ 
\Ind^{S_d}_{A_{d-2}} \mathbf{1}  & \cong  V_{(d)} \oplus V_{(d-1,1)}^2 \oplus V_{(d-2, 1^2)} \oplus V_{(d-2,2)}  \\
& \quad \ \oplus V_{(3, 1^{d-3})} \oplus V_{(2, 1^{d-2})}^2 \oplus V_{(2^2, 1^{d-4})} \oplus V_{(1^d)} 
\end{align*}
along with Proposition~\ref{prop:scrollar.invs.of.hooks} and Theorems~\ref{thm:schreyerisscrollar} and~\ref{thm:scrollar.invariants.resolvent}; in the case of $A_{d-2}$ we also used duality.
\end{proof}

\subsection{} For the sake of illustration, we also include an infinite family of resolvents whose scrollar invariants
we cannot express purely in terms of $\{e_i\}_i$ and $\{b_i\}_i$ (unless these would turn out to be related to the scrollar invariants of $(d-3,3)$; this is related to Problem~\ref{prob:different_exotic}).
\begin{corollary}
Consider a simply branched cover $\varphi : C \to \PP^1$ of degree $d \geq 6$. Let $\{ e_1, e_2, \ldots, e_{d-1} \}$ be its scrollar invariants, let 
\[ \{ b_1^{(1)}, b_2^{(1)}, \ldots, b_{d(d-1)/3}^{(1)} \}, \qquad \text{resp.} \qquad \{ b_1^{(2)}, b_2^{(2)}, \ldots, b^{(2)}_{d(d-1)/3} \}, \] 
be the splitting types of the first, resp.\ second, syzygy bundle of its relative minimal resolution, and let $\{ c_1, c_2, \ldots, c_{d(d-1)(d-5)/6} \}$
be the scrollar invariants of the partition $(d-3,3)$ with respect to $\varphi$.
The scrollar invariants of $\res_{S_2 \times S_{d-3}} C$ with respect to 
$\res_{S_2 \times S_{d-3}} \varphi$ are given by
\[ 
\{ e_i \}_i \, \cup \, \{e_i\}_i \, \cup \, \{e_i + e_j \}_{i < j} \, \cup \, \{ b_i^{(1)}  \}_i \, \cup
\, \{ b_i^{(1)}  \}_i \, \cup \, \{ b_i^{(2)}  \}_i \, \cup \, \{c_i\}_i, 
\]
as a union of multi-sets.
\end{corollary}
\begin{proof}
This follows from
\[ \Ind^{S_d}_{S_2 \times S_{d-3}} \mathbf{1} \cong V_{(d)} \oplus V_{(d-1,1)}^2 \oplus V_{(d-2, 1^2)}
\oplus V_{(d-2,2)}^2 \oplus V_{(d-3,2,1)} \oplus V_{(d-3,3)} \]
together with Proposition~\ref{prop:scrollar.invs.of.hooks}, Theorem~\ref{thm:schreyerisscrollar} and Theorem~\ref{thm:scrollar.invariants.resolvent}. The fact that $(d-3,3)$ comes equipped with $d(d-1)(d-5)/6$ scrollar invariants follows from the hook length formula.
\end{proof}
\noindent Note that the invariants $c_i$ sum up to $(d-2)(d-5)(g+d-1)/2 $
in view of Proposition~\ref{prop:genus.irrep}. If $d = 6$ then these invariants are dual to the exotic invariants discussed in~\ref{ssec:exotic}.

\subsection{Curves on Hirzebruch surfaces}\label{ssec:hirzebruch}
Consider a smooth curve $C$ on the Hirzebruch surface $F_e = \PP(\mathcal{O}_{\PP^1} \oplus \mathcal{O}_{\PP^1}(e))$ of invariant $e \geq 0$, along with the morphism $\varphi : C \to \PP^1$ induced by the bundle map $\pi : F_e \to \PP^1$. Assume that this morphism is dominant of degree $d \geq 2$, and simply branched.
We will give a conjectural description of the scrollar invariants of each partition of $d$ with respect to $\varphi$ in terms of the ``bidegree" of $C$, by which we mean the tuple $(c,d) \in \ZZ^2$ 
such that $C \sim dE + (c + de)F$, with $F$ a fiber of $\pi$ and $E$ a section with self-intersection $-e$. Equivalently, one can view $F_e$ as the projective toric surface polarized by the lattice polygon $\Delta_{c,d,e}$ from Figure~\ref{fig:hirzebruchpolygon}, 
\begin{figure}[ht] 
\begin{center}
 \begin{tikzpicture}
   \draw [->] (-1,0)--(6,0);
   \draw [->] (0,-1)--(0,3);
   \draw [thick] (0,0)--(0,2)--(1,2)--(5,0)--(0,0);
   \draw [fill=black] (0,0) circle (0.08);
   \draw [fill=black] (5,0) circle (0.08);
   \draw [fill=black] (0,2) circle (0.08);
   \draw [fill=black] (1,2) circle (0.08);
   \node at (-0.5,-0.3) {\small $(0,0)$};
   \node at (-0.5,2.2) {\small $(0,d)$};
   \node at (1.5,2.2) {\small $(c,d)$};
   \node at (5.6,0.4) {\small $(c + de,0)$};
   \node at (1.2,0.8) {\small $\Delta_{c,d,e}$};
 \end{tikzpicture}
\end{center}
\caption{\small The lattice polygon corresponding to bidegree $(c,d)$ on $F_e$.}
\label{fig:hirzebruchpolygon}
\end{figure}
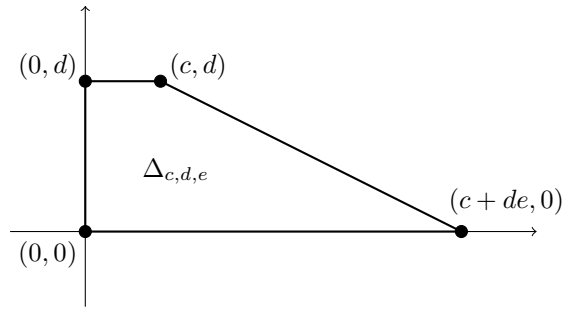
and then $C$ is defined by a sufficiently general bivariate polynomial that is supported on this polygon (with $\varphi$ corresponding to projection on the first coordinate). The most interesting special cases are $e = 0$, in which case we are looking at smooth curves on $\PP^1 \times \PP^1$ of bidegree $(c,d)$ in the traditional sense, and $e = c = 1$, corresponding to smooth plane curves of degree $d + 1$.

The scrollar invariants of $C$ with respect to $\varphi$ are given by
\[ e_i = c + ie, \qquad i = 1, \ldots, d-1. \]
This follows from~\cite[Thm.\,9.1]{linearpencils}, since $e_i$ equals the length of the longest line segment having lattice end points in the interior of $\Delta_{c,d,e}$ at height $d-i$.
Consequently, the genus of $C$ equals $(d-1)(c + de/2 - 1)$. This is the number of lattice points in the interior of 
$\Delta_{c,d,e}$.

\subsection{} 
Let $\lambda$ be a partition of $d$, and consider the following recipe for attaching an integer $e(T)$ to a standard Young tableau $T$ of shape $\lambda$. Start by writing down the ``reading word" of $T$, which is obtained by concatenating its rows, starting from the bottom row. For example, the reading word of the tableau
\vspace{0.1cm}
\begin{center}
    \begin{tikzpicture}[scale=0.5,baseline=(current  bounding  box.center)]
      \draw[thick] (0,-1) rectangle (1,0);
      \draw[thick] (0,0) rectangle (1,1);
      \draw[thick] (0,1) rectangle (1,2);
      \draw[thick] (1,0) rectangle (2,1);
      \draw[thick] (1,1) rectangle (2,2);
      \draw[thick] (2,1) rectangle (3,2);
      \draw[thick] (3,1) rectangle (4,2);
      \node at (0.5,-0.5) {\small $6$};
      \node at (0.5,0.5) {\small $2$};
      \node at (0.5,1.5) {\small $1$};
      \node at (1.5,0.5) {\small $4$};
      \node at (1.5,1.5) {\small $3$};
      \node at (2.5,1.5) {\small $5$};
      \node at (3.5,1.5) {\small $7$};
    \end{tikzpicture} 
\end{center}
\vspace{0.1cm}
of shape $(4,2,1)$
is $6\,2\,4\,1\,3\,5\,7$.
Consider the set $I(T)$ of indices $i \in \{ 1, \ldots, d-1 \}$ for which $i+1$ appears to the left of $i$ in this reading word. Then we let
\[ e(T) = \sum_{i \in I(T)} e_i  \]
which for the above example equals $e_1 + e_3 + e_5 = c + 9e$.
We propose:
\begin{conjecture} \label{conj.hirzebruch}
%With $e, c,d$ as above, let $C$ be a smooth curve on $F_e$ such that $C \sim dE + (c + de)F$. Let $\varphi: C\to \PP^1$ be the map induced by the bundle map $F_e\to \PP^1$. 
Let $e\geq 0$ be an integer and let $C$ be a smooth curve on $F_e$. Let $C \sim dE + (c + de)F$ for integers $d\geq 2$, $c\geq 0$. Assume that the map $\varphi: C\to \PP^1$ induced by the bundle map $F_e\to \PP^1$ is simply branched. Then the multi-set of 
scrollar invariants of any partition $\lambda \vdash d$ with respect to $\varphi$
is given by 
\[ \{ \, e(T) \, | \, T \text{ is a Young tableau of shape } \lambda \, \}. \]
\end{conjecture}
\noindent Note that this multi-set is typically not consecutive, let alone balanced. Of course, this does not
contradict~\cite[Ex.\,1.3.7]{LandesmanLitt}
since smooth curves on Hirzebruch surfaces have a very small locus inside Hurwitz space. 

\subsection{} For the partition $(d-1, 1)$ the conjecture just returns the scrollar invariants 
$\{e_1, e_2, \ldots, e_{d-1} \}$ of $C$ with respect to $\varphi$, as wanted.
By Proposition~\ref{prop:scrollar.invs.of.hooks} the conjecture is also true for hooks.
It also holds for the partition $(d-2,2)$ thanks to Theorem~\ref{thm:schreyerisscrollar} and~\cite[Thm.\,16]{intrinsicness}; by duality, the case $(2^2,1^{d-4})$ is covered as well. 

Thus the first open cases appear in degree $d = 6$.
We have carried out several experiments in \texttt{Magma}~\cite{magma}. Each experiment amounted to computing well-chosen resolvents of some random curve of given bidegree over a large finite field, and recovering the scrollar invariants of these resolvent curves using the command \texttt{ShortBasis()}; our code is available at~\url{https://homes.esat.kuleuven.be/~wcastryc/}.\footnote{Defining equations for these resolvents are found using a fast $t$-adic method that was suggested to us by Frederik Vercauteren and which may be found interesting in its own right.} Each time the output matched with the prediction from Conjecture~\ref{conj.hirzebruch}, when combined with Theorem~\ref{thm:scrollar.invariants.resolvent}. Our choices covered the three remaining partitions
$(2^3)$, $(3,2,1)$, $(3^2)$ of $d = 6$, as well as some new partitions of $d = 7,8$.

Let us emphasize that Conjecture~\ref{conj.hirzebruch} is more than just a guess interpolating between all known cases: these numerics naturally show up when studying scrollar invariants of the $S_d$-closure of monogenic extensions~\cite[\S6]{bhargavasatriano}, and this is how we came up with Conjecture~\ref{conj.hirzebruch} in the first place. In fact, we believe that the direct analogue of our conjecture applies to curves in a much more general class of toric surfaces than Hirzebruch surfaces; however, for arbitrary toric surfaces the combinatorics becomes more subtle and we expect the need for certain correction terms, as is already apparent from~\cite{intrinsicness}.

\section{Applications and concluding remarks} \label{sec:applications}

%We end with several applications of our work, in particular around bounds on Schreyer invariants and their balancedness. We also show how one can construct curves with highly non-balanced scrollar invariants, using resolvent covers. Finally, we give an application towards Gassmann equivalent function fields.

\subsection{Curves with highly non-balanced scrollar invariants.} 

Using our results, many new multi-sets of integers can be shown to be realizable as
the multi-set of scrollar invariants of some $\PP^1$-cover. Here is an example statement:

\begin{proposition} \label{prop:newscollarexample}
Consider integers $d \geq 2$ and $g \geq  d - 1$, and let $k$ be an algebraically closed field with $\charac k = 0$ or $\charac k > d$. Let $e_1\leq \ldots \leq e_{d-1}$ be integers 
summing to $g+d-1$ such that $e_{d-1}-e_1\leq 1$. There exists a genus $g$ curve $C$ over $k$ along with a morphism $\varphi : C \to \PP^1$ whose multi-set of scrollar invariants is given by
\[
\{ e_1, e_2, \ldots, e_{d-1}, g+d-1-e_1, g + d - 1 - e_2, \ldots, g+d-1-e_{d-1}, g+d-1 \}.
\]
\end{proposition}
\begin{proof}
Let $\varphi' : C'\to \PP^1$ be a general element of $\HH_{d,g}$. It is simply branched, and by a result of Ballico~\cite{ballico} its multi-set of scrollar invariants is balanced, i.e., it is 
given by $\{e_1, e_2, \ldots, e_{d-1}\}$. Then take $\varphi$ to be $\res_{ A_{d-1}} \varphi' : \res_{ A_{d-1}} C' \to \PP^1$ and apply Corollary~\ref{cor:scrollars.Ad-1}.
\end{proof}

\noindent Similar results can be obtained using the resolvents with respect to $A_{d-2}$, $S_{d-2}$, $S_2 \times A_{d-2}$ and $S_2 \times S_{d-2}$, whose scrollar invariants were determined in Corollary~\ref{cor:infinitefamilies} and Theorem~\ref{thm:S2Sd-2}, using generic balancedness of the first syzygy bundle~\cite[Main Thm.]{bujokas_patel} in addition to Ballico's result.

\subsection{} The typical resolvent curve is expected to have highly non-balanced scrollar invariants. Indeed, by Theorem~\ref{thm:scrollar.invariants.resolvent} the multi-set of scrollar invariants of the resolvent with respect to some subgroup $H \subseteq S_d$ is naturally subdivided into subsets, one for every partition $\lambda$ appearing in $\Ind_H^{S_d} \mathbf{1}$. From Landesman--Litt~\cite[Ex.\,1.3.7]{LandesmanLitt} 
we know that, generically, the scrollar invariants corresponding to $\lambda$ are all contained in the interval
\[
\left[\frac{\vol_K(\lambda)}{\dim V_\lambda} - \frac{\dim V_\lambda - 1}{2}, \frac{\vol_K(\lambda)}{\dim V_\lambda}+\frac{\dim V_\lambda - 1}{2} \right].
\]
%balanced, in which case the scrollar invariants can be computed using Proposition~\ref{prop:genus.irrep}. 
These values live in regimes that vary strongly with $\lambda$.
Thus, statements like Proposition~\ref{prop:newscollarexample} contrast with previously known ways of constructing curves with prescribed scrollar invariants, such as~\cite{linearpencils,coppens, deopurkar_patel_bundles}, which produce instances that are close to being balanced.

\subsection{New bounds on Schreyer invariants.}
The existing literature reports on several bounds on scrollar invariants. The most important such bound is the Maroni bound, but see e.g.~\cite[Prop.\,2.6]{deopurkar_patel}, \cite[Cond.\,1--3]{patelphd} and~\cite{vemulapalli} for other examples. These results can be combined with Theorem~\ref{thm:schreyerisscrollar} for obtaining bounds on the Schreyer invariants of a simply branched cover. 

For instance, applying the Maroni bound to a resolvent cover gives the following:

\begin{lemma}\label{lem:maroni.bound.resolvent}
Let $\varphi : C \to \PP^1$ be a simply branched cover of degree $d \geq 2$ over a field $k$ with $\charac k = 0$ or $\charac k > d$. Let $H$ be a proper subgroup of $S_d$ and let $\lambda \vdash d$ be such that $V_\lambda$ appears in $\Ind_H^{S_d} \mathbf{1}$. Then the scrollar invariants $e_{\lambda, j}$ of $\lambda$ with respect to $\varphi$ satisfy
\[
e_{\lambda, j} \leq \frac{|\{\text{transpositions }\sigma\notin H\}|}{\binom{d}{2}}(g+d-1),
\]
where $g$ denotes the genus of $C$.
\end{lemma}
\begin{proof}
By Theorem~\ref{thm:scrollar.invariants.resolvent} the scrollar invariants of $\lambda$ appear among those of $\res_H C$ with respect to $\res_H\varphi$. The Maroni bound for $\res_H C$ yields that
\[
e_{\lambda, j} \leq \frac{2g(\res_H C) + 2[S_d:H] - 2}{[S_d:H]}.
\]
The genus formula from Theorem~\ref{thm:genusresolvents} then gives the stated result.
\end{proof}

\noindent Applying this lemma to a well-chosen subgroup of $S_d$, we get the following bounds on 
the scrollar and Schreyer invariants. This gives some general range in which the scrollar invariants of a partition $\lambda$ can live. However, determining the actual range of possibilities seems very difficult, as this is not even known for the usual scrollar invariants $e_1, \ldots, e_{d-1}$.

\begin{theorem} \label{thm:maroni.bound.Young.subgrp}
Let $\varphi : C \to \PP^1$ be a simply branched cover of degree $d \geq 4$ over a field $k$ with $\charac k = 0$ or $\charac k > d$. Let $\lambda = (d_1, \ldots, d_r)$ be a partition of $d$. Then the scrollar invariants $e_{\lambda, j}$ of $\lambda$ with respect to $\varphi$ satisfy 
\[
e_{\lambda, j} \leq \frac{d^2-\sum_i d_i^2}{d(d-1)} (g+d-1).
\]
\end{theorem}
\begin{proof}
By \cite[Cor.\,2.4.7]{sagan}, the representation $V_\lambda$ appears in $\Ind_{S_\lambda}^{S_d}\mathbf{1}$. Hence we may apply Lemma \ref{lem:maroni.bound.resolvent} with $H = S_\lambda$. The result then follows from the fact that the number of transpositions in $S_\lambda$ is equal to
\[
\binom{d_1}{2} + \binom{d_2}{2} + \ldots + \binom{d_r}{2}. \qedhere
\] 
\end{proof}

\begin{corollary}\label{cor:schreyerbound}
For $i \in \{1, 2, \ldots, d-3\}$, the elements $b_j^{(i)}$ of the splitting type of the $i$th syzygy bundle in the relative minimal resolution of $C$ with respect to $\varphi$ are contained in
\[
\left[ \tfrac{i(i+1)+2}{d(d-1)}(g+d-1)  , \tfrac{(i+1)(2d-i-2)-2}{d(d-1)}(g+d-1)  \right].
\]
In particular, all $b_j^{(i)}$ are non-negative.
\end{corollary}
\begin{proof}
For the upper bound, we apply the previous theorem to $\lambda_i$ for $i=2, \ldots, d-2$ in combination with Theorem \ref{thm:schreyerisscrollar}. For the lower bound, we use the duality of the Schreyer invariants discussed in~\ref{ssec:schreyerduality}.
\end{proof}

\noindent We  believe that this result remains valid under weaker conditions than simple branching.

\subsection{} The proof of Corollary~\ref{cor:schreyerbound} used the resolvent with respect to  $S_2\times S_{d-i-1}$. In general however, it is unclear which resolvent gives the best upper bound. E.g., using the maximal resolvent (i.e., the Galois closure) would only give an upper bound  of the form 
\[ b_j^{(i)} \lesssim g, \] 
which is always worse than the bound given here. To obtain a good bound on the scrollar invariants $e_{\lambda,i}$, one wants a subgroup $H \subseteq S_d$ such that $V_\lambda$ appears in $\Ind_H^{S_d} \mathbf{1}$ and such that $H$ contains as many transpositions as possible, in view of Lemma~\ref{lem:maroni.bound.resolvent}. For the partitions $\lambda_i$, the reader can check that the optimal Young subgroups are indeed $S_2\times S_{d-i}$. One can also apply Lemma~\ref{lem:maroni.bound.resolvent} to non-Young subgroups to obtain bounds on the scrollar invariants. However, it seems that Young subgroups always give the strongest possible bounds.

\subsection{} If $C$ is a general curve in the Hurwitz space $\mathcal{H}_{d,g}$ then we can give stronger upper bounds on the scrollar invariants of $C$.

\begin{theorem}
Let $\varphi:C\to \PP^1$ be a general element in the Hurwitz space $\mathcal{H}_{d,g}$ and let $\lambda$ be a partition of $d$. Let $i$ be the number of boxes outside the first row of the Young diagram of $\lambda$. Then the scrollar invariants $e_{\lambda, j}$ of $\lambda$ with respect to $\varphi$ satisfy
\[
e_{\lambda, j} \leq \frac{i}{d-1}g + 2i.
\]
In particular, if $d \geq 4$ then the splitting types $b^{(i)}_j$ satisfy
\[
b_j^{(i)} \in \left[ \frac{i}{d-1}g +2(i+1)-d-1, \frac{i+1}{d-1}g+2(i+1)\right]
\]
for $i = 1, \ldots, d-3$.
\end{theorem}
\begin{proof}
Let $L$ be the Galois closure of the function field $k(C)/k(t)$ and recall that we denote by $W_\lambda$ the isotypic component. An adaptation of the proof of Lemma \ref{lem:containment.Wlambda} shows that $W_\lambda \subset W_{(d-1,1)}^i$. So if $\alpha_1, \ldots, \alpha_{d-1}$ is a reduced basis for $V_1 = W_{(d-1,1)}\cap L^{S_{d-1}}$, then $W_\lambda$ has a basis consisting of elements which are $k[t]$-linear combinations of conjugates of elements of the form
\[
\alpha_{j_1}^{(m_1)}\cdots \alpha_{j_i}^{(m_i)}.
\]
Note that such elements are integral over $k[t]$ and hence the scrollar invariants of $\lambda$ satisfy
\[
e_{\lambda, j}\leq i  \max_\ell e_\ell,
\]
where $e_1, \ldots, e_{d-1}$ are the usual scrollar invariants of $C\to \PP^1$. The condition that $C$ is general implies that the $e_\ell$ are balanced, so $e_\ell\leq \frac{g}{d-1}+2$ for every $\ell$. This gives the desired statement.

The upper bound on the splitting types $b_j^{(i)}$ follows from Theorem \ref{thm:schreyerisscrollar}, while the lower bound follows from duality as in Section \ref{ssec:schreyerduality}.
\end{proof}

\subsection{Gassmann equivalence implies scrollar equivalence.} \label{ssec:gassmannequiv}
 Two subgroups $H_1, H_2$ of a finite group $G$ are called ``Gassmann equivalent"
if for every conjugacy class $\mathcal{C}$ of $G$ it holds that $| H_1 \cap \mathcal{C} | = | H_2 \cap \mathcal{C} |$. For example, the subgroups 
$H_1 = \langle (1\,2)(3\,4), (1\,3)(2\,4) \rangle$
and $H_2 = \langle (1\,2)(3\,4), (1\,2)(5\,6) \rangle$
of $S_6$ are Gassmann equivalent, despite the fact that they are not conjugate. A classical result by Gassmann states that if $L$ is a finite Galois extension of $\QQ$ with Galois group $G$, then $H_1, H_2$ are Gassmann equivalent if and only if $L^{H_1}$ and $L^{H_2}$ are ``arithmetically equivalent", i.e., they have the same Dedekind zeta function~\cite[Thm.\,1.23]{solomatin}. In that case, they necessarily have the same degree and discriminant.

This story partly breaks down in the function field setting, see~\cite[\S3.1.1]{solomatin} for a discussion. 
However, without much effort we can conclude:
\begin{theorem}\label{thm:gassmann.eqv}
Let $\varphi : C \to \PP^1$ be a simply branched degree $d$ cover, and let $H_1, H_2$ be two Gassmann equivalent subgroups of $S_d$. Then the resolvent covers 
$\res_{H_1} \varphi : \res_{H_1} C \to \PP^1$ and
$\res_{H_2} \varphi : \res_{H_2} C \to \PP^1$ 
have the same degree, genus, and scrollar invariants.
\end{theorem}
\begin{proof}
$\Ind^{S_d}_{H_1} \mathbf{1} \cong \Ind^{S_d}_{H_2} \mathbf{1}$  by~\cite[\S1.4.2]{solomatin}, so the claim about the scrollar invariants follows from Theorem~\ref{thm:scrollar.invariants.resolvent}. This implies the claims about degree and genus.
\end{proof}
\noindent Equality of degree and genus has a clear counterpart for number fields (degree resp.\ discriminant), while the scrollar invariants of a curve correspond to the successive minima of the Minkowski lattice of a number field. An analogue of Theorem~\ref{thm:gassmann.eqv} in the number field setting was recently proved by the second-listed author~\cite{floris_minima}.  %but we are unaware of statements in the number-theory literature claiming that the Minkowski lattices attached to arithmetically equivalent number fields have the same successive minima, up to some bounded error factor.

\subsection{On the simple branching assumption} \label{ssec:assumptions}
Our first crucial use of the simple branching assumption was in the proof of Theorem~\ref{thm:genusresolvents}, giving a formula for the genus of a resolvent curve.
 We recall that this genus formula was used to prove our volume formula from Proposition~\ref{prop:genus.irrep}, therefore the assumption is an important ingredient in the proof of Theorem~\ref{thm:schreyerisscrollar}. However, by using a more careful notion of ``resolvent cover", we believe that it should be possible to get rid of the simple branching assumption.
On the geometry side, this begs for a vast generalization of Recillas' trigonal construction.
Here, an interesting first problem is to reinterpret some existing generalizations, e.g., as discussed in~\cite[\S4.4]{donagi} and~\cite{hoffmayer}, in terms of Galois theory. On the algebraic side, we seem to require a theory of resolvents on the level of rings, rather than fields, which is reminiscent of Lagrange's classical theory of resolvent polynomials~\cite[\S12.1]{coxgalois}.
   For degrees $d=4$ and $d=5$ it should be possible to extract such a theory from Bhargava's work on ring parametrizations~\cite{bhargavaquarticrings,bhargavaquinticrings}, see also~\cite{woodphd}. In the general case, the key ingredient seems to be Bhargava and Satriano's notion of ``$S_d$-closure"~\cite{bhargavasatriano}.  In particular, we still expect the relevant representation theory to be that of $S_d$, even in the case of a smaller Galois closure. %A full elaboration is work in progress.

 In some more detail, if $\varphi : C \to \PP^1$ is an $S_d$-cover with arbitrary ramification, then for certain subgroups $H \subseteq S_d$ the genus of $\res_HC$ might be smaller than what is predicted by Theorem~\ref{thm:genusresolvents}.
 The proof of Theorem~\ref{thm:schreyerisscrollar} suggests that, in order to fix this, one should allow for the use of non-maximal orders in $L^H$, i.e., singular models of $\res_HC$. This is also apparent from Recillas' trigonal construction applied to arbitrary $S_4$-covers of degree $4$, which produces singular curves in the presence of ramification of type $(2^2)$ or $(4)$. More generally, the trigonal construction suggests that if $\varphi$ is a $G$-cover for some arbitrary $G \subseteq S_d$, then the ``correct" resolvent curve with respect to $H \subseteq S_d$ may be singular and/or reducible. Algebraically, the corresponding resolvent cover is described in terms of orders in an \'etale algebra that may not be a field. 
 However, the notions of reduced bases and scrollar invariants naturally carry over to this setting, and we expect that all our main results continue to hold for this generalized notion of scrollar invariants.

 \subsection{} \label{ssec:assumptionssometimesok} Even in the case of non-simple branching, it might still happen that this more carefully constructed resolvent curve is smooth and geometrically integral,  in which case statements like in Section~\ref{sec:examples} should continue to hold without modification. For example, both Casnati's result~\cite[Def.\,6.3-6.4]{casnati} and Proposition~\ref{prop:schreyerrelaxed} are illustrations of this phenomenon. 
  However, we expect that most subgroups $H \subseteq S_d$ tolerate a small number of ramification patterns only. Furthermore, we expect that this is again determined by representation theory. In more detail, for every partition $\lambda \vdash d$ there should be a list of ``good ramification" patterns, and then the resolvent associated to $H$ will be smooth and geometrically integral if and only if there is good ramification for every $\lambda$ appearing in $\Ind_H^{S_d} \mathbf{1}$. For instance, simple branching is good for all partitions of $d$, and all ramification patterns should be good for the partition $(d-1,1)$. For $\lambda = (d-2,2)$ we expect that also $(3, 1^{d-3})$ is good; this is true for Recillas' trigonal construction in case $d = 4$, while for arbitrary $d \geq 4$ this follows from Lemma~\ref{lem:genus.good.ramification} under the assumption 
  that $\varphi$ is an $S_d$-cover.
 
%For instance, the trigonal construction is tolerant to $(3,1)$-ramification in that sense.  
% , and this we generalized by allowing for ramification of type $(3, 1^{d-3})$ in Theorem~\ref{thm:S2Sd-2relaxed}.
 
\subsection{} If we indeed manage to get rid of the simple branching assumption, then the resulting generalization of Theorem~\ref{thm:schreyerisscrollar} will provide us with an alternative, syzygy-free way of defining
the Schreyer invariants of any cover $\varphi : C \to \PP^1$. We remark that this definition would have a natural counterpart for number fields, in terms of successive minima of sublattices of the $S_d$-closure
of their ring of integers.

\subsection{Counting function fields} We end by noting that this project started with a rediscovery of Casnati's result 
during an investigation of the secondary term in the counting function for quartic extensions of $\FF_q(t)$ having bounded discriminant, where $\FF_q$ denotes a finite field whose cardinality $q$ satisfies $\gcd(q,6)=1$. This study was motivated by the Ph.D.\ work of the third-listed author~\cite{zhaophd} who
determined the secondary term in the cubic case, thereby settling the $\FF_q(t)$-counterpart of Roberts' conjecture~\cite{roberts},\footnote{Now a theorem thanks to independent work of Bhargava--Shankar--Tsimerman~\cite{BST} and Taniguchi--Thorne~\cite{taniguchithorne}.} which reads that
\[ 
N_3(X)   = \underbrace{\frac{1}{3\zeta(3)}}_{{}= \, 0.277\ldots} X + \underbrace{\frac{4(1 + \sqrt{3})\zeta(1/3)}{5 \Gamma(2/3)^3\zeta(5/3)}}_{{}=\, -0.403\ldots} X^{5/6} + o(X^{5/6}), \]
where $X$ is a real parameter tending to infinity and
 $N_3(X)$ denotes the number of non-isomorphic cubic extensions $K \supseteq \QQ$ for which $|\Delta_K| \leq X$. Proving the $\FF_q(t)$-analogue of this statement
 essentially boils down to estimating the number of non-isomorphic degree $3$ covers of $\PP^1$ over $\FF_q$ by curves of a given genus $g$, which can be done in bulks by first enumerating all possibilities for the scrollar invariants $e_1, e_2$. As it turns out, the appearance of a negative term in $X^{5/6}$ is naturally related to the offset in this enumeration coming from the Maroni bound $e_2 \leq (2g+4)/3$.

A heuristic reasoning \`a la Roberts makes it reasonable\footnote{This is based 
on Yukie's analysis~\cite{yukie} of the quartic Shintani zeta function, as was explained to us by
Takashi Taniguchi in personal communication.} to expect a similar negative term of order $X^{5/6}$ in the counting function $N_4(X)$ for quartic number fields $K$ with $|\Delta_K| \leq X$. However, it seems a hard open problem to make this prediction precise, and valuable support in its favour would be lended by
a proof of its $\FF_q(t)$-analogue, which essentially amounts to estimating the number of non-isomorphic degree $4$ covers of $\PP^1$ by curves with a given genus $g$. If this can be done by mimicking the ideas of~\cite{zhaophd}, it will involve several technical sieving steps, dealing with covers that are not necessarily simply branched, so this lies beyond the scope of this article. However, it is possible to make a rough back-of-the-envelope analysis suggesting that the exponent $5/6$ is again naturally related to a bound of Maroni type, but now on Schreyer's invariants $b_1,b_2$ rather than on the scrollar invariants $e_1,e_2,e_3$. And, of course, the best reason for the existence of such a bound is that Schreyer's invariants \emph{are} scrollar invariants, by Theorem~\ref{thm:schreyerisscrollar}.

\bibliographystyle{amsplain}
\bibliography{MyLibrary}

\providecommand{\bysame}{\leavevmode\hbox to3em{\hrulefill}\thinspace}
\providecommand{\MR}{\relax\ifhmode\unskip\space\fi MR }
% \MRhref is called by the amsart/book/proc definition of \MR.
\providecommand{\MRhref}[2]{%
  \href{http://www.ams.org/mathscinet-getitem?mr=#1}{#2}
}
\providecommand{\href}[2]{#2}
\begin{thebibliography}{10}

\bibitem{ballico}
E.~Ballico, \emph{A remark on linear series on general {$k$}-gonal curves},
  Bolletino dell'Unione Matematica Italiana \textbf{7} (1989), 195--197.

\bibitem{behnke}
K.~Behnke, \emph{On projective resolutions of {F}robenius algebras and
  {G}orenstein rings}, Mathematische Annalen \textbf{257} (1981), 219--238.

\bibitem{bhargavaquarticrings}
M.~Bhargava, \emph{Higher composition laws {III}: {The} parametrization of
  quartic rings}, Annals of Mathematics \textbf{159} (2004), no.~3, 1329--1360.

\bibitem{bhargavaquinticrings}
\bysame, \emph{Higher composition laws {IV}: {The} parametrization of quintic
  rings}, Annals of Mathematics \textbf{167} (2008), no.~1, 53--94.

\bibitem{bhargavasatriano}
M.~Bhargava and M.~Satriano, \emph{On a notion of {``{Galois} closure"} for
  extensions of rings}, Journal of the {E}uropean {M}athematical {S}ociety
  \textbf{16} (2014), no.~9, 1881--1913.

\bibitem{2torsionclassgroup}
M.~Bhargava, A.~Shankar, T.~Taniguchi, F.~Thorne, J.~Tsimerman, and Y.~Zhao,
  \emph{Bounds on 2-torsion in class groups of number fields and integral
  points on elliptic curves}, Journal of the American Mathematical Society
  \textbf{33} (2020), no.~4, 1087--1099.

\bibitem{BST}
M.~Bhargava, A.~Shankar, and J.~Tsimerman, \emph{On the {Davenport}-{Heilbronn}
  theorems and second order terms}, arXiv:1005.0672 [math] (2012), arXiv:
  1005.0672.

\bibitem{bopphoff}
C.~Bopp and M.~Hoff, \emph{Resolutions of {general} {canonical} {curves} on
  {rational} {normal} {scrolls}}, Archiv der Mathematik \textbf{105} (2015),
  no.~3, 239--249.

\bibitem{magma}
W.~Bosma, J.~Cannon, and C.~Playoust, \emph{The {M}agma algebra system. {I}.
  {T}he user lanuage}, Journal of {S}ymbolic {C}omputation \textbf{24} (1997),
  no.~3--4, 235--265.

\bibitem{bujokas_patel}
G.~Bujokas and A.~Patel, \emph{Invariants of a general branched cover of
  {$\mathbf{P}^1$}}, International Mathematics Research Notices \textbf{2021}
  (2021), no.~6, 4564--4604.

\bibitem{casnati}
G.~Casnati, \emph{Covers of {algebraic} {varieties} {III}. {The} {discriminant}
  of a {cover} of {degree} 4 and the {trigonal} {construction}}, Transactions
  of the American Mathematical Society \textbf{350} (1998), 1359--1378.

\bibitem{casnati_ekedahl}
G.~Casnati and T.~Ekedahl, \emph{Covers of algebraic varieties {I}. {A} general
  structure theorem, covers of degree {$3,4$} and {E}nriques surfaces}, Journal
  of {A}lgebraic {G}eometry \textbf{5} (1995), 439--460.

\bibitem{intrinsicness}
W.~Castryck and F.~Cools, \emph{Intrinsicness of the {N}ewton polygon for
  smooth curves on $\mathbf{P}^1 \times \mathbf{P}^1$}, Revista Matem\'atica
  Complutense \textbf{30} (2017), no.~2, 233--258.

\bibitem{linearpencils}
\bysame, \emph{Linear pencils encoded in the {N}ewton polygon}, International
  Mathematics Research Notices \textbf{2017} (2017), 2998--3049.

\bibitem{castryck_lifting_2020}
W.~Castryck and F.~Vermeulen, \emph{Lifting low-gonal curves for use in
  {Tuitman}'s algorithm}, Proceedings of ANTS-XIV. MSP Open Book Series
  \textbf{4} (2020), 109--125.

\bibitem{self}
W.~Castryck, F.~Vermeulen, and Zhao Y., \emph{Scrollar invariants, syzygies and
  representations of the symmetric group}, Journal f\"ur die reine und
  angewandte Mathematik, to appear.

\bibitem{coppens}
M.~Coppens, \emph{Existence of pencils with prescribed scrollar invariants of
  some general type}, Osaka Journal of Mathematics \textbf{36} (1999), no.~4,
  1049--1057.

\bibitem{coppenskeemmartens}
M.~Coppens, C.~Keem, and G.~Martens, \emph{Primitive linear series on curves},
  Manuscripta {M}athematica \textbf{77} (1992), 237--264.

\bibitem{coppensmartens}
M.~Coppens and G.~Martens, \emph{Linear series on a general $k$-gonal curve},
  Abhandlungen aus dem Mathematischen Seminar der Universit\"at Hamburg
  \textbf{69} (1999), 347--371.

\bibitem{coxgalois}
D.~Cox, \emph{Galois {Theory}}, John {W}iley \& {S}ons, 2012.

\bibitem{deopurkar_patel}
A.~Deopurkar and A.~Patel, \emph{The {P}icard rank conjecture for the {H}urwitz
  spaces of degree up to five}, Algebra \& Number Theory \textbf{9} (2015),
  no.~2, 459--492.

\bibitem{deopurkar_patel_bundles}
\bysame, \emph{{Vector bundles and finite covers}}, arXiv preprint (2019),
  \url{https://arxiv.org/abs/1608.01711}.

\bibitem{donagi}
R.~Donagi, \emph{The {S}chottky problem}, Proceedings of Theory of Moduli,
  Lecture Notes in Mathematics \textbf{1337} (1985), 84--137.

\bibitem{eisenbudelkies}
D.~Eisenbud, N.~Elkies, J.~Harris, and R.~Speiser, \emph{On the {H}urwitz
  scheme and its monodromy}, Compositio {M}athematica \textbf{77} (1991),
  no.~1, 95--117.

\bibitem{eisenbudharris}
D.~Eisenbud and J.~Harris, \emph{On varieties of minimal degree (a centennial
  account)}, Proceedings of {S}ymposia in {P}ure {M}athematics \textbf{46}
  (1987), 3--13.

\bibitem{fulton}
W.~Fulton, \emph{Hurwitz schemes and irreducibility of moduli of algebraic
  curves}, Annals of {M}athematics \textbf{90} (1969), no.~3, 542--575.

\bibitem{fultonharris}
W.~Fulton and J.~Harris, \emph{Representation theory: a first course}, Graduate
  {T}exts in {M}athematics, {S}pringer, 1991.

\bibitem{gangrosssavin}
W.~T. Gan, B.~Gross, and G.~Savin, \emph{Fourier coefficients of modular forms
  on {$G_2$}}, Duke {M}athematical {J}ournal \textbf{115} (2002), no.~1,
  105--169.

\bibitem{hamermesh}
M.~Hamermesh, \emph{Group theory and its application to physical problems},
  Dover {B}ooks on {P}hysics and {C}hemistry, {D}over {P}ublications, 1962.

\bibitem{hessRR}
F.~Hess, \emph{Computing {Riemann}-{Roch} spaces in algebraic function fields
  and related topics}, Journal of {S}ymbolic {C}omputation \textbf{33} (2002),
  no.~4, 425--445.

\bibitem{hessnotes}
\bysame, \emph{Algorithmics of function fields}, Commented slides for the
  {UNCG} Summer School in Computational Number Theory: Function Fields, 2016,
  \url{https://mathstats.uncg.edu/sites/number-theory/summerschool/2016/uncg-hess-comments.pdf},
  accessed 21/09/22.

\bibitem{thiery}
F.~Hivert and N.~M. Thi{\'e}ry, \emph{Symmetric functions and the rational
  {S}teenrod algebra}, {CRM} {P}roceedings and {L}ecture {N}otes \textbf{35}
  (2002), 91--125.

\bibitem{hoffmayer}
M.~Hoff and U.~Mayer, \emph{The osculating cone to special {B}rill--{N}oether
  loci}, Collectanea Mathematica \textbf{66} (2015), no.~3, 387--403.

\bibitem{landesmanlitt2}
A.~Landesman and D.~Litt, \emph{Applications of the algebraic geometry of the
  {P}utman--{W}ieland conjecture}, arXiv preprint (2022),
  \url{https://arxiv.org/pdf/2209.00718}, v2.

\bibitem{LandesmanLitt}
\bysame, \emph{Geometric local systems on very general curves and
  isomonodromy}, arXiv preprint (2022), \url{https://arxiv.org/abs/2202.00039},
  v2.

\bibitem{landesman_vakil_wood}
A.~Landesman, R.~Vakil, and M.~M. Wood, \emph{Low degree {H}urwitz stacks in
  the {G}rothendieck ring}, arXiv preprint (2022),
  \url{https://arxiv.org/pdf/2203.01840}, v1.

\bibitem{hyperelliptic}
H.~W. Lenstra, J.~Pila, and C.~Pomerance, \emph{A {hyperelliptic} {smoothness}
  {test}, {II}}, Proceedings of the London Mathematical Society \textbf{84}
  (2002), no.~1, 105--146.

\bibitem{MS}
V.B. Mehta and C.S. Seshadri, \emph{Moduli of vector bundles on curves with
  parabolic structures}, Mathematische {A}nnalen \textbf{248} (1980), no.~3,
  205--239.

\bibitem{neukirch}
J.~Neukirch, \emph{Algebraic number theory}, Springer-{V}erlag {B}erlin, 1999.

\bibitem{patelphd}
A.~Patel, \emph{The geometry of {H}urwitz space}, Ph.D. thesis, 2013.

\bibitem{peikertrosen}
C.~Peikert and A.~Rosen, \emph{Lattices that admit logarithmic worst-case to
  average-case connection factors}, full version, available at
  \url{https://eprint.iacr.org/2006/444}, of eponymous paper published at
  {P}roceedings of the {$39$}th {ACM} {S}ymposium on {T}heory of {C}omputing
  (2007), 478--487.

\bibitem{recillas}
S.~Recillas, \emph{Jacobians of curves with {$g^1_4$}'s are the {P}ryms of
  trigonal curves}, Bolet{\'i}n de la {S}ociedad {M}atem{\'a}tica {M}exicana
  \textbf{19} (1974), no.~1, 9--13.

\bibitem{roberts}
D.~Roberts, \emph{Density of cubic field discriminants}, Mathematics of
  {C}omputation \textbf{70} (2001), no.~236, 1699--1705.

\bibitem{sagan}
B.~E. Sagan, \emph{The symmetric group. {R}epresentations, combinatorial
  algorithms, and symmetric functions}, Springer, 2000.

\bibitem{schreyer}
F.-O. Schreyer, \emph{Syzygies of canonical curves and special linear series},
  Mathematische {A}nnalen \textbf{275} (1986), 105--137.

\bibitem{tanturri_schreyer}
F.-O. Schreyer and F.~Tanturri, \emph{Matrix factorizations and curves in
  $\mathbf{P}^4$}, Documenta Mathematica \textbf{23} (2018), 1895--1924.

\bibitem{serreRepr}
J.-P. Serre, \emph{Linear representations of finite groups}, Springer New York,
  NY, 1977.

\bibitem{solomatin}
P.~Solomatin, \emph{Global fields and their {$L$}-functions}, Ph.{D}. thesis,
  Leiden University, 2021.

\bibitem{taniguchithorne}
T.~Taniguchi and F.~Thorne, \emph{Secondary terms in counting functions for
  cubic fields}, Duke {M}athematical {J}ournal \textbf{162} (2013), no.~13,
  2451--2508.

\bibitem{vakiltwelvepts}
Ravi Vakil, \emph{Twelve points on the projective line, branched covers, and
  rational elliptic fibrations}, Mathematische Annalen \textbf{320} (1999),
  33--54.

\bibitem{vanderwaerden}
B.~L. van~der Waerden, \emph{Die {Z}erlegungs- und {T}r{\"a}gheitsgruppe als
  {P}ermutationsgruppen}, Mathematische {A}nnalen \textbf{111} (1935),
  731--733.

\bibitem{vangeemen}
B.~van Geemen, \emph{Some remarks on {B}rauer groups of {K}3 surfaces},
  Advances in {M}athematics \textbf{197} (2005), no.~1, 222--247.

\bibitem{vemulapalli}
S.~Vemulapalli, \emph{Bounds on successive minima of orders in number fields
  and scrollar invariants of curves}, arXiv preprint (2022),
  \url{https://arxiv.org/abs/2207.10522}.

\bibitem{vermeulenthesis}
F.~Vermeulen, \emph{Lifting curves of low gonality}, Master's thesis, KU
  Leuven, 2019.

\bibitem{floris_minima}
\bysame, \emph{Arithmetically equivalent number fields have approximately the
  same successive minima}, arXiv preprint (2022),
  \url{https://arxiv.org/abs/2206.13855}, v2.

\bibitem{wilsonphd}
K.~H. Wilson, \emph{Three perspectives on $n$ points in $\mathbf{P}^{n-2}$},
  Ph.{D}. thesis, Princeton University, 2013.

\bibitem{woodphd}
M.~M. Wood, \emph{Moduli spaces for rings and ideals}, Ph.{D}. thesis,
  Princeton University, 2009.

\bibitem{yukie}
A.~Yukie, \emph{Shintani zeta functions}, Cambridge {U}niversity {P}ress, 1993.

\bibitem{zhaophd}
Y.~Zhao, \emph{On sieve methods for varieties over finite fields}, Ph.{D}.
  thesis, University of Wisconsin--Madison, 2013.

\end{thebibliography}

\vspace{0.4cm} 
 
\scriptsize

\noindent \textsc{imec-COSIC, KU Leuven, Belgium}\\
\noindent \textsc{Department of Mathematics: Algebra and Geometry, Ghent University, Belgium}\\
\noindent \texttt{wouter.castryck@esat.kuleuven.be}\\

\vspace{-0.2cm}
\noindent \textsc{Department of Mathematics, KU Leuven, Belgium}\\
\noindent \texttt{floris.vermeulen@kuleuven.be}\\

\vspace{-0.2cm}
\noindent \textsc{School of Science, Westlake University, People’s Republic of
China}\\
\noindent \texttt{zhaoyongqiang@westlake.edu.cn}

\end{document}